\documentclass[10pt,reqno]{amsart}

\usepackage{enumerate}
\usepackage{amsmath, amssymb, amsthm}
\usepackage{mathrsfs}
\usepackage{esint}
\usepackage{xcolor}
\usepackage{mathtools}
\usepackage{hyperref}
\usepackage{bm}
\usepackage{kotex}
\usepackage{accents}
\usepackage[foot]{amsaddr}

\makeatletter

\numberwithin{equation}{section}

\newtheorem{thm}{Theorem}[section]

\newtheorem{corollary}[thm]{Corollary}
\newtheorem{lem}[thm]{Lemma}
\newtheorem{prop}[thm]{Proposition}
\theoremstyle{definition}
\newtheorem{defn}[thm]{Definition}

\theoremstyle{remark}
\newtheorem{rem}[thm]{\bf{Remark}}

\newtheorem{exam}[thm]{Example}{\bfseries}{\itshape}

\newcommand\bE{\mathbb{E}}

\newcommand\bH{\mathbb{H}}

\newcommand\bL{\mathbb{L}}

\newcommand\bN{\mathbb{N}}

\newcommand\bP{\mathbb{P}}

\newcommand\bR{\mathbb{R}}

\newcommand\bZ{\mathbb{Z}}

\newcommand\fO{\mathbf{O}}

\newcommand\cD{\mathcal{D}}
\newcommand\cE{\mathcal{E}}
\newcommand\cF{\mathcal{F}}

\newcommand\cK{\mathcal{K}}

\newcommand\cO{\mathcal{O}}

\newcommand\cQ{\mathcal{Q}}

\newcommand\cS{\mathcal{S}}
\newcommand\cT{\mathcal{T}}

\newcommand\frH{\mathfrak{H}}

\begin{document}

\title[Nonlocal Dirichlet problems with exterior condition]{Weighted Sobolev space theory for non-local elliptic and parabolic equations with non-zero exterior condition on $C^{1,1}$ open sets}

\author{Kyeong-Hun Kim$^1$}
\address{$^{1}$ Department of Mathematics, Korea University, 145 Anam-ro, Seongbuk-gu, Seoul,
02841, Republic of Korea}
\email{kyeonghun@korea.ac.kr}

\author{Junhee Ryu$^{*2}$}
\address{$^2$ School of Mathematics, Korea Institute for Advanced Study, 85 Hoegi-ro, Dongdaemun-gu, Seoul, 02455, Republic of Korea}
\email{junhryu@kias.re.kr}
\thanks{$^*$Corresponding author: Junhee Ryu}

\thanks{K.-H. Kim was supported by the National Research Foundation of Korea(NRF) grant funded by the Korea government(MSIT) (No. RS-2025-00556160). J. Ryu was supported by a KIAS Individual Grant
(MG101501) at Korea Institute for Advanced Study.}



\begin{abstract}
We introduce a weighted Sobolev space theory for the non-local elliptic equation
$$
\Delta^{\alpha/2}u=f, \quad  x\in \mathcal{O}\,; \quad r_{\overline{\mathcal{O}}^c}u=g
$$
as well as for the non-local parabolic equation
$$
u_t=\Delta^{\alpha/2}u+f, \quad t>0,\, x\in \mathcal{O} \,; \quad r_{\mathcal{O}}u(0,\cdot)=u_0, \,r_{(0,T)\times \overline{\mathcal{O}}^c}u=g.
$$
Here, $\alpha\in (0,2)$ and $\mathcal{O}$ is a $C^{1,1}$ open set.
We prove uniqueness and existence results in weighted Sobolev spaces. We measure the Sobolev and H\"older regularities of arbitrary order derivatives of solutions using a system of weights consisting of appropriate powers of the distance to the boundary.  

 One of the most interesting features of our results is that, unlike the classical results in Sobolev spaces without weights, the weighted regularities of solutions in $\mathcal{O}$ are less affected by those of exterior conditions on $\overline{\mathcal{O}}^c$. For instance, even if $g=\delta_{x_0}$, the dirac delta distribution concentrated at $x_0\in \overline{\mathcal{O}}^c $, the solution to the elliptic equation given with $f=0$ is infinitely differentiable in $\mathcal{O}$, and  for any $k=0,1,2, 3,\cdots$, $\varepsilon>0$, and $\delta\in (0,1)$, it holds that
$$
|d_x^{-\frac{\alpha}{2}+\varepsilon+k}D^k_xu|_{C_b(\mathcal{O})}
+|d_x^{-\frac{\alpha}{2}+\varepsilon+k+\delta} D^k_xu|_{C^{\delta}(\mathcal{O})}<\infty,
$$
where $d_x=dist(x, \partial \mathcal{O})$.
\end{abstract}

\maketitle

\section{Introduction}

We present a weighted Sobolev space theory for  the elliptic equation
\begin{equation} \label{elliptic eqn}
\begin{cases}
\Delta^{\alpha/2}u(x)=f(x),\quad &x\in \cO,
\\
u(x)=g(x),\quad &x\in \overline{\cO}^c,
\end{cases}
\end{equation}
and for the parabolic equation
\begin{equation} \label{parabolic eqn}
\begin{cases}
\partial_t u(t,x)=\Delta^{\alpha/2}u(t,x)+f(t,x),\quad &(t,x)\in(0,T)\times \cO,
\\
u(0,x)=u_0(x),\quad & x\in \cO,
\\
u(t,x)=g(t,x),\quad &(t,x)\in (0,T)\times \overline{\cO}^c,
\end{cases}
\end{equation}
  where $\alpha\in (0,2)$, $\cO$ is a $C^{1,1}$ open set, $\overline{\cO}^c$ is the complement of the closure of $\cO$, and the fractional Laplacian $\Delta^{\alpha/2}u$ is defined as 
\begin{equation}
  \label{def frac}
  \Delta^{\alpha/2} u(x) :=c_d\lim_{\varepsilon \downarrow 0}\int_{|y|>\varepsilon} \frac{u(x+y)-u(x)}{|y|^{d+\alpha}}dy
  \end{equation}
with $c_d:=\frac{2^{\alpha}\Gamma(\frac{d+\alpha}{2})}{\pi^{d/2}|\Gamma(-\alpha/2)|}$. One can easily check that \eqref{def frac} is well defined if $u\in C^2_b(\bR^d)$. Generally, if $u$ is a distribution on $\bR^d$, then on $\cO$ we interpret  $\Delta^{\alpha/2}u$  as the distribution (cf. \cite[Section 3]{BB}) defined by
$$
(\Delta^{\alpha/2}u, \phi)_{\cO}=(u, \Delta^{\alpha/2}\phi)_{\bR^d}, \qquad \phi\in C^{\infty}_c(\cO),
$$
given that the right hand side makes sense and defines a distribution on $\cO$. For instance, consider $u:=\delta_{x_0}$, where $x_0\in \overline{\cO}^c$, then for any $\phi \in C_c^\infty(\cO)$,
$$
(\Delta^{\alpha/2}u,\phi)_{\cO}= \Delta^{\alpha/2}\phi(x_0)= c_d \int_{\cO} \frac{\phi(x)}{|x-x_0|^{d+\alpha}}dx.
$$
Thus in the distribution sense, $u=\delta_{x_0}$ satisfies
\begin{equation*}
\begin{cases}
\Delta^{\alpha/2}u(x) = c_d|x-x_0|^{-d-\alpha},\quad &x\in \cO,
\\
u=\delta_{x_0},\quad &x\in \overline{\cO}^c.
\end{cases}
\end{equation*}

In this article, we study equations \eqref{elliptic eqn} and \eqref{parabolic eqn} in the weighted Sobolev spaces $H^{\gamma}_{p,\theta,\sigma}(\fO)$ and $L_p((0,T); H^{\gamma}_{p,\theta,\sigma}(\fO))$, respectively.  Here $\fO$ is either $\cO$ or $\overline{\cO}^c$, $p>1$ and  $\theta,\sigma, \gamma \in \bR$. For instance, if $\gamma=0,1,2,\cdots$, then
$$
\|u\|_{H^{\gamma}_{p,\theta,\sigma}(\cO)}=\left(\sum_{k=0}^{\gamma} \int_{\cO} |d^{k}_x D_x^{k}u|^p d_x^{\theta-d} (1+d_x)^{\sigma}dx \right)^{1/p},
$$
where $d_x$ denotes the distance from $x$ to $\partial \cO$. Here, the powers of $d_x$ and $(1+d_x)$ are used to control  the behaviors of functions near the boundary and at infinity, respectively. These types of spaces were introduced in Section 2.6.3 of \cite{LM} and Section 23 of \cite{E} for $p=2$, $\theta=d$ and $\sigma=0$. In \cite{Krylovhalf,Lototsky}, they were generalized for $p\in(1,\infty)$, $\theta,\gamma\in\bR$ and $\sigma=0$. 
We also refer the reader to \cite{CD} for anisotropic weighted Bessel potential spaces with results for general pseudodifferential operators.
In Section \ref{Secfunc}, we use a unified way introduced in \cite{Krylovhalf,Lototsky} to extend the spaces for all $\sigma\in \bR$.

The equations were studied in a similar setting in  \cite{Dirichlet} provided that  the exterior data is zero. 
The main contribution of this article is to prove that   the weighted regularities of solutions are barely affected by those of the exterior conditions.
 To be more precise,  for elliptic equation \eqref{elliptic eqn}, we prove
\begin{equation}
 \label{eqn 4.11.1}
\|u\|_{H_{p,\theta-\alpha p/2, 0}^{\gamma}(\cO)} \leq N\left(\|f\|_{H^{\gamma-\alpha}_{p,\theta+\alpha p/2,0}(\cO)}+ \|g\|_{H_{p,\theta-\alpha p/2,\sigma}^{\lambda}(\overline{\cO}^c)}\right),
\end{equation}
provided that $\theta \in (d-1, d-1+p)$ and $\sigma>-\theta-\alpha p/2$. Here $\gamma, \lambda \in \bR$ have no relation at all, and $N$ is a constant depending  also on $\gamma$ and $\lambda$. As an example, we can show if $x_0\in \overline{\cO}^c$, then 
$$
g:=\delta_{x_0} \in H^{\lambda}_{p,\theta-\alpha p/2,\sigma}(\overline{\cO}^c), \qquad \lambda < -(p-1)d/p.
$$
Consequently, if $f=0$, then the left side of \eqref{eqn 4.11.1} is finite for any $\gamma\in \bR$ and $p>1$.

Now we summarize the main results of this article.

\begin{enumerate}

\item[$(i)$] Together with the uniqueness and existence results, we prove  \eqref{eqn 4.11.1} for any given $\gamma, \lambda \in \bR$, $p>1$, $\theta\in (d-1,d-1+p)$ and $\sigma>-\theta-\alpha p/2$. The range of $\theta$ is sharp.  We are dealing with arbitrary positive or negative order regularity theory, and the parameter $\lambda$ corresponding to the regularity of exterior data $g$  has no impact on the regularity of the solution. Our results are new even if $g=0$ because \cite{Dirichlet} does not deal with negative regularity theory.  

\vspace{2mm}

\item[$(ii)$] We prove a parabolic version of  $(i)$ for parabolic equation  \eqref{parabolic eqn}.

\vspace{2mm}

\item[$(iii)$] We obtain weighted H\"older estimates of solutions in $\cO$ for both elliptic and parabolic equations.  These are based on   weighted Sobolev-H\"older embedding theorems described in Propositions  \ref{Holder para} and \ref{Holder ellip}. 

\vspace{2mm}

\item[$(iv)$]  We obtain a global regularity on $\bR^d$ in Sobolev spaces without weight.  For this,  we prove some embedding results from weighted Sobolev spaces into Sobolev spaces on $\bR^d$ without weight, and prove 
\begin{equation}
\label{eqn 4.12.1}
\|u\|_{H^{\beta}_p(\bR^d)}\leq N \left(\|f\|_{H^{\beta-\alpha}_{p}(\cO)}+\|g\|_{H^{\beta}_p(\overline{\cO}^c)} \right)
\end{equation}
for elliptic equation \eqref{elliptic eqn}.  Here, $ \frac{\alpha}{2}+ \frac{1}{p}-1<\beta < \frac{\alpha}{2}+ \frac{1}{p}$.
 We also prove a parabolic version of \eqref{eqn 4.12.1} for the parabolic equation \eqref{parabolic eqn}. 
Our approach is different from \cite{Grubb nonlocal} where \eqref{eqn 4.12.1} was already introduced on smooth domains.
 \end{enumerate}

Now we give a short review on related works. For the results on the equations with zero-exterior condition, we only  refer to the  introductions of \cite{AG23,Dirichlet,Grubb nonlocal,Grubb para,Grubb resol} and references therein. Regarding works dealing with non-zero exterior data, we first  refer to \cite{Grubb nonlocal} which includes results in classical Sobolev spaces together with some optimal  H\"older estimates. These results certainly cover \eqref{eqn 4.12.1} on $C^{\infty}$ domains,  which can be extended to domains with $C^{1,\tau}$ regularity for $\tau>\alpha$ when combined with the results of \cite{AG23}. The idea of \cite{Grubb nonlocal} is just to reduce elliptic problem \eqref{elliptic eqn} into the problem having $g=0$. More precisely, let $G\in H_p^{\beta}(\bR^d)$ be an extension of $g\in H^{\beta}_p(\overline{\cO}^c)$, then $v:=u-G$ satisfies
\begin{equation*}
\begin{cases}
\Delta^{\alpha/2}v(x) = f-\Delta^{\alpha/2}G,\quad &x\in \cO,
\\
v=0,\quad &x\in \overline{\cO}^c.
\end{cases}
\end{equation*}
To apply \eqref{eqn 4.12.1} for $v$ with $g=0$, $\Delta^{\alpha/2}G$ should be as smooth as $f$. This approach is no use for us because our exterior data $g$ can be extremely rough, that is it can belong to $H^{\lambda}_{p,\theta-\alpha p/2,\sigma}(\overline{\cO}^c)$ with $\lambda \approx -\infty$, without hurting the regularity of the solution. 
 We next refer to \cite{BCI} (resp. \cite{GKL}) for the existence and uniqueness  results of viscosity (resp. distributional) solution of  non-local elliptic equations.
Also see \cite{FKV,HJ,R} for $L_2$-solvability results of non-local elliptic equations. Finally we refer to \cite{Bogdan trace,Bogdanlptr,Kasstr,Grube} for the extension and trace problems for non-local elliptic operators.
For instance, in \cite{Bogdan trace}, the following Douglas-type formula for the quadratic form of the Poisson extension is introduced; for given $g:\cO^c\to\bR$,
\begin{align*}
&\int\int_{\bR^d\times\bR^d\setminus \cO^c\times \cO^c} (P_{\cO}g(x)-P_{\cO} g(y))^2 \nu(x,y)dxdy
\\
&= \int\int_{\cO^c\times \cO^c} (g(w)-g(z))^2\gamma_{\cO}(z,w)dwdz,
\end{align*}
where $\cO^c:=\bR^d \setminus \cO$.
Here, $\nu(x,y)$ is a unimodal L\'evy measure, $P_\cO$ is the Poisson kernel associated to the operator generated by $\nu$, and $\gamma_\cO$ is the interaction kernel (see (2.8) in \cite{Bogdan trace}). In particular, if $\nu(x,y)=c_d |x-y|^{-d-\alpha}$, then $P_\cO g(x)$ is a solution to \eqref{elliptic eqn} with $f=0$.

\section{Main results} \label{Main results}

Throughout this article, we assume that $\cO$ is either a half space $\bR^d_+$ or a bounded $C^{1,1}$ open set.

\subsection{Function spaces} \label{Secfunc}

First we introduce notations used in this article. As usual $\bR^d$ stands for the Euclidean space of points $x=(x^1,\dots,x^d)$, 
$\bR^d_+=\{(x^1,\dots,x^d)\in\bR : x^1>0\}$ is a half space, and 
$B_r(x)=\{y\in\bR^d : |x-y|<r\}$.
We use $``:="$ or $``=:"$ to denote a definition.  $\bN$ and $\bZ$ denote the natural number system and the integer number system, respectively.  We denote $\bN_+:=\bN\cup\{0\}$. For nonnegative functions $f$ and $g$, we write $f(x)\approx g(x)$  if there exists a constant $N>0$, independent of $x$,  such that $N^{-1} f(x)\leq g(x) \leq N f(x)$. For  multi-indices $\beta=(\beta_1,\cdots,\beta_d)$, $\beta_i\in\bN_+$, and functions $u(x)$ depending on $x$,
$$
 D_iu(x):=\frac{\partial u}{\partial x^i},\quad  D^{\beta}_xu(x):=D^{\beta_d}_{d}\cdots D_1^{\beta_1}u(x).
$$
We also use  $D^n_xu$ to denote   the  partial derivatives of order $n\in\bN_+$ with respect to the space variables.  By $\cF$ and $\cF^{-1}$ we denote the $d$-dimensional Fourier transform and the inverse Fourier transform respectively,
 i.e.,
$$
\cF[f](\xi):=\int_{\bR^d} e^{-i\xi\cdot x} f(x) dx, \quad  \cF^{-1}[f](x):=\frac{1}{(2\pi)^d}\int_{\bR^d} e^{i\xi\cdot x} f(\xi) d\xi.
$$
For an open set $U\subset \bR^d$, $C_b(U)$ denotes the space of continuous functions $u$ in $U$ such that $|u|_{C_b(U)}:=\sup_U |u(x)|<\infty$. 
By $C^n_b(U)$ we denote the space of functions whose derivatives of order up  to $n$ are in $C_b(U)$. Here we drop $U$ if $U=\bR^d$.   
For an open set $V\subset \bR^d$, where $d\in \bN$, by $C_c^\infty(V)$ we denote the space of infinitely differentiable functions with compact support in $V$. Here, we extend a function $u\in C_c^\infty(V)$ to all of $\bR^d$ by letting $f(x)=0$ on $V^c$, if necessary.
For  a Banach space $F$ and $\delta\in (0,1]$,   $C^{\delta}(V;F)$ denotes the space of $F$-valued continuous functions $u$ on $V$  such that
\begin{eqnarray*}
    |u|_{C^{\delta}(V;F)}&:=&|u|_{C_b(V;F)}+[u]_{C^{\delta}(V;F)}
    \\
    &:=&
     \sup_{x\in V}|u(x)|_F+\sup_{x,y\in V}\frac{|u(x)-u(y)|_F}{|x-y|^{\delta}}<\infty.
\end{eqnarray*}
Also, for  $p>1$ and a measure $\mu$ on $V$, $L_p(V, \mu; F)$  denotes the set of $F$-valued Lebesgue measurable functions $u$ such that 
$$
\|u\|_{L_p(V, \mu; F)}:=\left(\int_V |u|^p_F \,d\mu\right)^{1/p}<\infty.
$$
 We drop $F$ and $\mu$ if $F=\bR$ and $\mu$ is the Lebesgue  measure.
 For a Banach space $F$, the dual space of $F$ is denoted by $F^*$.
By  $\cD'(U)$, where $U$ is an open set in $\bR^d$, 
we denote the space of all distributions on $U$, and  for given $f\in \cD'(U)$, the action of $f$ on $\phi \in C_c^\infty(U)$ is denoted by
$$
( f, \phi)_{U} :=f(\phi).
$$
 For open sets $V\subset U\subset \bR^d$ and $f\in \cD'(U)$, the restriction of $f$ to $V$ is denoted by $r_Vf$. For $\cD'(\bR^d)$-valued function $g$ defined on $(0,T)$, $r_{(0,T)\times U}g(t,\cdot):=r_U(g(t))(\cdot)$ for each $t\in (0,T)$. The extension by zero from $U\subset \bR^d$  to $\bR^d$ is denoted by $e_{U}$.
Finally, if  we write $N=N(a,b,\cdots)$, then this means that the constant $N$ depends only on $a,b,\cdots$.

Now we recall Sobolev and Besov spaces on $\bR^d$.
For $p\in(1,\infty)$ and $\gamma\in\bR$,  the Sobolev space $H_p^\gamma=H_p^{\gamma}(\bR^d)$ is defined as the space of all tempered distributions $f$ on $\bR^d$ satisfying
$$
\|f\|_{H_p^{\gamma}}:=\|(1-\Delta)^{\gamma/2}f\|_{L_p}<\infty,
$$
where
$$
(1-\Delta)^{\gamma/2} f(x) := \cF^{-1} \left[(1+|\cdot|^2)^{\gamma/2}\cF [f] \right](x).
$$
For $T\in(0,\infty)$, we define
$$\bH_p^{\gamma}(T):=L_p((0,T);H_p^{\gamma}),\quad \bL_p(T):=\bH_p^0(T)=L_p((0,T);L_p).
$$
Next we take a smooth function $\Psi$ whose Fourier transform $ \cF[\Psi]$ is infinitely differentiable, supported in an annulus $\{\xi\in\bR^d : \frac{1}{2} \leq |\xi| \leq 2\}$, $\cF[\Psi]\geq0$ and
$$
\sum_{j\in \bZ} \cF[\Psi](2^{-j}\xi)=1, \qquad \forall \, \xi\neq0.
$$
For a tempered distribution $f$ and $j\in \bZ$, we denote
$$
\Delta_j f(x):=\cF^{-1}\left[\cF[\Psi](2^{-j}\cdot)\cF [f]\right](x), \qquad 
S_0 f(x):= \sum_{ j=-\infty}^0 \Delta_j f(x).
$$
 The Besov space $B_p^\gamma=B_p^\gamma(\bR^d)$, where $p>1, \gamma\in \bR$,  is defined as the space of all tempered distributions $f$ satisfying
$$
\|f\|_{B_p^\gamma}:=\| S_0 f\|_{L_p} + \left(\sum_{j=1}^\infty 2^{\gamma p j} \| \Delta_j f \|_{L_p}^p \right)^{1/p} < \infty.
$$

Now, we introduce weighted spaces on $\fO \subset\bR^d$, where $\fO$ is either $\cO$ or $\overline{\cO}^c$.  Here, $\overline{\cO}^c$ is the complement of the closure of $\cO$.
We denote $d_x:=d_{\cO,x}:=dist (x, \partial \cO)$ and 
 $$L_{p,\theta,\sigma}(\fO):=L_p(\fO, d_x^{\theta-d} (1+d_x)^{\sigma}dx)$$
 and
 $$
 H^{n}_{p,\theta,\sigma}(\fO):=\{u: u, d_x D_xu, \cdots, d_x^n D^n_xu \in L_{p,\theta, \sigma}(\fO)\}
 $$
   for any  $p>1, \theta,\sigma\in \bR$ and $n\in \bN_+$. The norm in this space is defined as
\begin{align} \label{21.10.06.0918}
\|u\|_{H_{p,\theta,\sigma}^{n}(\fO)}=\sum_{|\beta|\leq n} \left( \int_{\fO}|d_x^{|\beta|}D_x^{\beta}u(x)|^p d_x^{\theta-d}(1+d_x)^{\sigma} dx\right)^{1/p}.
\end{align}

Next, we generalize these spaces to arbitrary order weighted Sobolev and Besov spaces.
Take a sequence of nonnegative functions $\zeta_n \in C^{\infty}_c (\fO), n\in \bZ$,  having the following properties:
\begin{align} 
    &(i)\,\,supp (\zeta_n) \subset \{x\in \fO : k_1e^{-n}< d_x <k_2e^{-n}\}, \quad k_2>k_1>0, \label{zeta prop 1}
    \\
    &(ii)\,\,\sup_{x\in\bR^d}|D^m_x \zeta_n (x)| \leq N(m)e^{mn},\quad \forall m\in\bN_+ \label{zeta prop 2}
    \\
    &(iii)\,\,\sum_{n\in\bZ} \zeta_n(x) > c>0,\quad\forall x\in \fO. \label{zeta prop 3}
\end{align}
If $\{x\in \fO: k_1e^{-n}< d_x <k_2e^{-n}\}=\emptyset$ , we just take $\zeta_n=0$. One can easily construct such functions by mollifying indicator functions of the sets of the type $\{x\in\fO:k_3e^{-n}<d_x<k_4e^{-n}\}$.
We remark that the constant $e$ in \eqref{zeta prop 1} and \eqref{zeta prop 3} is not crucial. In other words, one may assume that $supp (\zeta_n) \subset \{x\in \fO : k_1c^{-n}< d_x <k_2c^{-n}\}$ with $c>0$, instead of \eqref{zeta prop 1}. However, in this paper, the constant $e$ is used for consistency with \cite{Dirichlet}.

Now we define weighted Sobolev spaces $H^{\gamma}_{p,\theta,\sigma}(\fO)$ and weighted Besov spaces $B^{\gamma}_{p,\theta}(\fO)$ for any $\gamma, \theta\in \bR$ and $p>1$. Note that for any distribution $u$ on $\fO$, one can define $\zeta_{-n}u$ as a distribution on $\bR^d$. Obviously, the action of $\zeta_{-n}u$ on $C_c^\infty(\bR^d)$ is defined as
$$
(\zeta_{-n}u,\phi)_{\bR^d}=(u,\zeta_{-n}\phi)_{\fO}, \quad \forall \phi\in C_c^\infty(\fO).
$$
By $H_{p,\theta,\sigma}^{\gamma}(\fO)$ and $B_{p,\theta}^{\gamma}(\fO)$, we denote the sets of  distributions $u$ on $\fO$ such that
\begin{equation}
\label{def sobolev}
\|u\|_{H_{p,\theta,\sigma}^{\gamma}(\fO)}^p:=\sum_{n\in\bZ}e^{n\theta} (1+e^n)^{\sigma} \|\zeta_{-n}(e^n\cdot)u(e^n\cdot)\|_{H_p^{\gamma}}^p<\infty,
\end{equation}
and
\begin{equation*}
\|u\|_{B_{p,\theta}^{\gamma}(\fO)}^p:=\sum_{n\in\bZ}e^{n\theta}  \|\zeta_{-n}(e^n\cdot)u(e^n\cdot)\|_{B_{p}^{\gamma}}^p<\infty,
\end{equation*}
respectively. We also denote $H^{\gamma}_{p,\theta}(\fO):=H^{\gamma}_{p,\theta,0}(\fO)$.
  One can easily check that those spaces are independent of choice of $\{\zeta_n\}$ (see e.g. \cite[Proposition 2.2]{Lototsky}). More precisely,  if $\{\zeta_{n}\}$ satisfies \eqref{zeta prop 1} and \eqref{zeta prop 2}, then
  $$
\sum_{n\in\bZ} e^{n\theta} (1+e^n)^{\sigma} \|\zeta_{-n}(e^n\cdot)u(e^n\cdot) \|_{H_p^\gamma}^p \leq N \|u\|^p_{H_{p,\theta,\sigma}^{\gamma}(\fO)}.
  $$
  The reverse inequality holds if $\{\zeta_{n}\}$ additionally satisfies \eqref{zeta prop 3}. Furthermore, if $\gamma=n \in \bN_+$,  then  the norms defined in \eqref{21.10.06.0918}  and \eqref{def sobolev} are equivalent.
Also, the similar properties are valid for the spaces $B_{p,\theta}^{\gamma}(\fO)$.

The weighted Besov space $B_{p,\theta}^{\gamma}(\fO)$ is introduced as above to handle the initial data, based on results in the whole space $\bR^d$. It is known that, for the parabolic Sobolev space given on $(0,T)\times \bR^d$, the trace space for initial data is the Besov space $B_p^\gamma(\bR^d)$ (see \cite[Theorem 1.8.2]{TrIn}). From this, $B_{p,\theta}^{\gamma}(\fO)$ naturally arises in accordance with \eqref{def sobolev}.

Next, we choose (cf. \cite{KK2004}) an infinitely differentiable function $\psi$ in $\bR^d\setminus \partial \cO$ such that $\psi\approx d_x$ and for any $m\in \bN_+$
\begin{align*}
\sup_{\bR^d\setminus\partial \cO} |d_x^m D^{m+1}_x \psi(x)|\leq N(m)<\infty.
\end{align*}
For instance, one can take $\psi(x):=\sum_{n\in\bZ} e^{-n} \zeta_n(x)$ on $\fO$ and similarly define $\psi$ on $\overline{\fO}^c$.

Below we collect some basic properties of the spaces $H^{\gamma}_{p,\theta,\sigma}(\fO)$ and $B_{p,\theta}^\gamma(\fO)$. For $\nu_1,\nu_2\in \bR$, we write $u\in \psi^{-\nu_1}(1+\psi)^{-\nu_2} H_{p,\theta,\sigma}^{\gamma}(\fO)$ (resp. $u\in \psi^{-\nu_1} B_{p,\theta}^{\gamma}(\fO)$) if $\psi^{\nu_1} (1+\psi)^{\nu_2} u \in  H_{p,\theta,\sigma}^{\gamma}(\fO)$ (resp. $\psi^{\nu_1}  u \in B_{p,\theta}^{\gamma}(\fO)$).

\begin{lem}\label{lem space}
Let $\gamma,\theta,\sigma \in \bR$ and $p\in(1,\infty)$.

(i) The space $C^{\infty}_c(\fO)$ is dense in $H^{\gamma}_{p,\theta,\sigma}(\fO)$ and $B^{\gamma}_{p,\theta}(\fO)$.

(ii) For each $i=1,2,\cdots, d$, the operator $D_i$ is a bounded operator from $H^{\gamma}_{p,\theta-p,\sigma}(\fO)$ (resp. $B^{\gamma}_{p,\theta-p}(\fO)$) to $H^{\gamma-1}_{p,\theta,\sigma}(\fO)$ (resp. $B^{\gamma-1}_{p,\theta}(\fO)$).

(iii) For $\delta_1,\delta_2\in\bR$, $H_{p,\theta,\sigma}^\gamma (\fO) = \psi^{\delta_1} (1+\psi)^{\delta_2} H_{p,\theta+\delta_1 p,\sigma+\delta_2 p}^\gamma(\fO)$ and $B_{p,\theta}^\gamma (\fO) = \psi^{\delta_1} B_{p,\theta+\delta_1 p}^\gamma(\fO)$. Moreover,
\begin{align*}
&\|u\|_{H_{p,\theta,\sigma}^\gamma (\fO)}\approx \|\psi^{-\delta_1} (1+\psi)^{-\delta_2} u\|_{H_{p,\theta+\delta_1 p,\sigma+\delta_2 p}^\gamma(\fO)}, 
\end{align*}
and
\begin{align*}
&\|u\|_{B_{p,\theta}^\gamma (\fO)} \approx \|\psi^{-\delta_1} u\|_{B_{p,\theta+\delta_1 p}^\gamma(\fO)}.
\end{align*}

(iv) (Duality) Let $1/p+1/p'=1, \theta/p+\theta'/p' = d$, and $\sigma/p+\sigma'/p'=0$.
Then, the dual spaces of $H_{p,\theta,\sigma}^\gamma(\fO)$ and $B_{p,\theta}^\gamma(\fO)$ are $H_{p',\theta',\sigma'}^{-\gamma}(\fO)$ and $B_{p',\theta'}^{-\gamma}(\fO)$, respectively.
 Moreover, for $u\in H_{p,\theta,\sigma}^\gamma(\fO)$ (resp. $u\in B_{p,\theta}^\gamma(\fO)$), $(u,\phi)_{\fO}$, defined on $C_c^\infty(\fO)$, can be extended by continuity to $H_{p',\theta',\sigma'}^{-\gamma}(\fO)$ (resp.$B_{p',\theta'}^{-\gamma}(\fO)$).

(v) (Sobolev-H\"older embedding)

Let
$\gamma-\frac{d}{p} \geq n+\delta$ for some $n\in\bN_+$ and $\delta\in(0,1)$. Then for any $k\leq n$,
\begin{eqnarray*}
    &|\psi^{k+\frac{\theta}{p}} (1+\psi)^{\frac{\sigma}{p}} D_x^k u|_{C_b(\fO)}+[\psi^{n+\frac{\theta}{p}+\delta} (1+\psi)^{\frac{\sigma}{p}} D_x^{n}u]_{C^{\delta}(\fO)}
    \\
    &\leq N(d,\gamma,p,\theta,\sigma)\|u\|_{H_{p,\theta,\sigma}^{\gamma}(\fO)}.
\end{eqnarray*}
\end{lem}

\begin{proof}
We only treat the special case $H_{p,\theta}^\gamma(\fO)=H_{p,\theta,0}^\gamma(\fO)$ since the proof for general spaces goes the same way.

When $\cO$ is a half space, the claims are proved by Krylov in \cite{KrylovSome}, and those are generalized by Lototsky in \cite{Lototsky} for arbitrary domains. See Proposition 2.2 and 2.4, and Theorems 4.1 and 4.3 in \cite{Lototsky}.
Here we remark that those in \cite{Lototsky} are still valid for general bounded $C^{1,1}$ open sets or their complements.
 The lemma is proved.
\end{proof}

\begin{rem} \label{rem10081613}
For $\sigma\in \bR$, we have $u\in H^{\gamma}_{p,\theta,\sigma}(\fO)$ if and only if $(1+\psi)^{\sigma/p}u \in H^{\gamma}_{p,\theta}(\fO)$, and 
\begin{equation*}
\|u\|_{H^{\gamma}_{p,\theta,\sigma}(\fO)} \approx \|(1+\psi)^{\sigma/p}u\|_{H^{\gamma}_{p,\theta}(\fO)}.
\end{equation*}
Also, using \eqref{def sobolev}, one can easily check $H^{\gamma}_{p,\theta,\sigma_2}(\fO) \subset H^{\gamma}_{p,\theta,\sigma_1}(\fO)$ if $\sigma_1<\sigma_2$.
 In particular, for any $\sigma<0$, $H^{\gamma}_{p,\theta}(\fO)\subset H^{\gamma}_{p,\theta,\sigma}(\fO)$.
Similarly,  if $\fO$ is bounded and $\theta_1<\theta_2$ then $H^{\gamma}_{p,\theta_1, \sigma}(\fO)\subset H^{\gamma}_{p,\theta_2,\sigma}(\fO)$.
\end{rem}

By $e_{\fO}H^{\gamma}_{p,\theta,\sigma}(\fO)$, we denote the space of distributions $u\in\cD'(\bR^d)$ such that $supp(u) \subset \overline{\fO}$ and $r_{\fO}u\in H^{\gamma}_{p,\theta,\sigma}(\fO)$.

In general, for a (tempered) distribution $u$ on $\bR^d$, since $\Delta^{\alpha/2}\phi$ may not be in the Schwartz space $\cS(\bR^d)$ for $\phi\in \cS(\bR^d)$, $\Delta^{\alpha/2}u$ may also not be well defined (see e.g. \cite{K17}). However, if $u$ is sufficiently regular in $\bR^d$, say $u\in H^{\alpha}_p(\bR^d)$, then for any $\phi\in C_c^\infty(\cO)$,
\begin{equation}
\label{eqn 3.14.5}
\langle \Delta^{\alpha/2}u,\phi\rangle_{\cO} =\langle\Delta^{\alpha/2}u,\phi\rangle_{\bR^d} = \langle u,\Delta^{\alpha/2}\phi\rangle_{\cO}+\langle u,\Delta^{\alpha/2}\phi\rangle_{\overline{\cO}^c},
\end{equation}
where $\langle f,g\rangle_{\fO}:=\int_{\fO} fgdx$.
Thus, $\Delta^{\alpha/2}u$ induces a distribution on $\cO$. Now, we consider $\Delta^{\alpha/2}$ on weighted Sobolev spaces. In \eqref{frac dist} below, we define $\Delta^{\alpha/2}u$ in accordance with relation \eqref{eqn 3.14.5}.
The proof of the lemma is given in Section \ref{sec prop}.

\begin{lem} \label{thm frac def}
Let $\gamma,\lambda\in\bR$, $\alpha\in(0,2)$, $p\in(1,\infty)$ and 
$d-1-\alpha p/2<\theta< d-1+p+\alpha p/2$.

(i) Let $\sigma\in \bR$. Then, for any $\eta\in C_c^\infty(\bR^d)$, we have $r_{\cO}\eta \in (H^{\gamma}_{p,\theta-\alpha p/2}(\cO))^*$ and $r_{\overline{\cO}^c}\eta \in (H^{\lambda}_{p,\theta-\alpha p/2,\sigma}(\overline{\cO}^c))^*$. Moreover, for given $v\in H^{\gamma}_{p,\theta-\alpha p/2}(\cO)$ and $w\in H^{\lambda}_{p,\theta-\alpha p/2,\sigma}(\overline{\cO}^c)$, the functional
\begin{align} \label{eq2311111704}
u(\eta):=(u,\eta)_{\bR^d} := (v,r_{\cO}\eta)_{\cO}+(w,r_{\overline{\cO}^c}\eta)_{\overline{\cO}^c}, \quad \eta \in C_c^\infty(\bR^d),
\end{align}
is well defined and $u\in e_{\cO}H^{\gamma}_{p,\theta-\alpha p/2}(\cO) \oplus e_{\overline{\cO}^c}H^{\lambda}_{p,\theta-\alpha p/2,\sigma}(\overline{\cO}^c)$, where $\oplus$ denotes the span of two linearly independent subspaces of $\cD'(\bR^d)$.

(ii) For any $\phi \in C^{\infty}_c(\cO)$,  we have
$$
r_{\cO}\Delta^{\alpha/2} (e_{\cO}\phi) \in (H^{\gamma}_{p,\theta-\alpha p/2}(\cO))^*, \,\text{ and }\, r_{\overline{\cO}^c}\Delta^{\alpha/2} (e_{\cO}\phi) \in (H^{\lambda}_{p,\theta-\alpha p/2,\sigma}(\overline{\cO}^c))^*,
$$
provided that $\sigma>-\theta-\alpha p/2$ if $\cO$ is bounded, and $\sigma=0$ if $\cO$ is a half space. 

(iii) Assume that $\sigma$ satisfies the condition in (ii). For $u\in e_{\cO}H^{\gamma}_{p,\theta-\alpha p/2}(\cO) \oplus e_{\overline{\cO}^c}H^{\lambda}_{p,\theta-\alpha p/2,\sigma}(\overline{\cO}^c)$,
the functional $\Delta^{\alpha/2}u$ defined as
\begin{align} \label{frac dist}
(\Delta^{\alpha/2}u,\phi)_{\cO} 
:= (r_{\cO}u,r_{\cO}\Delta^{\alpha/2} (e_{\cO}\phi))_{\cO}+(r_{\overline{\cO}^c}u,r_{\overline{\cO}^c}\Delta^{\alpha/2}(e_{\cO}\phi))_{\overline{\cO}^c}, \, \phi \in C^{\infty}_c(\cO),
\end{align}
is well defined and belongs to $H^{\gamma-\alpha}_{p,\theta+\alpha p/2}(\cO)$. 
 Moreover, we have
\begin{align} \label{ineq 1025-1}
&\|\psi^{\alpha/2} \Delta^{\alpha/2}u\|_{H^{\gamma-\alpha}_{p,\theta}(\cO)} \nonumber
\\
&\leq N \left(\|\psi^{-\alpha /2}r_{\cO}u\|_{H_{p,\theta}^{\gamma}(\cO)} + \|\psi^{-\alpha/2}(1+\psi)^{\sigma}r_{\overline{\cO}^c}u\|_{H_{p,\theta}^\lambda(\overline{\cO}^c)}\right).
\end{align}
\end{lem}

\begin{rem}
$(i)$ Thanks to Lemma \ref{lem space}$(iv)$, we can consider both $(v,r_{\cO}\eta)_{\cO}$ and $(w,r_{\overline{\cO}^c}\eta)_{\overline{\cO}^c}$ in \eqref{eq2311111704} although $r_{\cO}\eta$ and $r_{\overline{\cO}^c}\eta$ may not be in $C_c^\infty(\cO)$ and $C_c^\infty(\overline{\cO}^c)$, respectively.

$(ii)$ Under the assumptions of Lemma \ref{thm frac def}, for $v\in H^{\gamma}_{p,\theta-\alpha p/2}(\cO)$, one can find $\tilde{v}\in e_{\cO}H^{\gamma}_{p,\theta-\alpha p/2}(\cO)$ such that $r_{\cO}\tilde{v}=v$.

$(iii)$ Let conditions on $\gamma,\theta,\lambda, \sigma$ used  Lemma \ref{thm frac def}$(iii)$ hold. Obviously, if $\gamma, \lambda \geq 0$, that is $u$ is sufficiently regular, then we have
$$
(\Delta^{\alpha/2}u,\phi)_{\cO} = \langle u,\Delta^{\alpha/2}(e_{\cO}\phi)\rangle_{\cO}+\langle u,\Delta^{\alpha/2}(e_{\overline{\cO}^c}\phi)\rangle_{\overline{\cO}^c}, \quad \phi\in C_c^\infty(\cO).
$$

$(iv)$ If $\Delta^{\alpha/2}u\in\cD(\bR^d)$, which means that it is a well defined distribution on $\bR^d$, then the functional $\Delta^{\alpha/2}u$ in \eqref{frac dist} is actually $r_{\cO}\Delta^{\alpha/2}u$. However, for $u\in e_{\cO}H^{\gamma}_{p,\theta-\alpha p/2}(\cO) \oplus e_{\overline{\cO}^c}H^{\lambda}_{p,\theta-\alpha p/2,\sigma}(\overline{\cO}^c)$, $\Delta^{\alpha/2}u$ may not be defined on the whole space $\bR^d$. For instance, \eqref{ineq 1025-1} with $\gamma=\alpha$ implies that
\begin{equation*}
  \int_{D} |\Delta^{\alpha/2}u|^p d_x^{\theta-d+\alpha p/2} dx <\infty,
\end{equation*}
which allows $\Delta^{\alpha/2}u$ to blow up near the boundary $\partial\cO$.
Thus, in this paper, we do not employ the restriction notation for $\Delta^{\alpha/2}$. Furthermore, since we consider $\Delta^{\alpha/2}u$ with a distribution defined only on $\cO$, we do not explicitly specify the domain $\cO$ for $\Delta^{\alpha/2}$. 
  \end{rem}

Now we define the weak solution.

\begin{defn} \label{sol def}

$(i)$ (Parabolic problem)  Suppose $f(t,\cdot), u_0 \in \cD'(\cO)$ and $g(t,\cdot)\in \cD'(\overline{\cO}^c)$ for each $0<t< T$.   Then,
we say that a $\cD'(\bR^d)$-valued function $u$ is a (weak) solution to the problem
\begin{equation} \label{para def}
\begin{cases}
\partial_t u(t,x)=\Delta^{\alpha/2}u(t,x)+f(t,x),\quad &(t,x)\in(0,T)\times \cO,
\\
u(0,x)=u_0(x),\quad & x\in \cO,
\\
u(t,x)=g(t,x),\quad &(t,x)\in (0,T)\times \overline{\cO}^c,
\end{cases}
\end{equation}
if (a) $\Delta^{\alpha/2}u(t,\cdot)$ defined by \eqref{frac dist} makes sense and belongs to $\cD'(\cO)$,
 (b) $u(t,\cdot)=g(t,\cdot)$ on $\overline{\cO}^c$ for each $t\in (0,T)$, (c) for any $\phi\in C^{\infty}_c(\cO)$,
 \begin{equation*}
(\Delta^{\alpha/2}u(s,\cdot),\phi)_{\cO}, \, (f(s,\cdot),\phi)_{\cO} \in L_1((0,T)),
 \end{equation*}
 and the equality
\begin{align*}
(u(t,\cdot),\phi)_{\cO} &= (u_0,\phi)_{\cO} + \int_0^t (\Delta^{\alpha/2}u(s,\cdot),\phi)_{\cO} ds + \int_0^t (f(s,\cdot),\phi)_{\cO} ds
\end{align*}
holds for all $t < T$.

$(ii)$ (Elliptic problem) 
Let $f\in \cD'(\cO)$ and $g\in \cD'(\overline{\cO}^c)$. We say that $u\in \cD'(\bR^d)$ is a (weak) solution to the problem
\begin{equation} \label{elliptic def}
\begin{cases}
\Delta^{\alpha/2}u(x) = f(x),\quad &x\in \cO,
\\
u(x)=g(x),\quad &x\in \overline{\cO}^c,
\end{cases}
\end{equation}
if  (a)  $\Delta^{\alpha/2}u$ defined by \eqref{frac dist} makes sense and belongs to $\cD'(\cO)$, (b) $u=g $ on $\overline{\cO}^c$, (c) for any  $\phi\in C^{\infty}_c(\cO)$ we have
\begin{equation*}
(\Delta^{\alpha/2}u,\phi)_{\cO} = (f,\phi)_{\cO}.
\end{equation*}
\end{defn}

\subsection{Main results} \label{Regularity of solution}

 For $T>0, p>1$ and $\gamma, \theta, \sigma \in \bR$,  we denote
\begin{equation*}
\begin{aligned}
\bH_{p,\theta,\sigma}^\gamma(\fO,T) := L_{p}((0,T);H_{p,\theta,\sigma}^\gamma(\fO)).
\end{aligned}
\end{equation*}
We write $u\in \frH_{p,\theta}^{\gamma+\alpha}(\cO,T)$ if $u\in \psi^{\alpha/2}\bH_{p,\theta}^{\gamma+\alpha}(\cO,T)$, $u(0,\cdot) \in \psi^{\alpha/2-\alpha/p} B_{p,\theta}^{\gamma+\alpha-\alpha /p}(\cO)$, and  there exists $\tilde{f}\in \psi^{-\alpha/2}\bH_{p,\theta}^{\gamma} (\cO,T)$ such that for any $\phi\in C_c^\infty(\cO)$ 
$$
(u(t,\cdot),\phi)_{\cO}=(u(0,\cdot),\phi)_{\cO}+\int_0^t (\tilde{f}(s,\cdot),\phi)_{\cO} \,ds, \quad \forall \, t < T.
$$
Here, we write $u_t:=\partial_t u:=\tilde{f}$.  The norm in  $ \frH_{p,\theta}^{\gamma+\alpha}(\cO,T)$ is defined as 
\begin{align*}
\|u\|_{\frH_{p,\theta}^{\gamma+\alpha}(\cO,T)} :=& \|\psi^{-\alpha/2} u\|_{\bH_{p,\theta}^{\gamma+\alpha}(\cO,T)} + \|\psi^{\alpha/2} u_t\|_{\bH_{p,\theta}^{\gamma}(\cO,T)}
\\
& + \|\psi^{-\alpha/2+\alpha/p} u(0,\cdot) \|_{B_{p,\theta}^{\gamma+\alpha-\alpha /p}(\cO)}.
\end{align*}
Similar to $e_{\fO}H^{\gamma}_{p,\theta,\sigma}(\fO)$, by $e_{(0,T)\times\fO}\bH^{\gamma}_{p,\theta,\sigma}(\fO,T)$, we denote the space of $\cD'(\bR^d)$-valued distribution $u$ such that $supp(u(t,\cdot))\subset \overline{\fO}$ and $r_{(0,T)\times\fO}u\in\bH^{\gamma}_{p,\theta,\sigma}(\fO,T)$. One can also define $e_{(0,T)\times\cO}\frH^{\gamma+\alpha}_{p,\theta}(\cO,T)$ in a similar manner.

\begin{rem}
\label{remark 3.18}
$(i)$ Let $u$ be a weak solution to equation \eqref{para def}, and $\theta$ and $\sigma$ satisfy the conditions  in Lemma \ref{thm frac def}$(iii)$.
Then, by Lemma \ref{thm frac def}, $r_{(0,T)\times \cO}u\in \frH^{\gamma+\alpha}_{p,\theta}(\cO,T)$ if $r_{(0,T)\times \cO}u\in \psi^{\alpha/2}\bH_{p,\theta}^{\gamma+\alpha}(\cO,T)$, $f\in \psi^{-\alpha/2}\bH_{p,\theta}^{\gamma} (\cO,T)$, $g\in \bH^{\lambda}_{p,\theta,\sigma}(\overline{\cO}^c,T)$, and $u_0 \in \psi^{\alpha/2-\alpha/p} B_{p,\theta}^{\gamma+\alpha-\alpha /p}(\cO)$. Here, $\lambda\in \bR$. Moreover, in this case, $u_t=\Delta^{\alpha/2}u+f$ and 
\begin{align*}
&\|r_{(0,T)\times \cO}u\|_{\frH_{p,\theta}^{\gamma+\alpha}(\cO,T)} 
\\
&\leq N \|\psi^{-\alpha/2} r_{(0,T)\times \cO}u\|_{\bH_{p,\theta}^{\gamma+\alpha}(\cO,T)} + \|\psi^{\alpha/2} f\|_{\bH_{p,\theta}^{\gamma}(\cO,T)}
\\
&\quad+ \|\psi^{-\alpha/2+\alpha/p} u(0,\cdot) \|_{B_{p,\theta}^{\gamma+\alpha-\alpha /p}(\cO)} + \|\psi^{-\alpha/2}g\|_{\bH_{p,\theta,\sigma}^{\lambda}(\overline{\cO}^c,T)}.
\end{align*}

$(ii)$ 
By the trace theorem \cite[Theorem 1.8.2]{TrIn}, the Besov space $B_{p,\theta}^{\gamma+\alpha-\alpha /p}(\cO)$ is the optimal trace space for $\frH_{p,\theta}^{\gamma+\alpha}(\cO,T)$ (see also \cite[Remark 4.8]{Seo}).

\end{rem}

Theorem \ref{main thm para} and Theorem \ref{main thm ellip} below are main results of this article.

\begin{thm}[Parabolic case] \label{main thm para}
Let $p\in(1,\infty)$, $\gamma,\lambda\in\bR$,  $\theta\in(d-1, d-1+p)$, and $\sigma>-\theta-\alpha p/2$ if $\cO$ is bounded, and $\sigma=0$ if $\cO$ is the half space $\bR_+^d$. Then, for any $u_0\in \psi^{\alpha/2-\alpha/p} B_{p,\theta}^{\gamma+\alpha-\alpha/p}(\cO)$, $f \in \psi^{-\alpha/2}\bH_{p,\theta}^{\gamma}(\cO,T)$ and $g \in \psi^{\alpha/2}\bH_{p,\theta,\sigma}^{\lambda}(\overline{\cO}^c,T)$, 
 parabolic equation \eqref{para def} has a unique weak solution $u\in e_{(0,T)\times\cO}\frH^{\gamma+\alpha}_{p,\theta}(\cO,T)\, \oplus\, e_{(0,T)\times\overline{\cO}^c}\bH^{\lambda}_{p,\theta-\alpha p/2,\sigma}(\overline{\cO}^c,T)$ (see  Definition \ref{sol def}(i)), and for this solution $u$ we have
\begin{align} \label{est main para}
\|r_{(0,T)\times \cO}u\|_{\frH_{p,\theta}^{\gamma+\alpha}(\cO,T)} &\leq N \|\psi^{-\alpha/2+\alpha/p}u_0\|_{B_{p,\theta}^{\gamma+\alpha-\alpha/p}(\cO)} + N \|\psi^{\alpha/2}f\|_{\bH_{p,\theta}^{\gamma}(\cO,T)} \nonumber
\\
&\quad + N \|\psi^{-\alpha/2}g\|_{\bH_{p,\theta,\sigma}^{\lambda}(\overline{\cO}^c,T)},
\end{align}
where $N=N(d,p,\alpha,\gamma,\lambda,\theta,\sigma,\cO)$.
\end{thm}

The space $\frH_{p,\theta}^{\gamma+\alpha}(\cO,T)$ is a Banach space (see \cite[Remark 2.8]{Dirichlet}), and this space provides nice regularities in $\cO$, especially near $\partial \cO$, as follows.

\begin{prop}(H\"older regularity for parabolic equation in $\cO$)
\label{Holder para}
Let $u$ taken from Theorem \ref{main thm para}.
Suppose  $1/p<\nu\leq1$  and 
    $$
    \gamma+\alpha-\nu\alpha-\frac{d}{p} \geq n+\delta, \quad n\in \bN_+, \, \delta\in (0,1).
    $$
   Then, 
    \begin{align*}
        &\sum_{k=0}^n|\psi^{k+\frac{\theta}{p}+\alpha\left(\nu-\frac{1}{2}\right)}D^k_x(u-u(0,\cdot))|_{C^{\nu-1/p}([0,T];C(\cO))}
        \\
        &+\sup_{t,s\in[0,T]}\frac{[\psi^{n+\delta+\frac{\theta}{p}+\alpha\left(\nu-\frac{1}{2}\right)}D^n_x(u(t,\cdot)-u(s,\cdot))]_{C^{\delta}(\cO)}}{|t-s|^{\nu-1/p}}\leq N \|r_{(0,T)\times\cO}u\|_{\frH_{p,\theta}^{\gamma+\alpha}(\cO,T)},
    \end{align*}
where $N$ depends only on $d,\nu,p,\theta,\alpha$ and $T$.
    
\end{prop}
\begin{proof}
See e.g. \cite[Proposition 2.15]{Dirichlet}.
\end{proof}

\begin{thm}[Elliptic case] \label{main thm ellip}
Let $p\in(1,\infty)$, $\gamma,\lambda\in\bR$,  $\theta\in(d-1, d-1+p)$, and $\sigma>-\theta-\alpha p/2$ if $\cO$ is bounded and $\sigma=0$ if $\cO$ is a half space. Then, for any $f \in \psi^{-\alpha/2}H_{p,\theta}^{\gamma}(\cO)$ and $g \in \psi^{\alpha/2}H_{p,\theta,\sigma}^{\lambda}(\overline{\cO}^c)$, 
elliptic equation \eqref{elliptic def} has a unique weak solution $u\in e_{\cO}H^{\gamma+\alpha}_{p,\theta-\alpha p/2}(\cO)\, \oplus\, e_{\overline{\cO}^c}H^{\lambda}_{p,\theta-\alpha p/2,\sigma}(\overline{\cO}^c)$ (see Definition \ref{sol def}(ii)), and for this solution $u$ we have
\begin{align} \label{est main ellip}
\|\psi^{-\alpha/2}r_{\cO}u\|_{H_{p,\theta}^{\gamma+\alpha}(\cO)} \leq N \left( \|\psi^{\alpha/2}f\|_{H_{p,\theta}^{\gamma}(\cO)} + \|\psi^{-\alpha/2}g\|_{H_{p,\theta,\sigma}^{\lambda}(\overline{\cO}^c)}\right),
\end{align}
where $N=N(d,p,\alpha,\gamma,\lambda,\theta,\sigma,\cO)$.
\end{thm}

\begin{rem}
We discuss the necessity of assumptions  and optimality of results  in  \ref{main thm ellip}. The similar statements hold for Theorem \ref{main thm para}.

\begin{enumerate}

\item Range of weight parameter $\theta$ for $f$. 

 The range of $\theta$ is optimal in the sense that if $\theta\leq d-1$ or $\theta\geq d-1+p$, then \eqref{est main para} and \eqref{est main ellip} fail to hold even if  $f\in C^{\infty}_c(\cO)$ and $g=0$. 
See \cite[Remark 2.5]{Dirichlet} for details.

\vspace{1mm}

\item Ranges of weight parameters $\theta$ and $\sigma$ for $g$. 

 These ranges are technical  and   derived from the  behavior of the kernel which appears in the representation of solutions (see Lemma \ref{lem zeroth ell}).

\vspace{1mm}

(i) The weight parameter $\theta$ characterizes the boundary behaviors of functions. The range of $\theta$ for $u$ is determined by $f$ as discussed above, and  we  use  the same weight parameter $\theta$ for $g$. To explain the reason for this, for simplicity, let $f=0$ and  $g\in \psi^{\alpha/2}L_{p,\theta,\sigma}(\overline{\cO}^c)$. Then the solution to the elliptic equation is 
 represented as
\begin{equation}
    u(x)=  \int_{\overline{\cO}^c} K_{\cO}(x,y) g(y)dy, \label{eqn 400}
\end{equation}
where $K_{\cO}$  is the Poisson kernel associated with $\Delta^{\alpha/2}$ in $\cO$ (see Section \ref{sec repre}).  If both $x$ and $y$ are near the boundary $\partial \cO$, then $\psi^{-\alpha/2}(x) K_{\cO}(x,y) \leq N {\psi^{-\alpha/2}(y)}$.
This leads us to use the same weight parameter $\theta$ for $u$ and $g$.

\vspace{1mm}

(ii) Due to the decay of $K_{\cO}$ near infinity, we need certain blow up rate of $g$ near infinity to make sense of integral \eqref{eqn 400}. This leads us to the range of $\sigma$ prescribed in Theorem \ref{main thm ellip}.

\vspace{1mm}

\item Optimal regularity (differentiability) relation between solution and free terms. 

  At first, note that  there is no relation between $\lambda$ and $\gamma$, and consequently (weighted) regularity of solution in $\cO$ is not affected by that of exterior condition $g$ in $\overline{\cO}^c$.

\vspace{1mm}
For the relation between $u$ and $f$, consider \eqref{elliptic def} given with $g=0$. If $u$ is a solution such that $r_{\cO}u \in H^{\gamma+\alpha}_{p,\theta-\alpha p/2}(\cO)$, then by Lemma \ref{thm frac def}$(iii)$,  $f=\Delta^{\alpha/2}u$ automatically belongs to $H^{\gamma}_{p,\theta+\alpha p/2}(\cO)$. This shows that the assumptions for $f$ are necessary.

\end{enumerate}

\end{rem}
%
%
%

\begin{rem}
 If we (formally) take $\alpha\to2$, then we get classical results in \cite{KK2004,Krylovhalf} for the case $\alpha=2$.  For a rigorous proof,  we need to verify that the constants in \eqref{est main para} and \eqref{est main ellip} are uniformly bounded as $\alpha\to2$.  While we believe this to be true, we do not address it in this article. Most of our constants are inherited from \cite{Dirichlet}, and thus  to check the uniform bound it is needed to   examine   the constants in \cite{Dirichlet} and its  references.

\end{rem}

    


The following result is a consequence of Lemma \ref{lem space}$(v)$.

\begin{prop}(H\"older regularity for elliptic equation in $\cO$) \label{Holder ellip}
Let $u$ be taken from Theorem \ref{main thm ellip}   and 
    $$
    \gamma+\alpha-\frac{d}{p}\geq n+\delta, \quad n\in \bN_+, \, \delta\in (0,1).
    $$
Then,
    \begin{align*}
        \sum_{k=0}^n|\psi^{k+\frac{\theta}{p}-\frac{\alpha}{2}}D^k_xu|_{C_b(\cO)}+[\psi^{n+\delta+\frac{\theta}{p}-\frac{\alpha}{2}}D^n_xu]_{C^{\delta}(\cO)}\leq N\|\psi^{-\alpha/2}r_{\cO}u\|_{H_{p,\theta}^{\gamma+\alpha}(\cO)},
    \end{align*}
    where $N$ depends only on $d,\nu,p,\theta$ and $\alpha$.
\end{prop}

\begin{rem}
$(i)$ To compare Proposition \ref{Holder ellip} with results in \cite{Ros2014},  let us first consider the case $\gamma=0$, $\theta=d, g=0$ and  $\cO$ is bounded. Proposition \ref{Holder ellip} yields H\"older continuity of solutions provided that  $f$ belongs to appropriate $L_p$ spaces, and the H\"older exponent varies according to $p$. In particular, for any given $\delta\in (0,\alpha)$ and small $\varepsilon>0$, taking $p$ sufficiently large we get
\begin{equation} \label{eq8141636}
|\psi^{-\alpha/2 +\varepsilon}u|_{C_b(\cO)}+ |\psi^{\delta-\frac{\alpha}{2}+\varepsilon}u|_{C^{\delta}(\cO)}\leq
 N \|\psi^{\alpha/2}f\|_{L_{p}(\cO)},
\end{equation}
where the constant $N$ depends also on $p$.  On the other hand, it is shown  in \cite{Ros2014} that,  if $f\in L_{\infty}(\cO)$ and $g=0$, then the following optimal regulaity result holds; there exists $\delta < (\alpha/2 \wedge 1/2)$ such that
\begin{equation}
\label{eqn 4.2.3}
|u|_{C^{\alpha/2}(\bR^d)}+ |\psi^{-\frac{\alpha}{2}}u|_{C^{\delta}(\bR^d)} \leq N \|f\|_{L_\infty(\cO)}.
\end{equation}
Note that \eqref{eqn 4.2.3} gives better H\"older regularity of solution than \eqref{eq8141636} if $f\in L_{\infty}(\cO)$. Our approach  does not lead to \eqref{eqn 4.2.3} even if $f$ is bounded,

$(ii)$ Next we consider higher order H\"older regularity of solutions.  Let $\cO$ be  bounded, and $0<n+\delta<\gamma+\alpha$ where $n\in\bN_0$ and $\delta\in(0,1)$. Then, as in $(i)$, for any small $\varepsilon>0$, there exists (sufficiently large) $p$ such that
\begin{align*} 
  \sum_{k=0}^n|\psi^{k-\frac{\alpha}{2}+\varepsilon}D^k_xu|_{C_b(\cO)}+[\psi^{n+\delta-\frac{\alpha}{2}+\varepsilon}D^n_xu]_{C^{\delta}(\cO)} \leq N \|\psi^{\alpha/2}f\|_{H_{p,d}^\gamma(\cO)}.
\end{align*}
In \cite[Corollary 3.1]{ABC16}, it is shown that if $\partial \cO \in C^2_b$, and $f\in C^{\gamma'}(\cO)$ with $\gamma'\in(0,\alpha/2)$, then 
\begin{equation} \label{ineq1017-2}
  u\in C_{\alpha+\gamma'}^{(-\alpha/2)}(\cO)
\end{equation}
where $C_{\alpha+\gamma'}^{(-\alpha/2)}(\cO)$ denotes the interior (weighted) H\"older space (see e.g. \cite{GT01}).  This and the relation $C^{\gamma'}(\cO)\subset H_{p,d}^\gamma(\cO)$ for any $\gamma'>\gamma$ show that \eqref{ineq1017-2} can give better H\"older regularity of solutions if $f$ is H\"older continuous.

$(iii)$ If $u_0=0$ and $f=0$,  then by  Proposition \ref{Holder para} the solution to the parabolic equation \eqref{para def} is infinitely differentiable in $\cO$. Similar statement also holds for the elliptic equation.

\end{rem}

 The following example shows that exterior data $g$ can be rough near the boundary.
\begin{exam}
 Let $\cO$ be bounded and  
 $$g(x)=D^{\beta}_x F(x),$$
where  $\beta$ is a multi-index, and  $F$ is a function defined on $\overline{\cO}^c$ such that 
\begin{equation*}
|F(x)|\leq \begin{cases} N \psi^{ \upsilon+|\beta|}(x), \, \quad  \upsilon>-1+\frac{\alpha}{2} &: d_x \leq c,  \,x\in \overline{\cO}^c, 
\\
N \psi(x)^{\alpha-\varepsilon}, \quad \varepsilon>0 \quad &: d_x\geq c,  \,x\in \overline{\cO}^c.
\end{cases}
\end{equation*}
We take 
 $\sigma=-\theta-\alpha p/2+\kappa$, where  $\kappa\in (0,\varepsilon p)$,  such that $\sigma>-\theta-\alpha p/2$ and
\begin{align*}
&\|F\|^p_{L_{p,\theta-|\beta| p-\alpha p/2,\sigma}(\overline{\cO}^c)}
\\
&\leq N \int_{\{x\in \overline{\cO}^c: \psi(x)\leq c\}} d_x^{\theta-d-\alpha p/2+\upsilon p}dx + N \int_{\{x\in \overline{\cO}^c: \psi(x)\geq c\}} \frac{1}{d_x^{d+\varepsilon p-\kappa}}dx<\infty,
\end{align*}
provided that 
\begin{equation}
\label{eqn 3.30.1}
\theta-d-\alpha p/2+\upsilon p>-1.
\end{equation}
Therefore by Lemma \ref{lem space}$(ii)$, if \eqref{eqn 3.30.1} holds, then $g\in H^{-|\beta|}_{p,\theta-\alpha p/2,\sigma}(\overline{\cO}^c)$.
Note that since $\upsilon > -1+\alpha/2$, one can find 
$\theta_0=\theta_0(\upsilon)<d-1+p$ such that \eqref{eqn 3.30.1} holds for $\theta\in (\theta_0, d-1+p)$.
Thus, from the case $\beta=0$, measurable function that blows up near the boundary can be considered as an exterior data $g$.
\end{exam}

We also consider functions(distributions) whose supports are disjoint from $\overline{\cO}$ (see also \cite[Corollary 2.6]{Grubb nonlocal}).

\begin{exam}
  $(i)$ Let $\theta \in (d-1,d-1+p)$, $\sigma:=\alpha p/2-\theta+d$, $\delta\in \bR$, and $g$ belong to either $H_p^\delta(\bR^d)$ or $B_p^\delta(\bR^d)$. Here, for simplicity, we assume $g\in H_p^\delta(\bR^d)$. We show that if $supp(g) \subset \overline{\cO}^c$, then we have $r_{\overline{\cO}^c}g \in \psi^{\alpha/2} H_{p,\theta,\sigma}^\lambda(\overline{\cO}^c)$ for any $\lambda\in\bZ$ satisfying $\lambda<0\wedge \delta$. 
  Take $\eta\in C^\infty(\bR^d)$ such that $\eta=1$ on $supp(g)$, and $supp(\eta)\subset \overline{\cO}^c$.
  By \eqref{21.10.06.0918}, for any $\phi\in C_c^\infty(\overline{\cO}^c)$,
  \begin{align*}
    &\sum_{n=0}^{-\lambda} \int_{\bR^d} |D_x^n (e_{\overline{\cO}^c}\phi\eta)|^{p'} dx 
    \\
    &\leq N(supp(g)) \sum_{n=0}^{-\lambda} \int_{supp(g)} |D_x^n (\phi\eta)|^{p'} d_x^{\alpha p'/2 + \theta'-d}(1+d_x)^{-\alpha p'/2 - \theta' +d} dx 
    \\
    &\leq N \|\psi^{\alpha/2}\phi\|_{H_{p',\theta',\sigma'}^{-\lambda}(\overline{\cO}^c)}^{p'},
  \end{align*}
  where $1/p+1/p'=1, \theta/p+\theta'/p' = d$, and $\sigma/p+\sigma'/p'=0$.
  Therefore, we have $\phi\eta \in H_{p'}^{-\lambda}$ and
  \begin{align*}
    |(r_{\overline{\cO}^c}g,\phi)_{\overline{\cO}^c}| &= |(g,e_{\overline{\cO}^c}\phi\eta)_{\bR^d}| \leq \|g\|_{H_p^\delta} \|e_{\overline{\cO}^c}\phi\eta\|_{H_{p'}^{-\delta}} \leq N\|g\|_{H_p^\delta} \|\psi^{\alpha/2}\phi\|_{H_{p',\theta',\sigma'}^{-\lambda}}.
  \end{align*}
  By Lemma \ref{lem space}$(iv)$, we have $r_{\overline{\cO}^c}g \in H_{p,\theta,\sigma}^\lambda(\overline{\cO}^c)$ with $\sigma=\alpha p/2-\theta+d$.

  Thus, we obtain the solvability of solution to elliptic equation \eqref{elliptic def} with $r_{\overline{\cO}^c}g$. Moreover, due to Proposition \ref{Holder ellip}, if $f=0$, then the solution is infinitely differentiable in $\cO$.

  $(ii)$ 
  Next, we consider distributions having compact support in $\overline{\cO}^c$.
  Let $c_1,c_2>0$ and $g\in H_p^\lambda(\bR^d)$ such that $supp(g)\subset \{x\in \overline{\cO}^c: c_1<d_x<c_2\}$. Then, $r_{\overline{\cO}^c}g\in H_{p,\theta,\sigma}^\lambda(\overline{\cO}^c)$ for any $\theta,\sigma\in\bR$. Indeed, under our assumption, there exist $n_1,n_2\in\bZ$ such that for $n>n_2$ or $n<n_1$,
  \begin{equation*}
    g(e^n\cdot)\zeta_{-n}(e^n\cdot)=0.
  \end{equation*}
  Hence, \eqref{def sobolev} easily yields $r_{\overline{\cO}^c}g\in H_{p,\theta,\sigma}^\lambda(\overline{\cO}^c)$ for any $\theta,\sigma\in\bR$.

  It is well known that $\delta_{x_0}\in H_p^\lambda(\bR^d)$ if $\lambda p<(1-p)d$ (see e.g. \cite[Exercise 13.3.23]{KryLec}). Thus, as a candidate for an exterior data, we can consider $D^{\beta}_x\delta_{x_0}$ for any multi-index $\beta$.
\end{exam}

For the rest of this section we derive some  regularity results of solutions on $\bR^d$.  For $\beta\in \bR$,
define the space
$$
H_p^{\beta}(\fO) := \{ u: u=U \text{ on } \fO \text{ for some } U\in H_p^{\beta}(\bR^d) \}
$$
with the norm  given by
$$
\|u\|_{H_p^{\beta}(\fO)}:=\inf \{\|U\|_{H_p^\beta}: u=r_{\fO}U,  \,U\in H_p^\beta(\bR^d)\}.
$$
Similarly, we define
$$
\tilde{H}_p^{\beta}(\fO) := \{ u: u=r_{\fO}U \text{ for some } U\in H_p^{\beta}(\bR^d), \, supp\, U \subset \overline{\fO} \},
$$
and
$$
\|u\|_{\tilde{H}_p^{\beta}(\fO)}:=\inf \{\|U\|_{H_p^\beta}: u=r_{\fO}U,  \,U\in H_p^\beta(\bR^d), \, supp\, U \subset \overline{\fO}\}.
$$

The following result is proved in Section 3. 

\begin{lem}
\label{lem 3.29}
Let $\cO$ be bounded.

(i) Let $\kappa \geq 0$,  then we have continuous embeddings
$$
H^{\kappa}_{p,d-\kappa p}(\cO) \subset \tilde{H}^{\kappa}_p(\cO)\subset H^{\kappa}_p(\cO),
$$
and
$$
H^{-\kappa}_p(\cO) \subset  H^{-\kappa}_{p,d+\kappa p}(\cO).
$$
Moreover, $H^{\kappa}_{p,d-\kappa p}(\cO) =\tilde{H}^{\kappa}_p(\cO)=H^{\kappa}_p(\cO)$ if  $\kappa<1/p$, and $H^{-\kappa}_p(\cO) =H^{-\kappa}_{p,d+\kappa p}(\cO)$ if $\kappa<1/{p'}$. Here, $p'$ is the H\"older conjugate of $p$.

(ii) Let $0\leq \beta\leq 2$. Then, for any $g\in \tilde{H}^{\beta}_{p}(\overline{\cO}^c)$,
$$
\|\psi^{-\alpha/2}g\|_{L_{p,d-\beta p+\alpha p/2,\beta p}(\overline{\cO}^c)} \approx \|\psi^{-\beta}(1+\psi)^{\beta}g\|_{L_p(\overline{\cO}^c)}\leq N \|g\|_{\tilde{H}^{\beta}_p(\overline{\cO}^c)}.
$$

(iii)  Let $\beta\geq 0$ and
\begin{equation}
  \label{eqn 3.29.1}
\frac{\alpha}{2}+ \frac{1}{p}-1<\beta < \frac{\alpha}{2}+ \frac{1}{p},
\end{equation}
 then for any $u$ such that $r_{\cO}u\in H^{\beta}_{p,d-\beta p}(\cO)$ and $r_{\overline{\cO}^c}u \in \tilde{H}^{\beta}_p(\overline{\cO}^c)$, we have
\begin{equation}
\label{eqn 4.2.9}
\|u\|_{H^{\beta}_p(\bR^d)}\leq N \left(\|r_{\cO}u\|_{H^{\beta}_{p,d-\beta p}(\cO)}+\|r_{\overline{\cO}^c}u\|_{H^{\beta}_p(\overline{\cO}^c)} \right).
\end{equation}
\end{lem}

\begin{rem}
\label{rem 4.1}
If \eqref{eqn 3.29.1} holds, then $\beta-\alpha<1/p$, and therefore by Lemma \ref{lem 3.29}$(i)$, $H^{\beta-\alpha}_p(\cO) \subset H^{\beta-\alpha}_{p, d-(\beta-\alpha)p}(\cO)$.
\end{rem}

\begin{corollary}(Regularity of the elliptic equation on $\bR^d$)
 \label{cor 3.29}
Let $\cO$ be bounded, $\beta \geq 0$ and \eqref{eqn 3.29.1} hold. Then, for any $f \in H^{\beta-\alpha}_{p, d-(\beta-\alpha)p}(\cO)$ and 
$g \in \tilde{H}^{\beta}_{p}(\overline{\cO}^c)$, 
equation \eqref{elliptic def} has a unique solution $u\in\cD'(\bR^d)$ such that $r_{\cO}u\in  H_{p,d-\beta p}^{\beta}(\cO)$, and for this solution  we have
\begin{equation} \label{cor 3.39.2}
\|u\|_{H^{\beta}_p(\bR^d)}  \leq N \left( \|f\|_{H_{p, d-(\beta-\alpha)p}^{\beta-\alpha}(\cO)} + \|g\|_{\tilde{H}^{\beta}_{p}(\overline{\cO}^c)}\right),
\end{equation}
where $N=N(d,p,\alpha,\beta,\cO)$.
\end{corollary}

\begin{proof}
If \eqref{eqn 3.29.1} holds, then $\theta_0:=d-\beta p+\alpha p/2 \in (d-1, d-1+p)$.  Note that $\psi^{\alpha/2}f\in H^{\beta-\alpha}_{p,\theta_0}(\cO)$ if and only if $ f\in H^{\beta-\alpha}_{p,d-(\beta-\alpha)p}(\cO)$. Also, by Lemma \ref{lem 3.29}$(ii)$,
$$
\|\psi^{-\alpha/2}g\|_{L_{p,\theta_0,\beta p}(\overline{\cO}^c)}\leq N \|g\|_{\tilde{H}^{\beta}_p(\overline{\cO}^c)}<\infty.
$$
Thus, applying 
 Theorem  \ref{main thm ellip} with $\gamma=\beta-\alpha$, $\theta=\theta_0$ and $\sigma=\beta p$,   we have a unique solution $u\in\cD'(\bR^d)$ such that $u|_{\cO}\in H_{p,d-\beta p}^{\beta}(\cO)$. Furthermore, by 
 \eqref{eqn 4.2.9}  and Theorem \ref{main thm ellip},
 \begin{align*}
 \|u\|_{H^{\beta}_p(\bR^d)} &\leq N \left(\|r_{\cO}u\|_{H^{\beta}_{p,d-\beta p}(\cO)}+\|g\|_{\tilde{H}^{\beta}_p(\overline{\cO}^c)} \right) 
 \\
 &\leq N \left (\|f\|_{H_{p, d-(\beta-\alpha)p}^{\beta-\alpha}(\cO)}  +\|\psi^{-\alpha/2}g\|_{L_{p,\theta_0,\beta p}(\overline{\cO}^c)} 
 +\|g\|_{\tilde{H}^{\beta}_p(\overline{\cO}^c)} \right)
 \\
 &\leq N \left( \|f\|_{H_{p, d-(\beta-\alpha)p}^{\beta-\alpha}(\cO)} + \|g\|_{\tilde{H}^{\beta}_{p}(\overline{\cO}^c)}\right).
 \end{align*}
 The corollary is proved.
 \end{proof}

\begin{rem}
A version of \eqref{cor 3.39.2} is introduced in \cite{Grubb nonlocal}. More precisely, in  \cite{Grubb nonlocal}  it is proved that if $\cO$ is a bounded $C^{\infty}$ domain and \eqref{eqn 3.29.1} holds, then
$$
\|u\|_{H^{\beta}_p(\bR^d)}  \leq N \left( \|f\|_{H_{p}^{\beta-\alpha}(\cO)} + \|g\|_{H^{\beta}_{p}(\overline{\cO}^c)}\right),
$$
provided that $f\in H^{\beta-\alpha}_p(\cO)$ and $g\in H^{\beta}_p(\overline{\cO}^c)$.  By Remark \ref{rem 4.1}, 
 our condition on $f$ (resp. $g$) is  slightly weaker (resp. stronger) than the one in \cite{Grubb nonlocal}.  

\end{rem}

Theorem \ref{main thm para} and the argument in the proof of Corollary \ref{cor 3.29} yield the following. We assume $u(0,\cdot)=0$ for simplicity.

\begin{corollary}(Regularity of the parabolic equation on $\bR^d$)
Let $\cO$ be bounded, $\beta \geq 0$ and \eqref{eqn 3.29.1} hold. Then, for $f \in \bH^{\beta-\alpha}_{p, d-(\beta-\alpha)p}(\cO,T)$ and 
$g \in \tilde{\bH}^{\beta}_{p}(\overline{\cO}^c,T):=L_p((0,T); \tilde{H}^{\beta}_p(\overline{\cO}^c))$, 
 equation \eqref{para def} given with zero initial data has a unique  solution $u$ such that $r_{(0,T)\times \cO}u\in \bH_{p,\theta_0-\alpha p/2}^{\beta}(\cO,T)$, and for this solution we have
\begin{equation*}
\|u\|_{L_p((0,T); H^{\beta}_p(\bR^d))}  \leq N \left( \|f\|_{\bH_{p, d-(\beta-\alpha)p}^{\beta-\alpha}(\cO,T)} + \|g\|_{\tilde{\bH}^{\beta}_{p}(\overline{\cO}^c,T)}\right),
\end{equation*}
where $N=N(d,p,\alpha,\beta,\cO)$.
\end{corollary}

\section{Some properties of function spaces} \label{sec prop}

In this section, we introduce some properties of function spaces and  provide the proofs of Lemmas \ref{thm frac def} and \ref{lem 3.29}. Throughout this section, $\fO$ denotes either $\cO$ or $\overline{\cO}^c$. 

The following result  is taken from  \cite{Lototsky} if $\sigma=0$, and the  case $\sigma\neq 0$ also can be  obtained by repeating the proof of \cite[Theorem 4.1]{Lototsky}.

\begin{lem} \label{lem deriv}
The following conditions are equivalent:

(i) $u\in H_{p,\theta,\sigma}^\gamma(\fO)$

(ii) $u\in H_{p,\theta,\sigma}^{\gamma-1}(\fO)$ and $\psi D_x u\in H_{p,\theta,\sigma}^{\gamma-1}(\fO)$

(iii) $u\in H_{p,\theta,\sigma}^{\gamma-1}(\fO)$ and $D_x (\psi u) \in H_{p,\theta,\sigma}^{\gamma-1}(\fO)$.
\vspace{1mm}

Under any of these three conditions, we have
\begin{eqnarray*}
&\|u\|_{H_{p,\theta,\sigma}^\gamma(\fO)} \approx \left(\|u\|_{H_{p,\theta,\sigma}^{\gamma-1}(\fO)}+\|\psi D_x u\|_{H_{p,\theta,\sigma}^{\gamma-1}(\fO)}\right)
\\
&\approx \left(\|u\|_{H_{p,\theta,\sigma}^{\gamma-1}(\fO)}+\|D_x (\psi u)\|_{H_{p,\theta,\sigma}^{\gamma-1}(\fO)}\right).
\end{eqnarray*}
Furthermore, all the claims   hold  with  $B_{p,\theta}^\gamma(\fO)$. 
\end{lem}

The following result is a version of \cite[Lemma A.5]{Seo}, which will be used to handle negative order regularity of solutions.

\begin{lem} \label{negative rep}
There exist operators $\Lambda_0, \Lambda_2, \cdots, \Lambda_d$ such that

 (i) $\Lambda_0 : H_{p,\theta,\sigma}^{\gamma}(\fO)\to H_{p,\theta,\sigma}^{\gamma+1}(\fO)$ and $\Lambda_i : H_{p,\theta,\sigma}^{\gamma}(\fO)\to H_{p,\theta-p,\sigma}^{\gamma+1}(\fO)$ $(i=1,\cdots,d)$ for any  $\gamma,\theta,\sigma \in \bR$ and $p\in(1,\infty)$;
 
 (ii)   for any $u\in H_{p,\theta,\sigma}^{\gamma}(\fO)$,
\begin{align} \label{eq. 1003-2}
u=\Lambda_0 u + \sum_{i=1}^d D_i \Lambda_i u
\end{align}
and
\begin{align*}
\|u\|_{H_{p,\theta,\sigma}^{\gamma}(\fO)} \approx \left( \|\Lambda_0 u\|_{H_{p,\theta,\sigma}^{\gamma+1}(\fO)} + \sum_{i=1}^d \|\Lambda_i u\|_{H_{p,\theta-p,\sigma}^{\gamma+1}(\fO)} \right).
\end{align*}
Moreover, the above assertions  hold true with $B_{p,\theta}^\gamma(\fO)$. 
\end{lem}

\begin{proof}
The proofs for $B_{p,\theta}^{\gamma}(\fO)$ are similar to those for $H_{p,\theta,\sigma}^{\gamma}(\fO)$, and we only prove the claims for $H_{p,\theta,\sigma}^{\gamma}(\fO)$.

Consider the operators $L_0,L_1,\dots,L_d$ on $H_p^\gamma$ defined as $L_0:=(1-\Delta)^{-1}$ and $L_i:=-D_i(1-\Delta)^{-1}$ $(i=1,\dots,d)$.
Then, $L_0 + D_i L_i$ is the identity map on $H_p^\gamma$,  and for any $v\in H_p^\gamma$,
\begin{align} \label{eq. 1003}
\|v\|_{H_p^\gamma} =\|L_0v\|_{H^{\gamma+2}_p} \approx  \left( \|L_0v\|_{H_p^{\gamma+1}} + \sum_{i=1}^d \|L_i v\|_{H_p^{\gamma+1}} \right).
\end{align}

Take a collection $\{\zeta_n\in C^{\infty}_c(\fO): n\in \bZ\}$  satisfying \eqref{zeta prop 1}-\eqref{zeta prop 3} and $\sum_{n\in\bZ} \zeta_n(x) =1$ in $\fO$.
 We also take $\{\eta_n \in C^{\infty}_c(\fO): n\in \bZ\}$ satisfying \eqref{zeta prop 1}-\eqref{zeta prop 3} (possibly with different $(k_1,k_2)$) and $\eta_n=1$ on the support of $\zeta_n$. Consequently, $\zeta_n \eta_n=\zeta_n$. 
 
 For $u\in H_{p,\theta,\sigma}^\gamma(\fO)$, define
\begin{align} \label{ineq 1105-1}
\Lambda_0 u(x) :=& \sum_{m\in\bZ} \eta_{-m}(x) L_0[\zeta_{-m}(e^m\cdot)u(e^m\cdot)](e^{-m}x) \nonumber
\\
&-\sum_{i=1}^d \sum_{m\in\bZ} e^m (D_i \eta_{-m})(x) L_i[\zeta_{-m}(e^m\cdot)u(e^m\cdot)](e^{-m}x),
\end{align}
and
\begin{align} \label{ineq 1105-2}
\Lambda_i u(x) := \sum_{m\in\bZ} e^m \eta_{-m}(x) L_i[\zeta_{-m}(e^m\cdot)u(e^m\cdot)](e^{-m}x), \quad i=1,\dots,d.
\end{align}
We first show that these operators are well defined and
\begin{align} \label{ineq. 1003-1}
\|\Lambda_0 u\|_{H_{p,\theta,\sigma}^{\gamma+1}(\fO)} + \sum_{i=1}^d \|\Lambda_i u\|_{H_{p,\theta-p,\sigma}^{\gamma+1}(\fO)} \leq N \|u\|_{H_{p,\theta,\sigma}^{\gamma}(\fO)}.
\end{align}
To show this, due to the similarity, we only control the term
$$
\sum_{m\in\bZ} \eta_{-m}(x) L_0[\zeta_{-m}(e^m\cdot)u(e^m\cdot)](e^{-m}x).
$$
For each $n\in \bZ$, denote $\Gamma(n):=\{m\in\bZ : \zeta_{n}(x)\eta_{m}(x)\neq0 \text{ for some } x\in \fO\}$.
Then, the number of elements in $\Gamma(n)$ is bounded by some constant $l\in\bN$, which is independent of $n$.
Thus, by \eqref{zeta prop 2} and the result on a pointwise multiplier (see e.g. \cite[Lemma 5.2]{kry99analytic}),
\begin{align} \label{ineq. 1007}
& \bigg\|\zeta_{-n}(e^nx)\sum_{m\in\bZ} \eta_{-m}(e^nx) L_0[\zeta_{-m}(e^m\cdot)u(e^m\cdot)](e^{n-m}x)\bigg\|_{H_p^\gamma} \nonumber
\\
&\leq \bigg\| \zeta_{-n}(e^nx)\sum_{m=n-l}^{n+l} \eta_{-m}(e^nx) L_0[\zeta_{-m}(e^m\cdot)u(e^m\cdot)](e^{n-m}x) \bigg\|_{H_p^\gamma} \nonumber
\\
&\leq N \sum_{m=n-l}^{n+l} \|L_0[\zeta_{-m}(e^m\cdot)u(e^m\cdot)](e^{n-m}x)\|_{H_p^{\gamma+1}}.
\end{align}
Then, by the result on change of variables in $H_p^\gamma$ (see. e.g \cite[Theorem 4.3.2]{Triebel2}),
\begin{align*}
&\sum_{m=n-l}^{n+l} \|L_0[\zeta_{-m}(e^m\cdot)u(e^m\cdot)](e^{n-m}x)\|_{H_p^{\gamma+1}} 
\\
&\leq N \sum_{m=n-l}^{n+l} \|L_0[\zeta_{-m}(e^m\cdot)u(e^m\cdot)]\|_{H_p^{\gamma+1}}
\\
&\leq N \sum_{m=n-l}^{n+l} \|\zeta_{-m}(e^m\cdot)u(e^m\cdot)\|_{H_p^{\gamma}}.
\end{align*}
For the last inequality we used \eqref{eq. 1003}. Therefore, this and \eqref{ineq. 1007} yield
\begin{align*}
&\bigg\| \sum_{m\in\bZ} \eta_{-m}(x) L_0[\zeta_{-m}(e^m\cdot)u(e^m\cdot)](e^{-m}x) \bigg\|^p_{H^{\gamma+1}_{p,\theta,\sigma}(\fO)}
\\
&= \sum_{n\in\bZ} e^{n\theta} (1+e^n)^\sigma \bigg\|\zeta_{-n}(e^nx)\sum_{m\in\bZ} \eta_{-m}(e^nx) L_0[\zeta_{-m}(e^m\cdot)u(e^m\cdot)](e^{n-m}x)\bigg\|_{H_p^{\gamma+1}}^p
\\
&\leq N \sum_{n\in\bZ} e^{n\theta} (1+e^n)^\sigma \sum_{m=n-l}^{n+l} \|\zeta_{-m}(e^m\cdot)u(e^m\cdot)\|_{H_p^{\gamma}}^p
\leq N \|u\|^p_{H_{p,\theta,\sigma}^{\gamma}(\fO)}.
\end{align*}

Next, we show \eqref{eq. 1003-2}. Since $L_0 + D_i L_i$ is the identity map on $H_p^\gamma$, 
\begin{align*} 
\Lambda_0 u(x) + \sum_{i=1}^d D_i \Lambda_i u(x) &= \sum_{m\in\bZ} \eta_{-m}(x) (L_0 + D_i L_i)[\zeta_{-m}u(e^m\cdot)](e^{-m}x)
\\
&= \sum_{m\in\bZ} \eta_{-m} (x) \zeta_{-m} (x) u(x) = \sum_{m\in\bZ} \zeta_{-m} (x) u(x) = u(x).
\end{align*}
The reverse inequality of \eqref{ineq. 1003-1} can be easily proved  by using identity \eqref{eq. 1003-2} and Lemma \ref{lem deriv}.   Therefore, the lemma is proved.
\end{proof}

Next, we present some auxiliary results regarding the operator $\Delta^{\alpha/2}$ on $\fO$. 
 Note that, for a function $\phi\in C^{\infty}_c(\fO)$, $\Delta^{\alpha/2}(e_{\fO}\phi)$, defined as in \eqref{def frac}, belongs to $L_\infty(\bR^d)$, and thus $r_{\overline{\fO}^c}\Delta^{\alpha/2}(e_{\fO}\phi) \in \cD'(\overline{\fO}^c)$.

\begin{lem} \label{lem out}
Let $d-1-\alpha p/2<\theta<d-1+p+\alpha p/2$ and $\gamma,\lambda\in\bR$. 

(i)  If $\cO$ is bounded and $\sigma<-\theta+dp+\alpha p/2$, then for any $\phi\in C_c^\infty(\cO)$,
\begin{align} \label{ineq 1030}
\|\psi^{\alpha/2}r_{\overline{\cO}^c}\Delta^{\alpha/2}(e_{\cO}\phi)\|_{H_{p,\theta,\sigma}^\lambda(\overline{\cO}^c)} \leq N \|\psi^{-\alpha/2}\phi\|_{H_{p,\theta}^\gamma(\cO)},
\end{align}
where $N=N(d,p,\gamma,\alpha,\theta,\cO)$.

(ii) If $\cO$ is  bounded  and $\sigma>-\theta-\alpha p/2$, then for any $\phi\in C_c^\infty(\overline{\cO}^c)$,
\begin{align*}
\|\psi^{\alpha/2}r_{\cO}\Delta^{\alpha/2}(e_{\overline{\cO}^c}\phi)\|_{H_{p,\theta}^\gamma(\cO)} \leq N \|\psi^{-\alpha/2}\phi\|_{H_{p,\theta,\sigma}^\lambda(\overline{\cO}^c)},
\end{align*}
where $N=N(d,p,\gamma,\alpha,\theta,\cO)$.

(iii) If $\cO$ is a half space,  then inequality \eqref{ineq 1030} holds provided that  $\sigma=0$.
\end{lem}

\begin{proof}
$(i)$ It suffices to prove \eqref{ineq 1030} only for  $\lambda=m\in\bN_+$. 

We consider three cases according to the range of $\gamma$.

\textbf{1.} $\gamma\geq0$. It is enough to prove \eqref{ineq 1030} with $\gamma=0$. 

Let $\phi \in C_c^\infty(\cO)$. Then, by \eqref{def frac}, for $x\in \overline{\cO}^c$ and $1\leq k\leq m$,
\begin{align*}
\left|D^k_x\Delta^{\alpha/2}(e_{\cO}\phi)(x)\right|=\left|c_d D^k_x\left(\int_{\cO} \frac{\phi(y)}{|x-y|^{d+\alpha}}dy\right)\right| \leq N \int_{\cO} \frac{|\phi(y)|}{|x-y|^{d+\alpha+k}}dy.
\end{align*}
Since $d_x\leq |x-y|$ for $x\in \overline{\cO}^c$ and $y\in \cO$, by Lemma \ref{lem deriv},
\begin{align} \label{23.03.27-1}
\|\psi^{\alpha/2}r_{\overline{\cO}^c}\Delta^{\alpha/2}(e_{\cO}\phi)\|_{H_{p,\theta,\sigma}^m(\overline{\cO}^c)}^p &\leq N \sum_{k=0}^m \|\psi^{\alpha/2+k} D_x^k \left(r_{\overline{\cO}^c}\Delta^{\alpha/2}(e_{\cO}\phi)\right)\|_{L_{p,\theta,\sigma}(\overline{\cO}^c)}^p \nonumber
\\
&\leq N \| \psi^{\alpha/2} H \|_{L_{p,\theta,\sigma}(\overline{\cO}^c)}^p,
\end{align}
where
$$
H(x):=\int_{\cO} \frac{|\phi(y)|}{|x-y|^{d+\alpha}}dy.
$$
Thus, we only need to  prove
\begin{align} \label{ineq 1025}
\| \psi^{\alpha/2} H \|_{L_{p,\theta,\sigma}(\overline{\cO}^c)}^p \leq N \|\psi^{-\alpha/2}\phi\|_{H_{p,\theta}^\gamma(\cO)}.
\end{align}
Take a constant $\upsilon$ such that $(-\alpha p/2)\vee\left(-\theta+d-\alpha p/2\right)<\upsilon p <1\wedge\left(-\theta+d+p-1\right)$.
By H\"older's inequality, for $x\in \overline{\cO}^c$,
\begin{align} \label{ineq 1016-1}
H(x) &\leq N \left(\int_{\cO} \frac{|\phi(y)|^p d_y^{\theta-d+\upsilon p}}{|x-y|^{d+\alpha p/2}}dy\right)^{1/p} \left(\int_{\cO} \frac{d_y^{\theta'-d-\upsilon p'}}{|x-y|^{d+\alpha p'/2}}dy \right)^{1/p'} \nonumber
\\
&=:NI(x)II(x),
\end{align}
where $1/p+1/p'=1, \theta/p+\theta'/p' = d$, and $\sigma/p+\sigma'/p'=0$.
Due to \eqref{aux ineq 2} with $\nu_0=\theta'-d-\upsilon p'$ and $\nu_1=\alpha p'/2$,
$$
II(x)^{p'} \leq N d_x^{\theta'-d-\upsilon p'-\alpha p'/2}(1+d_x)^{-\theta'+\upsilon p'}.
$$
Thus, by \eqref{ineq 1016-1} and Fubini's theorem, 
\begin{align*}
\| \psi^{\alpha/2} H \|_{L_{p,\theta,\sigma}(\overline{\cO}^c)}^p &\leq N \int_{\cO} \int_{\overline{\cO}^c} |\phi(y)|^p d_y^{\theta-d+\upsilon p} \frac{d_x^{-\upsilon p}(1+d_x)^{\sigma+\theta-dp+\upsilon p}}{|x-y|^{d+\alpha p/2}}  dxdy.
\end{align*}
Therefore, to prove \eqref{ineq 1025}, it suffices to show
\begin{align} \label{ineq 221017-1}
\int_{\overline{\cO}^c} \frac{d_x^{-\upsilon p}(1+d_x)^{\sigma+\theta-dp+\upsilon p}}{|x-y|^{d+\alpha p/2}} dx \leq N d_y^{-\upsilon p-\alpha p/2}.
\end{align}
If $d_x\leq 1$, then $1+d_x\approx 1$. Hence, by \eqref{aux ineq 1} with $\nu_0=-\upsilon p$ and $\nu_1=\alpha p/2$,
\begin{align*} 
\int_{\{x\in \overline{\cO}^c : d_x \leq1\}} \frac{d_x^{-\upsilon p}(1+d_x)^{\sigma+\theta-dp+\upsilon p}}{|x-y|^{d+\alpha p/2}} dx &\leq N \int_{\overline{\cO}^c} \frac{d_x^{-\upsilon p}}{|x-y|^{d+\alpha p/2}} dx 
\\
&\leq N d_y^{-\upsilon p-\alpha p/2}.
\end{align*}
Since $1+d_x\approx d_x\approx |x-y|$ for $y\in \cO$ and $x\in \overline{\cO}^c$ satisfying $d_x\geq 1$,
\begin{align*} 
&\int_{\{x\in \overline{\cO}^c:d_x\geq 1\}} \frac{d_x^{-\upsilon p}(1+d_x)^{\sigma+\theta-dp+\upsilon p}}{|x-y|^{d+\alpha p/2}}  dx 
\\
&\leq N \int_{|x-y|\geq 1} |x-y|^{\sigma+\theta-dp-d-\alpha p/2} dx \leq N \leq N d_y^{-\upsilon p-\alpha p/2}.
\end{align*}
Here, the last inequality follows from $-\upsilon p-\alpha p/2<0$ and the fact that $\cO$ is bounded. Thus, we have \eqref{ineq 221017-1}, and thus $(i)$ is proved if $\gamma \geq 0$.

\textbf{2.} $-1\leq\gamma<0$. In the case, it suffices to assume $\gamma=-1$.
By Lemma \ref{negative rep}, $\phi$ can be expressed as $\phi=\Lambda_0 \phi + \sum_{i=1}^d D_i \Lambda_i \phi$,
and we have
\begin{align*} 
\|\phi\|_{H_{p,\theta}^{\gamma}(\cO)} \approx \left( \|\Lambda_0 \phi\|_{H_{p,\theta}^{\gamma+1}(\cO)} + \sum_{i=1}^d \|\Lambda_i \phi\|_{H_{p,\theta-p}^{\gamma+1}(\cO)} \right).
\end{align*}
Here, due to \eqref{ineq 1105-1} and \eqref{ineq 1105-2}, we have $\Lambda_i \phi\in C_c^\infty(\cO)$. Thus, by \eqref{def frac}, for $x\in \overline{\cO}^c$  and $1\leq k\leq m$,
\begin{align*}
\left|D^k_x \left(\Delta^{\alpha/2}(e_{\cO}\Lambda_0\phi)\right)(x)\right|=\left|c_d D^k_x\left(\int_{\cO} \frac{\Lambda_0\phi(y)}{|x-y|^{d+\alpha}}dy\right)\right| \leq N \int_{\cO} \frac{|\Lambda_0\phi(y)|}{|x-y|^{d+\alpha+k}}dy,
\end{align*}
and for $i=1,2,\cdots, d$,
\begin{align*}
&\left|D^k_x\left(\Delta^{\alpha/2}(e_{\cO}D_i\Lambda_i\phi)\right)(x)\right|=\left|c_d D^k_x\left(\int_{\cO} \frac{D_i\Lambda_i\phi(y)}{|x-y|^{d+\alpha}}dy\right)\right| 
\\
&\leq N \int_{\cO} \frac{|\Lambda_i\phi(y)|}{|x-y|^{d+\alpha+k+1}}dy
\leq N \int_{\cO} \frac{|(\psi(y))^{-1}\Lambda_i\phi(y)|}{|x-y|^{d+\alpha+k}}dy.
\end{align*}
For the first inequality above we used the integration by parts, and for the second  inequality  we used $\psi(y)\approx d_y\leq|x-y|$.
  Thus, as in \eqref{23.03.27-1}, we have
  \begin{align*}
    \|\psi^{\alpha/2} r_{\overline{\cO}^c}\Delta^{\alpha/2}(e_{\cO}\phi)\|_{H_{p,\theta,\sigma}^m(\overline{\cO}^c)}^p &\leq N \sum_{k=0}^m \|\psi^{\alpha/2+k} D_x^k \left(r_{\overline{\cO}^c}\Delta^{\alpha/2}(e_{\cO}\phi)\right)\|_{L_{p,\theta,\sigma}(\overline{\cO}^c)}^p 
    \\
    &\leq N \| \psi^{\alpha/2} \widetilde{H} \|_{L_{p,\theta,\sigma}(\overline{\cO}^c)}^p,
  \end{align*}
  where
  \begin{equation*}
    \widetilde{H}(x):=\int_{\cO} \frac{|\Lambda_0\phi(y)|}{|x-y|^{d+\alpha}}dy + \sum_{i=1}^d\int_{\cO} \frac{|(\psi(y))^{-1}\Lambda_i\phi(y)|}{|x-y|^{d+\alpha}}dy.
  \end{equation*}
 Since $\gamma+1=0$, one can apply \eqref{ineq 1025} to both $\Lambda_0\phi$ and $\psi^{-1}\Lambda_i\phi$ (in place of $\phi$), and get
\begin{align*}
\|\psi^{\alpha/2}\Delta^{\alpha/2}\phi\|_{H_{p,\theta,\sigma}^\lambda(\overline{\cO}^c)} &\leq N \left( \|\Lambda_0 \phi\|_{H_{p,\theta}^{\gamma+1}(\cO)} + \sum_{i=1}^d \|\psi^{-1}\Lambda_i \phi\|_{H_{p,\theta}^{\gamma+1}(\cO)} \right) 
\\
&\approx N \|\phi\|_{H_{p,\theta}^{\gamma}(\cO)}.
\end{align*}
Therefore, we have \eqref{ineq 1030} for $-1\leq\gamma<0$.

\textbf{3.} $\gamma<-1$.  First assume  $-2\leq \gamma<-1$.  Then, by applying Lemma \ref{negative rep} repeatedly, one can further express $\phi$ in the form of $\phi=\Lambda_{0}^1 \phi + \sum_{i=1}^d D_i \Lambda_{i}^1 \phi + \sum_{i,j=1}^d D_{i}D_j \Lambda_{ij}^1 \phi$,
together with the relation
\begin{align*} 
\|\phi\|_{H_{p,\theta}^{\gamma}(\cO)} \approx \left( \|\Lambda_{0}^1 \phi\|_{H_{p,\theta,\sigma}^{\gamma+2}(\cO)} + \sum_{i=1}^d \|\Lambda_{i}^1 \phi\|_{H_{p,\theta-p,\sigma}^{\gamma+2}(\cO)} + \sum_{i,j=1}^d \|\Lambda_{ij}^1 \phi\|_{H_{p,\theta-2p,\sigma}^{\gamma+2}(\cO)} \right).
\end{align*}
By repeating the argument used in \textbf{2},  one can easily handle the case $-2\leq \gamma<-1$ and   continue in the same way to complete the proof for the case $\gamma<-1$.

$(ii)$ and $(iii)$ The similar argument used in the proof of $(i)$ yields the desired result. 
The lemma is proved.
\end{proof}

{\textbf{Proof of Lemma \ref{thm frac def}}

\begin{proof}

$(i)$ Let $\phi \in C_c^\infty(\bR^d)$ and take a compact set $K$ satisfying $supp\phi\subset K$. 
First, note that
\begin{equation} \label{21.09.20.1850}
  \int_{A} d_x^{\nu} <\infty.
\end{equation}
for any $\nu>-1$ and bounded set $A$ (see e.g. \cite[Lemma A.4]{Dirichlet}). Thus, by Lemma \ref{lem deriv}, for any $m\in \bN_+$,
\begin{align}
&\|\psi^{\alpha/2} r_{\cO}\phi\|_{H_{p',\theta'}^m (\cO)}^{p'} + \|\psi^{\alpha/2} r_{\overline{\cO}^c}\phi\|_{H_{p',\theta',\sigma}^m (\overline{\cO}^c)}^{p'} \nonumber
\\
&\leq N \sum_{i=0}^m \left(\|\psi^{\alpha/2+i}D_x^i(r_{\cO}\phi)\|_{L_{p',\theta'} (\cO)}^{p'} + \|\psi^{\alpha/2+i}D_x^i(r_{\overline{\cO}^c}\phi)\|_{L_{p',\theta',\sigma} (\overline{\cO}^c)}^{p'}\right) \nonumber
\\
&\leq N \sum_{i=0}^m \|D_x^i\phi\|_{L_\infty(\bR^d)}^{p'} \int_K d_x^{\alpha p'/2 +ip' +\theta'-d} dx  \leq N(K) \sum_{i=0}^m \|D_x^i\phi\|_{L_\infty(\bR^d)}^{p'}, \label{eqn 3.18}
\end{align}
where $1/p+1/p'=1, \theta/p+\theta'/p' = d$, and $\sigma/p+\sigma'/p'=0$.
Thus, we have $r_{\cO}\phi \in (\psi^{\alpha/2}H_{p,\theta}^{\gamma+\alpha}(\cO))^*$ and $r_{\overline{\cO}^c}\phi\in(\psi^{\alpha/2}H_{p,\theta,\sigma}^\lambda(\overline{\cO}^c))^*$.
Furthermore, thanks to \cite[Theorem 6.5(f)]{Rudin} and \eqref{eqn 3.18}, we also have $\phi\in \cD'(\bR^d)$. Therefore $(i)$ is proved.

$(ii)$ This is a consequence of \eqref{ineq 1030} and \cite[Lemma 4.4]{Dirichlet}.

$(iii)$ Let  $\phi \in C_c^\infty(\cO)$. Take two sequences $v_n\in C_c^\infty(\cO)$ and $w_n\in C_c^\infty(\overline{\cO}^c)$ such that 
\begin{eqnarray} 
&&v_n \to r_{\cO}u \text{ in } \psi^{\alpha/2}H_{p,\theta}^{\gamma+\alpha}(\cO), \label{eq91310150}
\\
&&w_n \to r_{\overline{\cO}^c}u \text{ in } \psi^{\alpha/2}H_{p,\theta,\sigma}^\lambda(\overline{\cO}^c). \label{eq91310150-1}
\end{eqnarray}
Then, by \cite[Corollary 4.5$(i)$]{Dirichlet},
\begin{align*} 
|(r_{\cO}\Delta^{\alpha/2} (e_{\cO}v_n),\phi)_{\cO}| &= |(v_n,r_{\cO}\Delta^{\alpha/2}(e_{\cO}\phi))_{\cO}| 
\\
&\leq N \|\psi^{-\alpha/2}v_n\|_{H_{p,\theta}^{\gamma+\alpha}(\cO)}\|\psi^{\alpha/2}r_{\cO}\phi\|_{H_{p',\theta'}^{-\gamma}(\cO)}.
\end{align*}
By \eqref{ineq 1030} with $(-\gamma,-\lambda,p',\theta',\sigma')$ instead of $(\gamma,\lambda,p,\theta,\sigma)$,
\begin{align*} 
&|(r_{\overline{\cO}^c}\Delta^{\alpha/2} (e_{\overline{\cO}^c}w_n),\phi)_{\cO}| 
\\
&= |(w_n,r_{\overline{\cO}^c}\Delta^{\alpha/2}(e_{\cO}\phi)))_{\overline{\cO}^c}| 
\\
&\leq N \| \psi^{-\alpha/2}w_n\|_{H_{p,\theta,\sigma}^\lambda(\overline{\cO}^c)} \| \psi^{\alpha/2}r_{\overline{\cO}^c}\Delta^{\alpha/2}(e_{\cO}\phi)\|_{H_{p',\theta',\sigma'}^{-\lambda}(\overline{\cO}^c)}  
\\
&\leq N \| \psi^{-\alpha/2}w_n\|_{H_{p,\theta,\sigma}^\lambda(\overline{\cO}^c)} \| \psi^{-\alpha/2}\phi\|_{H_{p',\theta'}^{-\gamma}(\cO)}.
\end{align*}
Thus, if we denote $u_n:=e_{\cO}v_n+e_{\overline{\cO}^c}w_n$, then both $(r_{\cO}u_n,r_{\cO}\Delta^{\alpha/2} (e_{\cO}\phi))_{\cO}$and $(r_{\overline{\cO}^c}u_n,r_{\overline{\cO}^c}\Delta^{\alpha/2}(e_{\cO}\phi))_{\overline{\cO}^c}$ are well defined and 
\begin{align*}
&|(\Delta^{\alpha/2}u_n,\phi)_{\cO}| 
\\
&\leq N \left( \|\psi^{-\alpha/2} r_{\cO}u_n\|_{H_{p,\theta}^{\gamma+\alpha}(\cO)} + \|\psi^{-\alpha/2} r_{\overline{\cO}^c}u_n\|_{H_{p,\theta,\sigma}^{\lambda}(\overline{\cO}^c)} \right) \| \psi^{-\alpha/2}\phi\|_{H_{p',\theta'}^{-\gamma}(\cO)}.
\end{align*}
Therefore, by the duality (Lemma \ref{lem space}$(iv)$), we get \eqref{ineq 1025-1} with $u_n$ in place of $u$. Letting $n\to\infty$, we get the desired result.
The lemma is proved.  
  \end{proof}

  In the following lemma, we take a collection of functions $\{\zeta_n\in C^{\infty}_c(\cO) : n\in \bZ\}$ satisfying \eqref{zeta prop 1}-\eqref{zeta prop 3} with  $(k_1,k_2)=(1,e^2)$.

\begin{lem} \label{lem perturb}
Let $d-1-\alpha p/2<\theta<d-1+p+\alpha p/2$, $\gamma\geq-\alpha/2$ and $\lambda\in\bR$. Suppose $\sigma>-\theta-\alpha p/2$ if $\cO$ is bounded, and $\sigma=0$ is if $\cO$ is a half space. Then, for any $u\in e_{\cO}H^{\gamma+\alpha}_{p,\theta-\alpha p/2}(\cO) \oplus e_{\overline{\cO}^c}H^{\lambda}_{p,\theta-\alpha p/2,\sigma}(\overline{\cO}^c)$, 
\begin{align} \label{ineq. 08.01-4}
&\sum_{n\in\bZ} e^{n(\theta-\alpha p/2)} \left\|\Delta^{\alpha/2}\Big(u(e^n\cdot)\zeta_{-n}(e^n\cdot)\Big)-\zeta_{-n}(e^n\cdot)\Delta^{\alpha/2}(u(e^n\cdot)) \right\|_{H_p^{\gamma}}^p \nonumber
\\
&\leq N \|\psi^{-\alpha /2}r_{\cO}u\|_{H_{p,\theta}^{0\vee(\gamma+\alpha/2)}(\cO)}^p + N \|\psi^{-\alpha/2} r_{\overline{\cO}^c}u\|_{H_{p,\theta,\sigma}^\lambda(\overline{\cO}^c)}^p,
\end{align}
where $N=N(d,p,\gamma,\alpha,\theta,\cO,\sigma)$. 
\end{lem}

\begin{proof}
We first note that  the left-hand side of \eqref{ineq. 08.01-4} makes sense due to Lemma \ref{thm frac def}. Take two sequences $v_m\in C_c^\infty(\cO)$ and $w_m\in C_c^\infty(\overline{\cO}^c)$ as in \eqref{eq91310150} and \eqref{eq91310150-1}. 
By \cite[Corollary 4.5$(ii)$]{Dirichlet}, we have \eqref{ineq. 08.01-4} with $v_m$ instead of $u$. Since $e_{\overline{\cO}^c}w_m\zeta_{-n}=0$ for all $n$ (note $\zeta_{-n}=0$ on $\overline{\cO}^c$),
\begin{eqnarray*}
&&\sum_{n\in\bZ} e^{n(\theta-\alpha p/2)} \left\|\Delta^{\alpha/2}\Big(e_{\overline{\cO}^c}w_m(e^n\cdot)\zeta_{-n}(e^n\cdot)\Big)-\zeta_{-n}(e^n\cdot)\Delta^{\alpha/2}(e_{\overline{\cO}^c}w_m(e^n\cdot)) \right\|_{H_p^{\gamma}}^p 
\\
&&= \sum_{n\in\bZ} e^{n(\theta-\alpha p/2)} \left\|\zeta_{-n}(e^n\cdot)\Delta^{\alpha/2}(e_{\overline{\cO}^c}w_m(e^n\cdot)) \right\|_{H_p^{\gamma}}^p 
\\
&&= \sum_{n\in\bZ} e^{n(\theta+\alpha p/2)} \left\|\zeta_{-n}(e^n\cdot)(\Delta^{\alpha/2}(e_{\overline{\cO}^c}w_m))(e^n\cdot) \right\|_{H_p^{\gamma}}^p 
\\
&&\leq N \| \psi^{\alpha/2} r_{\cO}\Delta^{\alpha/2}(e_{\overline{\cO}^c}w_m) \|_{H_{p,\theta}^{\gamma}(\cO)}.
\end{eqnarray*}
Thus, Lemma \ref{lem out} easily yields \eqref{ineq. 08.01-4} with $w_n$. Therefore \eqref{ineq. 08.01-4} also holds with $u_n:=v_n+w_n$. Letting $n\to\infty$, we get the desired result. The lemma is proved.
\end{proof}

{\textbf{Proof of Lemma \ref{lem 3.29}}}.

\begin{proof}

$(i)$ For $k\in \bN_+$ and $\phi\in C^{\infty}_c(\cO)$, 
$$
\|\phi\|_{\tilde{H}^k_p(\cO)} \approx \sum_{|\beta|\leq k} \|D^{\beta}_x\phi\|_{L_p(\cO)}\leq N \|\phi\|_{H^k_{p,\theta-kp}(\cO)}.
$$
Since $C^{\infty}_c(\cO)$ is dense in $H^k_{p,\theta-kp}(\cO)$, this shows that  $H^k_{p,\theta-kp}(\cO) \subset \tilde{H}^k_p(\cO)$. Thus, by the complex interpolation theorems (see \cite[Theorem 3.5]{T02}) we get
\begin{equation}
\label{eqn 4.10.1}
H^{\kappa}_{p,\theta-kp}(\cO) \subset \tilde{H}^{\kappa}_p(\cO), \quad \kappa \geq 0.
\end{equation}
If $\kappa<2/p$, then by Theorem \ref{main thm ellip} with $\gamma=0$,
  $\alpha=\kappa$ and $\theta=d-\kappa p/2$, for $u\in C_c^\infty(\cO)$,
  \begin{align*}
   \|u\|_{H^{\kappa}_{p,d-\kappa p}(\cO)}& \approx \|\psi^{-\kappa/2} u\|_{H_{p,d-\kappa p/2}^\kappa(\cO)} \leq N \| \psi^{\kappa/2} \Delta^{\kappa/2}u\|_{L_{p,d-\kappa p/2}(\cO)}
 \\& \approx N \|  \Delta^{\kappa/2}u\|_{L_{p}(\cO)} \leq N \|u\|_{\mathring{H}_p^\kappa(\cO)},
  \end{align*}
  where $\mathring{H}_p^\kappa(\cO)$ denotes the closure of $C_c^\infty(\cO)$ in $H_p^\kappa(\bR^d)$. 
  Consequently, the embedding $\mathring{H}_p^\kappa(\cO)\subset H^{\kappa}_{p,d-\kappa p}(\cO)$ is continuous.
This and \eqref{eqn 4.10.1} together with the relations (see \cite{T02}) 
$$
    (\widetilde{H}_p^\kappa(\cO))^*=H_{p'}^{-\kappa}(\cO), \quad    ({H}_p^\kappa(\cO))^*=\widetilde{H}_{p'}^{-\kappa}(\cO), \quad \kappa>0.
$$
and
$$
   \mathring{H}_p^\kappa(\cO)=\widetilde{H}_p^\kappa(\cO)=H_p^\kappa(\cO), \quad 1/p-1<\kappa<1/p
    $$
    prove $(i)$. 
    
    $(ii)$ Let $G\in H^{\beta}_p(\bR^d)$ such that $r_{\overline{\cO}^c}G=g$ and $supp \,G \subset \cO^c$. Note, choosing sufficiently small $\varepsilon>0$ and appropriate $x^0_i \in \partial \cO$, we get
    $$
    \|\psi^{-\beta}(1+\psi)^{-\beta}g\|_{L_p(\overline{\cO}^c)}\leq N \sum_{i} \int_{\{x\in \overline{\cO}^c\cap B_{\varepsilon}(x^0_i) \}} |d_x^{-\beta}G|^p dx + N \|G\|_{L_p(\overline{\cO}^c)}.
    $$
    Now we use (5.66) in \cite{Tr} to conclude that  if $\varepsilon$ is sufficiently small then
    $$
    \int_{\{x\in \overline{\cO}^c\cap B_{\varepsilon}(x^0_i) \}} |d_x^{-\beta}G|^p dx\leq N \|G\|_{H^{\beta}_p(\bR^d)}.
        $$
        Taking the infimum over $G$,  we get $(ii)$.

  $(iii)$ Since $\cO$ is bounded,
 \begin{align*}
\|u\|_{L_p(\bR^d)} =\|r_{\cO}u\|_{L_p(\cO)}+ \|r_{\overline{\cO}^c}u\|_{L_p(\overline{\cO}^c)} \leq \|r_{\cO}u\|_{L_{p,d-\beta p}(\cO)} +  \|r_{\overline{\cO}^c}u\|_{L_p(\overline{\cO}^c)}.
 \end{align*}
 
 To finish the proof for the case $\beta\geq0$, we will show that
 \begin{align*}
|(\Delta^{\beta/2}u,\phi)_{\bR^d}|  \leq NN_0 \|\phi\|_{L_{p'}(\bR^d)}, \quad \forall \phi \in C_c^\infty(\bR^d),
 \end{align*}
 where 
\begin{align*}
N_0&:=\|\psi^{-\alpha/2}r_{\cO}u\|_{H_{p,\theta}^{\beta}(\cO)} +\|g\|_{\tilde{H}^{\beta}_p(\overline{\cO}^c)},
 \end{align*}
 $1/p+1/p'=1$ and $\theta=d-\beta p+\alpha p/2$.
Let $\phi \in C_c^\infty(\bR^d)$. Then, by Theorem \ref{thm frac def} with $\lambda=0$ and $\sigma=0$,
\begin{align*}
\|\psi^{\alpha/2}r_{\cO}\Delta^{\beta/2}\phi\|_{H_{p',\theta'}^{-\beta}(\cO)} &= N \|\psi^{\beta/2}r_{\cO}\Delta^{\beta/2}\phi\|_{H_{p',d+\beta p/2}^{-\beta}(\cO)}
\\
&\leq N \|r_{\cO}\phi\|_{L_{p',d}(\cO)}+ N \|r_{\overline{\cO}^c}\phi\|_{L_{p',d}(\overline{\cO}^c)} = N\|\phi\|_{L_{p'}(\bR^d)}.
\end{align*}
Here, $\theta/p+\theta'/p'=d$. Let $G\in H^{\beta}_p(\bR^d)$ such that $r_{\overline{\cO}^c}G=g$ and $\text{supp}\, G \subset \cO^c$. Then, 
since $\Delta^{\beta/2}\phi\in H_{p'}^{-\beta}(\bR^d)$,
\begin{align*}
|(g,r_{\overline{\cO}^c}\Delta^{\beta/2}\phi)_{\overline{\cO}^c}| = |(G,\Delta^{\beta/2}\phi)_{\bR^d}| \leq N \|G\|_{H_p^\beta(\bR^d)} \|\Delta^{\beta/2} \phi\|_{H_{p'}^{-\beta}(\bR^d)}.
\end{align*}
Thus, 
\begin{align*} 
 |(u,\Delta^{\beta/2}\phi)_{\bR^d}|&\leq N \|\psi^{-\alpha/2}r_{\cO}u\|_{H_{p,\theta}^{\beta}(\cO)} \|\psi^{\alpha/2}r_{\cO}\Delta^{\beta/2}\phi\|_{H_{p',\theta'}^{-\beta}(\cO)} 
\\
&\quad +N \|G\|_{H_p^\beta(\bR^d)} \|\Delta^{\beta/2} \phi\|_{H_{p'}^{-\beta}(\bR^d)}.
\end{align*}
Taking the infimum over $G$, we get the desired result. The lemma is proved.
 \end{proof}

\section{Representation of solution and zero-th order estimate} \label{sec repre}

In this section, we introduce a probabilistic  representation of weak solution and estimate the zero-th order derivative  of solution. 

We consider the parabolic equation
\begin{equation} \label{parabolic sec 3}
\begin{cases}
\partial_t u(t,x)=\Delta^{\alpha/2}u(t,x),\quad &(t,x)\in(0,T)\times \cO,
\\
u(0,x)=0,\quad & x\in \cO,
\\
u(t,x)=g(t,x),\quad &(t,x)\in (0,T)\times \overline{\cO}^c,
\end{cases}
\end{equation}
as well as the elliptic equation
\begin{equation} \label{elliptic sec 3}
\begin{cases}
\Delta^{\alpha/2}u(x)=0,\quad &x\in \cO,\\
u(x)=g(x),\quad & x\in \overline{\cO}^c.
\end{cases}
\end{equation}

Let $X=(X)_{t\geq 0}$ be a rotationally symmetric $\alpha$-stable $d$-dimensional L\'evy process  defined on a probability space $(\Omega,\cF,\bP)$, that is, $X_t$ is a L\'evy process such that
$$
\bE e^{i \xi \cdot X_t}=e^{-|\xi|^{\alpha}t}, \quad \forall\, \xi\in \bR^d.
$$ 
Let $p(t,x)$ be the transition density of $X$.
For $x\in \bR^d$, let $\tau_{\cO}=\tau^x_{\cO}:=\inf\{t\geq0: x+X_t\not\in \cO\}$
denote the first exit time of $\cO$ by $X$. 
Let $p_{\cO} (t,x,y)$ denote the transition density of the process $X$ killed upon $\tau_{\cO}$, i.e., for functions $f\geq0$,
\begin{equation*}
  \int_{\bR^d} p_{\cO}(t,x,y) f(y) dy = \bE [f(x+X_t); \tau_{\cO}>t], \quad t>0, x\in \bR^d.
\end{equation*}
 We also define the Green function of $X$ in $\cO$ by
 $$
G_{\cO}(x,y)=\int_0^\infty p_{\cO} (t,x,y)dt.
 $$
Then, for functions $f\geq0$,
\begin{equation*}
  \int_{\bR^d} G_{\cO}(x,y) f(y) dy = \bE \int_0^{\tau_{\cO}} f(x+X_t) dt, \quad x\in \bR^d.
\end{equation*}
Due to the definition of the transition density, one can easily find $p_{\cO}(t,x,y)=0$ and $G_{\cO}(x,y)=0$ if $x\in \cO^c$ or $y\in \cO^c$.

Next, we define the Poisson kernel of $\cO$ for $\Delta^{\alpha/2}$ by
 \begin{align*}
K_{\cO}(x,z)=c_d \int_{\cO} G_{\cO}(x,y)|y-z|^{-d-\alpha} dy, \quad x\in \cO, z\in \cO^c,
\end{align*}
where $c_d=\frac{2^{\alpha}\Gamma(\frac{d+\alpha}{2})}{\pi^{d/2}|\Gamma(-\alpha/2)|}$. Then, by Ikeda-Watanabe formula (see \cite[Theorem 1]{IW}), for every $g\geq0$ on $\cO^c$,
$$
\bE[g(x+X_{\tau_{\cO}})]=\int_{\cO^c} K_{\cO}(x,z)g(z)dz, \quad x\in \cO.
$$
That is, $K_{\cO}(x,z)dz$ is the distribution of the random variable $x+X_{\tau_{\cO}}$. 
We also define the parabolic Poisson kernel of $\cO$ for $\Delta^{\alpha/2}$ by
 \begin{align*}
Q_{\cO}(t,x,z)=c_d \int_{\cO} p_{\cO}(t,x,y)|y-z|^{-d-\alpha} dy, \quad x\in \cO, z\in \cO^c, t>0.
\end{align*}

 Below we summarize some notations which will be used in this section.
\begin{itemize}
  \item $p(t,x)$: the transition density of $X$

  \item $p_{\cO}(t,x,y)$: the transition density of $X$ killed upon $\tau_{\cO}$

  \item $G_{\cO}(x,y)$: the Green function of $X$ in $\cO$

  \item $Q_{\cO}(t,x,y)$: the parabolic Poisson kernel of $\cO$ for $\Delta^{\alpha/2}$

  \item $K_{\cO}(x,y)$: the (elliptic) Poisson kernel of $\cO$ for $\Delta^{\alpha/2}$
\end{itemize}

In the following Lemmas \ref{pD est} and \ref{Qd est} and Corollary \ref{Kd est}, we introduce upper bounds of $p_{\cO}$, $K_{\cO}$ and $Q_{\cO}$.

\begin{lem}(\cite[Lemma 3.1]{Dirichlet}) \label{pD est}
For any $x,y\in \cO$,
\begin{align*} 
    p_{\cO} (t,x,y) 
 \leq
    \begin{cases}
     N\left(1\wedge \frac{d_x^{\alpha/2}}{\sqrt{t}}\right)\left(1\wedge \frac{d_y^{\alpha/2}}{\sqrt{t}}\right)p(t,x-y)\,&\text{if $\cO$ is a half space},
     \\
     Ne^{-c t}\left(1\wedge \frac{d_x^{\alpha/2}}{\sqrt{t}}\right)\left(1\wedge \frac{d_y^{\alpha/2}}{\sqrt{t}}\right)p(t,x-y)\, &\text{if $\cO$ is   bounded}.
    \end{cases}
\end{align*}
Here, $c,N>0$ depend only on $d,\alpha$ and $\cO$. 
\end{lem}

\begin{lem} \label{Qd est}
(i) For any $t>0, x\in \cO$ and $z\in \overline{\cO}^c$, there exists $N=N(d,\alpha,\cO)$ such that
\begin{align} \label{para poi}
Q_{\cO}(t,x,z) \leq N \left( t^{-d/\alpha-1} \wedge |x-z|^{-d-\alpha} \right) \left( 1\wedge \frac{d_x^{\alpha/2}}{\sqrt{t}} \right) \left( 1\wedge \frac{d_z^{\alpha/2}}{\sqrt{t}} \right)^{-1}.
\end{align}

(ii) In particular, if $\cO$ is bounded, then for $z\in \overline{\cO}^c$ such that $d_z>diam(\cO)$,
\begin{align} \label{para poi dom}
Q_{\cO}(t,x,z) \leq N |x-z|^{-d} \left( 1\wedge t^{-d/\alpha-1/2} \right) \left( 1\wedge \frac{d_x^{\alpha/2}}{\sqrt{t}} \right) d_z^{-\alpha},
\end{align}
where $N=N(d,\alpha,\cO)$.
\end{lem}

\begin{proof}
$(i)$ By Lemma \ref{pD est},
\begin{align*}
Q_{\cO}(t,x,z) &\leq N R_{t,x} \int_{\cO} p(t,x-y) R_{t,y} |y-z|^{-d-\alpha} dy,
\end{align*}
where $R_{t,x}:=\frac{d_x^{\alpha/2}}{\sqrt{t}+d_x^{\alpha/2}}$.
Thus, we only need to show that 
\begin{align} \label{para poi aim}
\int_{\cO} p(t,x-y) R_{t,y} |y-z|^{-d-\alpha} dy &\leq N \left( t^{-d/\alpha-1} \wedge |x-z|^{-d-\alpha} \right) R_{t,z}^{-1}.
\end{align}

First, assume $t\leq |x-z|^\alpha$. We divide the above integral as follows;
\begin{align} \label{070422-1}
&\int_{\cO} p(t,x-y) R_{t,y} |y-z|^{-d-\alpha} dy \nonumber
\\
&= \int_{\{y\in \cO:|x-y|>|x-z|/2\}} \cdots dy + \int_{\{y\in \cO :|x-y|\leq|x-z|/2\}} \cdots dy =: I_1+I_2.
\end{align}
Due to the triangle inequality $|x-z|\geq |y-z|-|x-y|$, the inequality $|x-y|\geq|x-z|/2$ implies that $|y-z|\leq 3|x-y|$.  Using the relation (see e.g. \cite[Theorem 1.1]{CK})\begin{align} \label{p est}
p(t,x) \approx t^{-d/\alpha} \wedge \frac{t}{|x-y|^{d+\alpha}},
\end{align}
 we get that  if $|x-y|\geq|x-z|/2$ then
$$
p(t,x-y)\leq N \frac{t}{|x-y|^{d+\alpha}} \leq N \frac{t}{|x-z|^{d+\alpha}}.
$$
Thus,
\begin{align} \label{070422-3}
I_1 &\leq N t|x-z|^{-d-\alpha}\int_{|y-z|\geq d_z} \left( |y-z|^{-d-\alpha} \wedge \frac{|y-z|^{-d-\alpha/2}}{\sqrt{t}} \right) dy \nonumber
\\
&\leq N |x-z|^{-d-\alpha} td_z^{-\alpha}R_{t,z} \leq N |x-z|^{-d-\alpha} R_{t,z}^{-1}.
\end{align}
Here, for the first inequality, we used $d_y\leq |y-z|$ for $y\in \cO$ and $z\in \overline{\cO}^c$.
If $|x-y|<|x-z|/2$, then $|y-z|\geq|x-z|-|x-y|>|x-z|/2$. Thus,
\begin{align*}
I_2 &\leq N \int_{\{y\in \cO :|x-y|\leq|x-z|/2\}} \left( |y-z|^{-d-\alpha} \wedge \frac{|y-z|^{-d-\alpha/2}}{\sqrt{t}} \right) p(t,x-y) dy
\\
&\leq N \int_{\{y\in \cO :|x-y|\leq|x-z|/2\}} |y-z|^{-d-\alpha} p(t,x-y) dy
\\
&\leq N |x-z|^{-d-\alpha} \int_{\bR^d} p(t,x-y)dy = N |x-z|^{-d-\alpha} \leq N|x-z|^{-d-\alpha} R_{t,z}^{-1}.
\end{align*}
This, \eqref{070422-1} and \eqref{070422-3} yield \eqref{para poi aim} when $t\leq |x-z|^\alpha$.

Next, we assume $t>|x-z|^\alpha$. Due to \eqref{p est}, 
\begin{align*}
\int_{\cO} p(t,x-y) R_{t,y} |y-z|^{-d-\alpha} dy &\leq N t^{-d/\alpha} \int_{\cO} R_{t,y}|y-z|^{-d-\alpha} dy
\\
&\leq N t^{-d/\alpha} \int_{|y-z|\geq d_z} \frac{|y-z|^{-d-\alpha/2}}{\sqrt{t}} dy 
\\
&= N t^{-d/\alpha-1/2} d_z^{-\alpha/2} = N t^{-d/\alpha-1} R_{t,z}^{-1}.
\end{align*}
Here, the last equality follows from $t>|x-z|^{\alpha}>d_z^{\alpha}$.
Thus, we also have \eqref{para poi aim} when $t>|x-z|^\alpha$.

$(ii)$ As in the proof of $(i)$, it suffices to show that for $d_z>diam(\cO)$,
\begin{align} \label{para aid dom}
\int_{\cO} p(t,x-y) R_{t,y} |y-z|^{-d-\alpha} dy &\leq N |x-z|^{-d} \left( 1\wedge t^{-d/\alpha-1/2} \right) d_z^{-\alpha}.
\end{align}
Assume $d_z>diam(\cO)$. Since $d_y\leq N$, $d_z\leq |y-z|$ for $y\in \cO$ and $z\in \cO^c$,
\begin{align} \label{070516-1}
\int_{\cO} p(t,x-y) R_{t,y} |y-z|^{-d-\alpha} dy
&\leq \left(1 \wedge \frac{1}{\sqrt{t}}\right) \int_{\cO} p(t,x-y)|y-z|^{-d-\alpha}dy \nonumber
\\
&\leq N \left(1 \wedge \frac{1}{\sqrt{t}}\right) d_z^{-d-\alpha} \int_{\cO} p(t,x-y) dy.
\end{align}
By \eqref{p est},
\begin{align} \label{07.19}
\int_{\cO} p(t,x-y) dy \leq \left(\int_{\bR^d} p(t,x-y) dy\right) \wedge \left(\int_{\cO} t^{-d/\alpha} dy\right) \leq N \left(1 \wedge t^{-d/\alpha} \right).
\end{align}
Moreover, due to $d_z>diam(\cO)$, we have $|x-z|\approx d_z$. 
Thus, \eqref{070516-1} and \eqref{07.19} yield \eqref{para aid dom}. The lemma is proved.
\end{proof}

\begin{corollary} \label{Kd est}
 There exists $N=N(d,\cO,\alpha)$ such that for any $x \in \cO$ and $z\in \overline{\cO}^c$, 
\begin{align} \label{poisson half}
K_{\cO}(x,z) \leq N \frac{d_x^{\alpha/2}}{d_z^{\alpha/2}}\frac{1}{|x-z|^d}.
\end{align}
Furthermore, if $\cO$ is  bounded, then 
\begin{align} \label{poisson dom}
K_{\cO}(x,z) \leq N \frac{d_x^{\alpha/2}}{(1+d_z)^{\alpha/2} d_z^{\alpha/2}}\frac{1}{|x-z|^d}.
\end{align}
\end{corollary}

\begin{proof}
 By Fubini's theorem and \eqref{para poi}, we have
\begin{align*} 
K_{\cO}(x,z) &=\int_0^\infty Q_{\cO}(t,x,z) dt 
\\
&\leq N \int_0^\infty \left( t^{-d/\alpha-1} \wedge |x-z|^{-d-\alpha} \right) \left( 1\wedge \frac{d_x^{\alpha/2}}{\sqrt{t}} \right) \left( 1\wedge \frac{d_z^{\alpha/2}}{\sqrt{t}} \right)^{-1} dt
\\
&= N \left(\int_0^{d_z^\alpha} \cdots dt + \int_{d_z^\alpha}^{|x-z|^{\alpha}} \cdots dt + \int_{|x-z|^{\alpha}}^\infty \cdots dt \right) 
\\
&=: N (I+II+III).
\end{align*}
Using the relation $|x-z|>d_z$, one can easily show that
\begin{align*}
I &\leq \int_0^{d_z^\alpha} |x-z|^{-d-\alpha} \frac{d_x^{\alpha/2}}{\sqrt{t}} dt \leq N |x-z|^{-d-\alpha} d_x^{\alpha/2} d_z^{\alpha/2} \leq N |x-z|^{-d} d_x^{\alpha/2} d_z^{-\alpha/2}.
\end{align*}
For $II$,
\begin{align*}
II \leq |x-z|^{-d-\alpha} d_x^{\alpha/2} d_z^{-\alpha/2} \int_{d_z^\alpha}^{|x-z|^{\alpha}} dt \leq N |x-z|^{-d} d_x^{\alpha/2} d_z^{-\alpha/2}.
\end{align*}
Lastly,
\begin{align*}
III &\leq d_x^{\alpha/2} d_z^{-\alpha/2} \int_{|x-z|^{\alpha}}^\infty t^{-d/\alpha-1} dt \leq N |x-z|^{-d} d_x^{\alpha/2} d_z^{-\alpha/2}.
\end{align*}
Thus, we have \eqref{poisson half}.

Next we prove \eqref{poisson dom}. Since $1+d_z\approx 1$ if $d_z\leq diam(\cO)$, one can easily get \eqref{poisson dom} from \eqref{poisson half}. Thus, we now assume $d_z> diam(\cO)$.  By Fubini's theorem and \eqref{para poi dom},
\begin{align*} 
K_{\cO}(x,z) &\leq N d_z^{-\alpha} |x-z|^{-d} \int_0^\infty \left( 1\wedge t^{-d/\alpha-1/2} \right) \left( 1\wedge \frac{d_x^{\alpha/2}}{\sqrt{t}} \right) dt
\\
&\leq N d_z^{-\alpha} d_x^{\alpha/2} |x-z|^{-d} \int_0^\infty \left( 1\wedge t^{-d/\alpha-1} \right) dt \leq N d_z^{-\alpha} d_x^{\alpha/2} |x-z|^{-d}.
\end{align*}
Due to $1+d_z\approx d_z$, this actually yields \eqref{poisson dom}. The corollary is proved.
\end{proof}

Denote
$$
\cQ_{\cO} g(t,x):=\int_0^t\int_{\overline{\cO}^c} Q_{\cO}(t-s,x,z)g(s,z)dz ds,
$$
and
$$
\cK_{\cO} g(x):=\int_{\overline{\cO}^c} K_{\cO}(x,z)g(z)dz.
$$

\begin{lem} \label{lem para pp}
Let $\cO$ be a half space and suppose $d-1<\theta<d-1+p$.
 Then, there exists $N=N(d,\alpha,\theta,p,\cO)$ such that  for any  $g\in \psi^{\alpha/2}\bL_{p,\theta}(\overline{\cO}^c,T)$,
 \begin{align*}
\|\psi^{-\alpha/2}\cQ_{\cO} g\|_{\bL_{p,\theta}(\cO,T)} \leq N \|\psi^{-\alpha/2}g\|_{\bL_{p,\theta}(\overline{\cO}^c,T)}.
\end{align*}
\end{lem}

\begin{proof}
Since $\psi\approx d_x$, it suffices to show
\begin{align} \label{ineq 1101}
\int_0^T \int_{\cO} d_x^{\mu-\alpha p/2} |\cQ_{\cO} g(t,x)|^p dxdt \leq N \int_0^T \int_{\overline{\cO}^c} d_z^{\mu-\alpha p/2} |g(t,z)|^p dzdt,
\end{align}
where $\mu=\theta-d$. 
 Since $\mu\in(-1,p-1)$, one can take $\upsilon_0$ satisfying
\begin{align} \label{ineq. 7.13}
-2+\frac{2}{p}=-\frac{2}{p'}<\upsilon_0<\left(2+\frac{2}{\alpha}\right)\frac{1}{p'}=\left(2+\frac{2}{\alpha}\right)\frac{p-1}{p},
\end{align}
and
\begin{align} \label{ineq. 7.13.1}
\frac{2\mu}{\alpha p}-1-\frac{2}{p}<\upsilon_0<\frac{2\mu}{\alpha p}-1+\frac{2}{p}+\frac{2}{\alpha p}, \quad \upsilon_0 \neq \frac{2p-2}{\alpha p}, \quad \upsilon_0 \neq \frac{2\mu}{\alpha p}-1+\frac{2}{\alpha p}.
\end{align}
Since $d-1<\theta<d-1+p$, one can also take $\upsilon_1$ and $\upsilon_2$ satisfying
\begin{align} \label{ineq. 7.13.2}
-1 < \upsilon_0-\upsilon_1 < -1 + \frac{2\mu}{\alpha p} + \frac{2}{\alpha p} \,\text{ and }\, \frac{2\mu}{\alpha p}-1<\upsilon_0+\upsilon_2<\frac{2(p-1)}{\alpha p} - 1.
\end{align}
 By \eqref{para poi} and H\"older's inequality,
\begin{align*}
|\cQ_{\cO} g(t,x)| &\leq N \left(\int_0^t \int_{\overline{\cO}^c} S_{t-s,x,z} d_z^{\alpha\upsilon_0p/2} R_{t-s,x}^{\upsilon_1 p} R_{t-s,z}^{\upsilon_2 p} |g(s,z)|^p dzds \right)^{1/p}
\\
&\quad \times \left(\int_0^t \int_{\overline{\cO}^c}  S_{t-s,x,z} d_z^{-\alpha\upsilon_0p'/2} R_{t-s,x}^{(1-\upsilon_1) p'} R_{t-s,z}^{(-1-\upsilon_2) p'} dzds \right)^{1/p'}
\\
&=:N\times I(t,x) \times II(t,x),
\end{align*}
where $p'=p/(p-1)$, $S_{t,x,z}:=t^{-d/\alpha-1}\wedge |x-z|^{-d-\alpha}$ and $R_{t,x}:=1\wedge \frac{d_x^{\alpha/2}}{\sqrt{t}}$. By Lemma \ref{aux para half} with $\nu_0=(-1-\upsilon_2)p'$ and $\nu_1=-\upsilon_0p'$, and the change of variables,
\begin{align} \label{ineq. 07.27.}
&II(t,x)^{p'} \nonumber
\\
&\leq N \int_0^t \left( (t-s)^{-\upsilon_0p'/2-1} \wedge \left( (t-s)^{-\upsilon_0p'/2+1/\alpha}d_x^{-1-\alpha} \vee d_x^{-\alpha\upsilon_0p'/2-\alpha} \right) \right) \nonumber
\\
& \qquad\qquad\qquad\qquad\qquad\qquad\qquad\qquad\qquad\qquad\qquad\qquad\qquad\qquad\times R_{t-s,x}^{(1-\upsilon_1)p'} ds \nonumber
\\
&\leq N \int_0^t \left( s^{-\upsilon_0p'/2-1} \wedge \left( s^{-\upsilon_0p'/2+1/\alpha}d_x^{-1-\alpha} \vee d_x^{-\alpha\upsilon_0p'/2-\alpha} \right) \right) \nonumber
\\
& \qquad\qquad\qquad\qquad\qquad\qquad\qquad\qquad\qquad\qquad\qquad\quad \times \frac{d_x^{(1-\upsilon_1)p'\alpha/2}}{(\sqrt{s}+d_x^{\alpha/2})^{(1-\upsilon_1)p'}} ds \nonumber
\\
&= N d_x^{-\alpha\upsilon_0p'/2} \int_0^\infty \left( s^{-\upsilon_0p'/2-1} \wedge \left( s^{-\upsilon_0p'/2+1/\alpha} \vee 1 \right) \right) (\sqrt{s}+1)^{(\upsilon_1-1)p'} ds \nonumber
\\
&\leq N d_x^{-\alpha\upsilon_0p'/2}.
\end{align}
Here, the last inequality is due to $-\upsilon_0p'/2+1/\alpha>-1$ and $-\upsilon_0p'/2-1+(\upsilon_1-1)p'/2<-1$ which follow from \eqref{ineq. 7.13} and \eqref{ineq. 7.13.2}, respectively.
By Fubini's theorem,
\begin{align*}
&\int_0^T \int_{\cO} d_x^{\mu-\alpha p/2} |\cQ_{\cO} g(t,x)|^p dxdt 
\\
&\leq N \int_0^T \int_{\overline{\cO}^c} d_z^{\alpha\upsilon_0p/2} |g(s,z)|^p 
\\
&\qquad \qquad \times \left(\int_s^T \int_{\cO} S_{t-s,x,z} d_x^{\mu - \alpha p/2 - \alpha\upsilon_0 p/2} R_{t-s,x}^{\upsilon_1 p}  R_{t-s,z}^{\upsilon_2 p} dxdt \right) dzds.
\end{align*}
As in \eqref{ineq. 07.27.}, by \eqref{ineq. 7.13.1}, \eqref{ineq. 7.13.2}, and Lemma \ref{aux para half} with $\nu_0=\upsilon_1p$ and $\nu_1=2\mu/\alpha -p -\upsilon_0p$,
\begin{align*}
&\int_s^T \int_{\cO} S_{t-s,x,z} d_x^{\mu-\alpha p/2 - \alpha\upsilon_0 p/2} R_{t-s,x}^{\upsilon_1 p}  R_{t-s,z}^{\upsilon_2 p} dxdt \leq N d_z^{\mu-\alpha p/2-\alpha\upsilon_0p/2}.
\end{align*}
Thus, we have \eqref{ineq 1101}, and the lemma is proved.
\end{proof}

\begin{lem} \label{lem zeroth para}
Let $\cO$ be bounded, $d-1<\theta<d-1+p$ and $\sigma >-\theta-\alpha p/2$.
 Then, there exists $N=N(d,\alpha,\theta,p,\sigma,\cO)$ such that for any  $g\in \psi^{\alpha/2}\bL_{p,\theta,\sigma}(\overline{\cO}^c,T)$,
 \begin{align*}
\|\psi^{-\alpha/2}\cQ_{\cO} g\|_{\bL_{p,\theta}(\cO,T)} \leq N \|\psi^{-\alpha/2} g\|_{\bL_{p,\theta,\sigma}(\overline{\cO}^c,T)}.
\end{align*}
\end{lem}

\begin{proof}
Let $\Phi\in C_c^\infty(\bR^d)$ be a cut-off function such that $\Phi(x)=1$ if $d_x\leq diam(\cO)$, and $\Phi(x)=0$ if $d_x>2diam(\cO)$. Then, $g(t,z)\Phi(z)=0$ if $d_z>2diam(\cO)$, and thus $g\Phi \in \psi^{\alpha/2}\bL_{p,\theta}(\overline{\cO}^c)$.
By repeating the same argument used in the proof of Lemma \ref{lem para pp},
\begin{align*}
\|\psi^{-\alpha/2}\cQ_{\cO} (g\Phi)\|_{\bL_{p,\theta}(\cO,T)}
 \leq N\|\psi^{-\alpha/2} g\Phi\|_{\bL_{p,\theta}(\overline{\cO}^c,T)} \leq N \|\psi^{-\alpha/2} g\|_{\bL_{p,\theta,\sigma}(\overline{\cO}^c,T)}.
\end{align*}
Hence, it suffices to assume that $g(t,z)=0$ if $d_z \leq diam(\cO)$.
Denote $\mu:=\theta-d$, $\kappa:=\sigma/p+\theta/p+\alpha/2$, and $p'=p/(p-1)$. Take $\upsilon$ such that
\begin{align} \label{ineq. 07.27-1}
-\frac{2\mu+2}{\alpha p}-1+\frac{2}{p}<\upsilon<-1+\frac{2}{p}.
\end{align}
 By \eqref{para poi dom} and H\"older's inequality,
\begin{align*}
|\cQ_{\cO} g(t,x)| &\leq N \left(\int_0^t \int_{\overline{\cO}^c} |x-z|^{-d} S_{t-s}^{p} d_z^{-\alpha p +\kappa p} R_{t-s,x}^{\upsilon p}  |g(s,z)|^p dzds \right)^{1/p}
\\
&\quad \times \left(\int_0^t \int_{\{z\in \overline{\cO}^c: d_z>diam(\cO)\}} |x-z|^{-d} d_z^{-\kappa p'} R_{t-s,x}^{(1-\upsilon) p'}  dzds \right)^{1/p'}
\\
&=: N \times I(t,x) \times II(t,x),
\end{align*}
where $p'=p/(p-1)$, $S_{t}:=1\wedge t^{-d/\alpha-1/2}$ and $R_{t,x}:=1\wedge \frac{d_x^{\alpha/2}}{\sqrt{t}}$. 
Since $|x-z|\approx d_z$ if $x\in \cO$ and $d_z>diam(\cO)$,
\begin{align} \label{ineq. 07.31-3}
&\int_{\{z\in \overline{\cO}^c: d_z>diam(\cO)\}} |x-z|^{-d} d_z^{-\kappa p'}  dz \nonumber
\\
&\leq N \int_{|x-z|\geq diam(\cO)+d_x} |x-z|^{-d-\kappa p'} dz \leq N (1+d_x)^{-\kappa p'}.
\end{align}
Thus, by the change of variables and \eqref{ineq. 07.27-1},
\begin{align*}
II(t,x)^{p'} &\leq N (1+d_x)^{-\kappa p'} \int_0^t R_{t-s,x}^{(1-\upsilon) p'} ds
\\
&\leq N d_x^{\alpha}(1+d_x)^{-\kappa p'} \int_0^\infty \left( 1\wedge s^{-1/2} \right)^{(1-\upsilon) p'} ds \leq N d_x^{\alpha}(1+d_x)^{-\kappa p'}.
\end{align*}
Hence, by Fubini's theorem,
\begin{align*}
&\int_0^T \int_{\cO} d_x^{\mu-\alpha p/2} |\cQ_{\cO} g(t,x)|^p dxdt 
\\
&\leq N \int_0^T \int_{\overline{\cO}^c} d_z^{-\alpha p +\kappa p} |g(s,z)|^p 
\\
&\qquad \qquad \times \left(\int_s^T \int_{\cO} |x-z|^{-d} S_{t-s}^p d_x^{\mu + \alpha p/2 -\alpha} (1+d_x)^{-\kappa p} R_{t-s,x}^{\upsilon p} dxdt \right) dzds.
\end{align*}
Since $1+d_x\approx 1$ for $x\in \cO$, by Lemma \ref{aux dom} with $\nu_0=\upsilon p$ and $\nu_1=2\mu/\alpha+p-2$, 
\begin{align*}
&\int_s^T \int_{\cO} |x-z|^{-d} S_{t-s}^p d_x^{\mu + \alpha p/2 -\alpha} (1+d_x)^{-\kappa p} R_{t-s,x}^{\upsilon p} dxdt 
\\
&\leq N \int_s^T d_z^{-d} S_{t-s}^p ((t-s)^{-\upsilon p/2}+1) ds
\\
&\leq N d_z^{-d} \int_0^\infty (s^{-\upsilon p/2}+1) \left( 1\wedge s^{-d/\alpha-1/2} \right)^p ds \leq N d_z^{-d}.
\end{align*}
For the last inequality, we used $-\upsilon p/2>-1$, $-dp/\alpha-p/2<-1$ and $-\upsilon p/2-dp/\alpha-p/2<-1$ which follow from \eqref{ineq. 07.27-1}.
Therefore, 
\begin{align*}
\int_0^T \int_{\cO} d_x^{\mu-\alpha p/2} |\cQ_{\cO} g(t,x)|^p dxdt &\leq N \int_0^T \int_{\overline{\cO}^c} d_z^{-d-\alpha p +\kappa p} |g(s,z)|^p dzds
\\
&\leq N \int_0^T \int_{\overline{\cO}^c} d_z^{\mu-\alpha p/2} (1+d_z)^{\sigma} |g(s,z)|^p dzds.
\end{align*}
The lemma is proved.
\end{proof}

\begin{lem} \label{lem zeroth ell}
Let $\alpha\in(0,2)$ and $p\in (1,\infty)$. Suppose that $d-1<\theta<d-1+p$ and $\sigma >-\theta-\alpha p/2$.

(i) There exists $N=N(d,\alpha,\theta,p,\cO)$ such that for any $g\in \psi^{\alpha/2}L_{p,\theta}(\overline{\cO}^c)$,
\begin{align} \label{22.06.04.1700}
\| \psi^{-\alpha/2} \cK_{\cO} g \|_{L_{p,\theta}(\cO)} \leq N \| \psi^{-\alpha/2} g \|_{L_{p,\theta}(\overline{\cO}^c)}.
\end{align}

(ii) If $\cO$ is  bounded, then there exists $N=N(d,\alpha,\theta,p,\cO,\sigma)$ such that for any $g\in \psi^{\alpha/2}L_{p,\theta}(\overline{\cO}^c)$,
\begin{align} \label{ineq 1105-3}
\| \psi^{-\alpha/2} \cK_{\cO} g \|_{L_{p,\theta}(\cO)} \leq N \| \psi^{-\alpha/2} g \|_{L_{p,\theta,\sigma}(\overline{\cO}^c)}.
\end{align}
\end{lem}

\begin{proof}
$(i)$ Since $\mu:=\theta-d\in(-1,p-1)$, one can take $\upsilon\in\bR$ satisfying
\begin{align} \label{ineq. 07.31-1}
0<\upsilon <\frac{p-1}{p} \,\text{ and }\, \frac{\theta-d}{p} < \upsilon < \frac{\theta-d}{p} + \frac{1}{p}.
\end{align}
Due to \eqref{ineq. 07.31-1}, one can apply \eqref{aux ineq 1} with $\nu_0=-\upsilon p'$ and $\nu_1=0$ to get
\begin{align*}
\int_{\overline{\cO}^c} d_z^{-\upsilon p'}|x-z|^{-d} dz &\leq N d_x^{-\upsilon p'}.
\end{align*}
Thus, by \eqref{poisson half} and H\"older's inequality,
\begin{align*}
|\cK_{\cO} g(x)|^p &\leq N d_x^{\alpha p/2-\nu p} \int_{\overline{\cO}^c} |g(z)|^p d_z^{\upsilon p -\alpha p /2} |x-z|^{-d} dz,
\end{align*}
where $p'=p/(p-1)$. 
By Fubini's theorem,
\begin{align*}
\int_{\cO} d_x^{\mu-\alpha p/2} |\cK_{\cO} g(x)|^p dx &\leq N \int_{\cO} \int_{\overline{\cO}^c} d_x^{\mu-\upsilon p}  |g(z)|^p d_z^{\upsilon p - \alpha p /2}|x-z|^{-d} dzdx.
\end{align*}
Again by \eqref{aux ineq 1} with $\nu_0=\mu-\upsilon p$, we have \eqref{22.06.04.1700}.

$(ii)$ Let $\cO$ be bounded and $\kappa:=\sigma/p+\theta/p+\alpha/2$. As in the proof of Lemma \ref{lem zeroth para}, we assume $g(z)=0$ if $d_z\leq diam(\cO)$. By H\"older's inequality, \eqref{poisson dom}, and \eqref{ineq. 07.31-3},
\begin{align*}
|\cK_{\cO} g(x)|^p &\leq N d_x^{\alpha p/2} (1+d_x)^{-\kappa p} \int_{\overline{\cO}^c} |g(z)|^p d_z^{\kappa p -\alpha p /2} (1+d_z)^{-\alpha p/2} |x-z|^{-d} dz.
\end{align*}
Applying Fubini's theorem and Lemma \ref{23.05.07.1603} with $t=1$, $\nu_0=0$ and $\nu_1=2\mu/\alpha$ in order (note also $1+d_x \approx 1$), we get
\begin{align*}
\int_{\cO} d_x^{\mu-\alpha p/2} |\cK_{\cO} g(x)|^p &\leq N \int_{\overline{\cO}^c} \int_{\cO} d_x^{\mu} |g(z)|^p d_z^{\kappa p - \alpha p /2} (1+d_z)^{-\alpha p /2} |x-z|^{-d} dxdz
\\
&\leq N \int_{\overline{\cO}^c} d_z^{\mu-\alpha p/2} (1+d_z)^{\kappa p-\mu -d -\alpha p /2} |g(z)|^p dz.
\end{align*}
Thus, we have \eqref{ineq 1105-3} and the lemma is proved.
\end{proof}

\begin{lem} \label{lem prob para}
Let  $\theta\in(d-1,d-1+p)$ and $g\in C_c^\infty((0,T)\times\overline{\cO}^c)$. If we define
\begin{align*}
u(t,x):=\begin{cases}\cQ_{\cO} g(t,x)=\int_0^t \int_{\overline{\cO}^c} Q_{\cO} (t-s,x,z) g(s,z) dzds \quad &:   \, x\in \cO 
\\
g(t,x) \quad &:  \, x\in \overline{\cO}^c
\end{cases}
\end{align*}
then $r_{(0,T)\times \cO}u\in \frH_{p,\theta}^{0}(\cO,T)$, and $u$ becomes a  weak solution to \eqref{parabolic sec 3} in the sense of Definition \ref{sol def}.
\end{lem}

\begin{proof}
First, by Lemma \ref{lem zeroth para} we have $r_{(0,T)\times\cO}u\in \psi^{\alpha/2}\bL_{p,\theta}(\cO,T)$, and by Lemma \ref{thm frac def} we also have $u(t,\cdot)\in \cD'(\bR^d)$ and $\Delta^{\alpha/2}u \in \bH^{-\alpha}_{p,\theta+\alpha p/2}(\cO,T)$.
 Thus, thanks to  Remark \ref{remark 3.18},  it is enough to prove that  $u$ is a weak solution to \eqref{parabolic sec 3}.

Let $\phi\in C_c^\infty(\cO)$. For $x\in \cO$ and $z\in supp(g)$, $|x-z|\geq dist(supp(g), \partial \cO)$. Thus, by \eqref{para poi} or \eqref{para poi dom},
\begin{align*} 
Q_{\cO}(s-r,x,z)|g(r,z)| &\leq N(supp(g),\alpha,d,\cO) |g(r,z)|.
\end{align*}
Thus, one can apply Fubini's theorem to get
\begin{align*}
&\int_0^t (u(s,\cdot), r_{\cO}\Delta^{\alpha/2} (e_{\cO}\phi))_{\cO} ds
\\
&= c_d \int_0^t \int_{\overline{\cO}^c} \int_{\cO} \int_r^t g(r,z)|y-z|^{-d-\alpha} T^{\cO}_{s-r}(r_{\cO}\Delta^{\alpha/2} (e_{\cO}\phi))(y) ds dydzdr,
\end{align*}
where $c_d=\frac{2^{\alpha}\Gamma(\frac{d+\alpha}{2})}{\pi^{d/2}|\Gamma(-\alpha/2)|}$ and 
$$
T^{\cO}_s\varphi(y):=\int_{\cO} p_{\cO}(s,x,y) \varphi(x)dy = \int_{\cO} p_{\cO}(s,y,x) \varphi(x)dy.
$$
Here, we remark that $\{T^{\cO}_s\}_{s\geq0}$ is a Feller semigroup in $L_\infty(\cO)$ (see e.g. \cite[Example 1.3]{BS13} and page 68 of \cite{C86}). Thus, we have $\partial_s T^{\cO}_s \varphi = T^{\cO}_s (A^{\cO}\varphi)$, and
$$
\lim_{s\to0+} \|T^{\cO}_s\phi-\phi\|_{L_\infty(\cO)}=0,
$$
where $A^{\cO}$ is the infinitesimal generator of $\{T^{\cO}_s\}_{s\geq0}$. Here, by \cite[Lemma 2.6]{BLM18}, we have $r_{\cO}\Delta^{\alpha/2} (e_{\cO}\phi)=A^{\cO}\phi$, and thus $\partial_s T^{\cO}_s \phi=T^{\cO}_s(r_{\cO}\Delta^{\alpha/2} (e_{\cO}\phi))$.

Thus,
$$
\int_r^t T^{\cO}_{s-r}(r_{\cO}\Delta^{\alpha/2} (e_{\cO}\phi))(y) ds = T^{\cO}_{t-r}\phi (y) - \phi(y) = \int_{\cO} p_{\cO}(t-r,x,y)\phi(x)dx - \phi(y).
$$
Since $\phi \in C_c^\infty(\cO)$, for $z\in \cO^c$,
$$
 c_d \int_{\cO} \phi(y)|y-z|^{-d-\alpha}dy =\Delta^{\alpha/2}(e_{\cO}\phi)(z),
$$
 applying Fubini's theorem again we have
\begin{align*}
&\int_0^t (r_{(0,T)\times\cO}u(s,\cdot), r_{\cO}\Delta^{\alpha/2} (e_{\cO}\phi))_{\cO} ds 
\\
&=(r_{(0,T)\times\cO}u(t,\cdot),\phi)_{\cO} -\int_0^t (g(r,\cdot), r_{\overline{\cO}^c}\Delta^{\alpha/2} (e_{\cO}\phi))_{\overline{\cO}^c} dr.
\end{align*}
This actually shows that $u$ is a weak solution to \eqref{parabolic sec 3}. The lemma is proved.
\end{proof}

\begin{lem} \label{lem prob ellip}
Let  $\theta\in(d-1,d-1+p)$ and $g\in C_c^\infty(\overline{\cO}^c)$. If we define
$$
u(x):= \begin{cases} \cK_{\cO} g(x) = \int_{\overline{\cO}^c} K_{\cO}(x,z) g(z) dz\quad &:\quad x\in \cO
 \\
 g(x) \quad &: \quad x\in \overline{\cO}^c
 \end{cases}
$$
then $r_{\cO}u \in \psi^{\alpha/2}H_{p,\theta}^{0}(\cO)$, and $u$ becomes a weak solution to \eqref{elliptic sec 3} in the sense of Definition \ref{sol def}.
\end{lem}

\begin{proof}
First, by Lemma \ref{lem zeroth ell} we have $u\in \psi^{\alpha/2}L_{p,\theta}(\cO,T)$. Also by Lemma \ref{thm frac def}, $u\in \cD'(\bR^d)$ and $\Delta^{\alpha/2}u\in H^{-\alpha}_{p,\theta+\alpha p/2}(\cO)$.  

Now we show $\Delta^{\alpha/2}u=0$ in $\cO$ in the sense of distribution. By \cite[Theorem 5.5]{Bogdan trace}, $\cE_{\cO}(u,u)<\infty$ and $\cE_{\cO}(u,\phi)=0$ for all $\phi \in C^{\infty}_c(\cO)$,
where
\begin{align*}
\cE_{\cO}(u,\eta):= c(d,\alpha) \int\int_{\bR^d\times\bR^d\setminus \cO^c\times \cO^c} (u(y)-u(x))(\eta(y)-\eta(x))|x-y|^{-d-\alpha} dydx.
\end{align*}
  Put $\chi(x)=0$ if $\alpha\in (0,1)$, $\chi(x)=1_{|x|<1}$ if $\alpha=1$, and $\chi(x)=1$ if $\alpha\in (1,2)$.  Then for any $\phi\in C_c^\infty(\cO)$,
  \begin{align} \label{eqn 3.19.1}
\Delta^{\alpha/2} (e_{\cO}\phi)(x) &=\lim_{\varepsilon\to0} c_d \int_{|x-y|>\varepsilon} \frac{e_{\cO}\phi(y)-e_{\cO}\phi(x)- \chi(y-x) (y-x)\cdot \nabla \phi (x)}{|x-y|^{d+\alpha}}dy.
\end{align}  
Now we check that  \eqref{eqn 3.19.1} also holds in $L_1(\bR^d)$.  Since $(1+|x|)\leq N(\phi)|x-y|$ if $y\in supp(\phi)$ and $|x-y|\geq1$,
 we have
\begin{align*} \label{frac bdd}
&\int_{\bR^d} \frac{|e_{\cO}\phi(y)-e_{\cO}\phi(x)-(y-x)\cdot \nabla\phi(x) \chi(x-y)|}{|x-y|^{d+\alpha}} dy
\\ 
&\leq \int_{|x-y|>1} \frac{|e_{\cO}\phi(y)|+|e_{\cO}\phi(x)|}{|x-y|^{d+\alpha}} dy + 1_{\alpha>1}\int_{|x-y|>1} \frac{|D_x(e_{\cO}\phi)(x)|}{|x-y|^{d-1+\alpha}} dy
\\
&\quad+ \sup_{z\in \bR^d}|D^2_z\phi(z)| 1_{|x|\leq N(\phi)} \int_{|x-y|\leq1} |x-y|^{2-d-\alpha} dy \nonumber
\\
&\leq N(\phi) (1+|x|)^{-d-\alpha} + N(\phi)1_{\alpha>1} (1+|x|)^{-d+1-\alpha}.
\end{align*}
Since the last term is in $L_1(\bR^d)$, we conclude that  \eqref{eqn 3.19.1}  holds in $L_1(\bR^d)$.
Also note that 
by the definition of $u$, $\|u\|_{L_\infty(\bR^d)} \leq \|g\|_{L_\infty(\cO^c)}$.
Thus, we can apply Fubini's theorem and the dominated convergence theorem to get 
\begin{align*}
 0=\cE_{\cO}(u,\phi) &= -2 c_d\lim_{\varepsilon\to0}\int_{\bR^d} u(x) \int_{|x-y|>\varepsilon} (e_{\cO}\phi(y)-e_{\cO}\phi(x))|x-y|^{-d-\alpha} dydx 
 \\
 &= -2 \int_{\bR^d} u(x)\Delta^{\alpha/2}(e_{\cO}\phi)(x) dx.
\end{align*}
Hence, $u$ is a weak solution to \eqref{elliptic sec 3} in the sense of Definition \ref{sol def}. The lemma is proved.
\end{proof}

\section{Proof of main results}

We first consider the equations with zero exterior condition. Lemmas \ref{lem zero para}-\ref{lem zero initial}, and Lemma \ref{lem zero ellip} are extensions of Theorems 2.9 and 2.10 in \cite{Dirichlet}, respectively. Here, we obtain arbitrary (real) order regularity results.

\begin{lem} \label{lem zero para}
Let $p\in(1,\infty)$, $\gamma\in\bR$,  and $\theta\in(d-1,d-1+p)$. Then, for any $f\in \psi^{-\alpha/2}\bH_{p,\theta}^\gamma(\cO,T)$, the equation
\begin{equation} \label{eqn. zero para}
\begin{cases}
\partial_t u = \Delta^{\alpha/2}u+f,\quad &(t,x)\in(0,T)\times \cO,
\\
u(0,x)=0,\quad & x\in \cO,
\\
u(t,x)=0,\quad &(t,x)\in (0,T)\times \overline{\cO}^c
\end{cases}
\end{equation}
has a unique solution $u\in e_{(0,T)\times \cO}\frH_{p,\theta}^{\gamma+\alpha}(\cO,T)$, and for this solution we have
\begin{align} \label{ineq zero para}
\|r_{(0,T)\times \cO}u\|_{\frH_{p,\theta}^{\gamma+\alpha}(\cO,T)} \leq N \|\psi^{\alpha/2} f\|_{\bH_{p,\theta}^\gamma(\cO,T)}.
\end{align}
\end{lem}

\begin{proof}
 \textbf{1}: Existence and estimate.
First, assume $f\in C_c^\infty((0,T)\times \cO)$. Recall $C_c^\infty((0,T)\times \cO) \subset \psi^{-\alpha/2}\bH_{p,\theta}^{\gamma'}(\cO,T)$ for any $\gamma'\in \bR$.  

Let $\gamma \geq 0$ and
$$
u:=\cT_{\cO} f:= \int_0^t \int_{\cO} p_{\cO}(t-s,x,y) f(s,y) dyds.
$$ 
Then, by \cite[Theorems 2.2 and 2.9]{Dirichlet}, $u$ is a solution to \eqref{eqn. zero para} such that $u\in e_{(0,T)\times \cO}\frH_{p,\theta}^{\gamma+\alpha}(\cO,T)$, and \eqref{ineq zero para} also holds.

Also, due to the existence result for $\gamma=0$ (recall $f\in C_c^\infty((0,T)\times \cO)$) and the embedding $\frH_{p,\theta}^{\alpha}(\cO,T) \subset \frH_{p,\theta}^{\gamma+\alpha}(\cO,T)$ if $\gamma<0$, we only need to prove estimate \eqref{ineq zero para} for $\gamma<0$.

Let $\gamma<0$.  For  $g\in C_c^\infty((0,T)\times \cO)$, put $v:=\cT_{\cO} g$. Then by Fubini's theorem,
\begin{align*}
&\int_0^T \int_{\cO} r_{(0,T)\times \cO}u(t,x)g(t,x)dxdt 
\\
&= \int_0^T\int_{\cO} \left( \int_0^s \int_{\cO} p_{\cO}(s-t,x,y) g(T-t,x) dxdt \right) f(T-s,y) dyds.
\end{align*}
Hence, if $\gamma<-\alpha$, then by Lemma \ref{lem space}$(iii)$ and  \eqref{ineq zero para} with $-\gamma-\alpha$ instead of $\gamma$,
\begin{align*}
&\left| \int_0^T \int_{\cO} r_{(0,T)\times \cO}u(t,x)g(t,x)dxdt \right| 
\\
&\leq N \|f\|_{\bH_{p,\theta+\alpha p/2}^\gamma(\cO,T)} \| r_{(0,T)\times \cO}\cT_{\cO} g_T\|_{\bH_{p',\theta'-\alpha p'/2}^{-\gamma}(\cO,T)}
\\
&\leq N \|\psi^{\alpha/2}f\|_{\bH_{p,\theta}^\gamma(\cO,T)} \|g\|_{\bH_{p',\theta'+\alpha p'/2}^{-\gamma-\alpha}(\cO,T)},
\end{align*}
where $g_T(t,x):=g(T-t,x)$, $1/p+1/p'=1$ and $\theta/p+\theta'/p' = d$.
Since $C_c^\infty((0,T)\times \cO)$ is dense in $\bH_{p',\theta'+\alpha p'/2}^{-\gamma-\alpha}(\cO,T)$ (see \cite[Remark 5.5]{Krylovhalf}), if $\gamma<-\alpha$, then
\begin{align} \label{ineq. 1008-1}
\|r_{(0,T)\times \cO}u\|_{\bH_{p,\theta-\alpha p/2}^{\gamma+\alpha}(\cO,T)} \leq N \|f\|_{\bH_{p,\theta+\alpha p/2}^\gamma(\cO,T)}.
\end{align}
Now we use \cite[Proposition 2.4]{Lototsky} and the complex interpolation of operators (see e.g. \cite[Theorem C.2.6]{Veraar}) to conclude that  \eqref{ineq. 1008-1} holds for all $\gamma \in \bR$. Therefore \eqref{ineq zero para} follows from \eqref{ineq. 1008-1} and \eqref{ineq 1025-1}.

For general $f\in \psi^{-\alpha/2}\bH_{p,\theta}^\gamma(\cO,T)$, we take a sequence $f_n\in C_c^\infty((0,T)\times \cO)$ such that $f_n\to f$ in $\psi^{-\alpha/2}\bH_{p,\theta}^\gamma(\cO,T)$. For each $n$, denote $u_n:=\cT_{\cO} f_n$. Then, 
\begin{eqnarray} \label{ineq. 1007-1} 
\|r_{(0,T)\times \cO}u_n\|_{\frH_{p,\theta}^{\gamma+\alpha}(\cO,T)} &\leq& N \|\psi^{\alpha/2}f_n\|_{\bH_{p,\theta}^\gamma(\cO,T)} 
\\
\|r_{(0,T)\times \cO}(u_n-u_m)\|_{\frH_{p,\theta}^{\gamma+\alpha}(\cO,T)} &\leq& N \|\psi^{\alpha/2}(f_n-f_m)\|_{\bH_{p,\theta}^\gamma(\cO,T)} \nonumber
\end{eqnarray}
Therefore, $r_{(0,T)\times \cO}u_n$ is Cauchy in $\frH_{p,\theta}^{\gamma+\alpha}(\cO,T)$, which converges to a function $r_{(0,T)\times \cO}u\in \frH_{p,\theta}^{\gamma+\alpha}(\cO,T)$.
We conclude that $u$ (extended to $\overline{\cO}^c$ by $0$) is a weak solution to \eqref{eqn. zero para}, and  \eqref{ineq zero para} also follows from the estimates of $u_n$. We remark that we actually have the representation $u=\cT_{\cO} f$.

\textbf{2}: Uniqueness. Let $u\in e_{(0,T)\times \cO}\frH_{p,\theta}^{\gamma+\alpha}(\cO,T)$ be a solution to \eqref{eqn. zero para} with $f=0$. We will show $u=0$. 
Following the argument in \cite[Remark 5.5]{Krylovhalf} (see also \cite[Remark 2.8]{Dirichlet}), we can take $u_n\in C_c^\infty([0,T]\times \cO)$ so that $u_n(0,\cdot)=0$ and $u_n\to r_{(0,T)\times \cO}u$ in $\frH_{p,\theta}^{\gamma+\alpha}(\cO,T)$.
Let $f_n:=\partial_t u_n - \Delta^{\alpha/2}u_n$. Then, by Lemma \ref{thm frac def}, $\Delta^{\alpha/2}(e_{(0,T)\times \cO}u_n)\to \Delta^{\alpha/2}u$ in $\psi^{-\alpha/2}\bH_{p,\theta}^\gamma(\cO,T)$ and thus
$$
f_n\to \partial_t u - \Delta^{\alpha/2}u =0
$$
as $n\to \infty$ in $\psi^{-\alpha/2}\bH_{p,\theta}^\gamma(\cO,T)$.
Due to \cite[Lemma 3.2$(ii)$]{Dirichlet}, we have $u_n=\cT_{\cO} f_n$,
and therefore $u_n$ satisfies \eqref{ineq. 1007-1}.  Thus we conclude  $u=0$. The lemma is proved.
\end{proof}

\begin{lem} \label{lem zero initial}
Let $p\in(1,\infty)$, $\gamma\in\bR$ and $\theta\in(d-1,d-1+p)$. Then, for any $u_0\in \psi^{\alpha/2-\alpha/p} B_{p,\theta}^{\gamma+\alpha-\alpha/p}(\cO)$, the equation
\begin{equation} \label{eqn. zero para ini}
\begin{cases}
\partial_t u = \Delta^{\alpha/2}u,\quad &(t,x)\in(0,T)\times \cO,
\\
u(0,x)=u_0,\quad & x\in \cO,
\\
u(t,x)=0,\quad &(t,x)\in(0,T)\times \overline{\cO}^c.
\end{cases}
\end{equation}
has a unique weak solution $u\in e_{(0,T)\times \cO}\frH_{p,\theta}^{\gamma+\alpha}(\cO,T)$, and for this solution we have
\begin{align} \label{ineq zero ini}
\|r_{(0,T)\times \cO}u\|_{\frH_{p,\theta}^{\gamma+\alpha}(\cO,T)} \leq N \|\psi^{-\alpha/2+\alpha/p}u_0\|_{B_{p,\theta}^{\gamma+\alpha-\alpha/p}(\cO)}.
\end{align}
\end{lem}

\begin{proof}
If $\gamma\geq0$, then all the claims follow from \cite[Theorem 2.9]{Dirichlet}. Also, for any $\gamma\in \bR$, the uniqueness  follows from Lemma \ref{lem zero para}. Thus, we only need to prove the existence result and estimate \eqref{ineq zero ini} when $\gamma<0$.

Let $-1\leq\gamma<0$. Then, by Lemma \ref{negative rep}, there are $u_0^i\in \psi^{\alpha/2-\alpha/p} B_{p,\theta}^{\gamma+1+\alpha-\alpha/p}(\cO)$ $(i=0,1,\dots,d)$ such that $u_0=u_{0}^0+ \sum_{i=1}^d D_i \left(\psi u_{0}^i\right)$
and
\begin{align} \label{ineq 1025-8}
\|\psi^{-\alpha/2+\alpha/p} u_0\|_{B_{p,\theta}^{\gamma+\alpha-\alpha/p}(\cO)} \approx \sum_{i=0}^d\|\psi^{-\alpha/2+\alpha/p} u_0^i\|_{B_{p,\theta}^{\gamma+1+\alpha-\alpha/p}(\cO)}.
\end{align}
Since $\gamma+1\geq0$, by the result for $\gamma\geq0$, for each $i$ there exists a unique solution $v_i$ to \eqref{eqn. zero para ini} with $u_0^i$ in place of $u_0$. Moreover, for each $i$,
\begin{align*} 
\|r_{(0,T)\times \cO}v_i\|_{\frH_{p,\theta}^{\gamma+1+\alpha}(\cO,T)} \leq N \|\psi^{-\alpha/2+\alpha/p}u_0^i\|_{B_{p,\theta}^{\gamma+1+\alpha-\alpha/p}(\cO)}.
\end{align*}
By Lemma \ref{lem deriv}, if we denote $v:=v_0+ \sum_{i=1}^d D_i \left(\psi v_i\right)$, then $v \in e_{(0,T)\times \cO}\frH_{p,\theta}^{\gamma+\alpha}(\cO,T)$ and $v$ is a solution to
\begin{equation*}
\begin{cases}
\partial_t v = \Delta^{\alpha/2}v+ h,\quad &(t,x)\in(0,T)\times \cO,
\\
v(0,x)=u_0,\quad & x\in \cO,
\\
v(t,x)=0,\quad &(t,x)\in(0,T)\times \overline{\cO}^c,
\end{cases}
\end{equation*}
where $h=\sum_{i=1}^d D_i\left(\psi\Delta^{\alpha/2}v_i\right) - \sum_{i=1}^d \Delta^{\alpha/2}\left(D_i\psi v_i\right)$ in $(0,T)\times\cO$.
Here, by Lemmas \ref{thm frac def} and  \ref{lem deriv},  we have $h\in \psi^{-\alpha/2}\bH_{p,\theta}^{\gamma}(\cO,T)$. Thus, by Lemma \ref{lem zero para}, there exists a unique solution $w$ to
\begin{equation*}
\begin{cases}
\partial_t w = \Delta^{\alpha/2}w + h,\quad &(t,x)\in(0,T)\times \cO,
\\
w(0,x)=0,\quad & x\in \cO,
\\
w(t,x)=0,\quad &(t,x)\in(0,T)\times \overline{\cO}^c.
\end{cases}
\end{equation*}
Also, $w$ satisfies 
\begin{align} \label{ineq 1025-9}
\|r_{(0,T)\times \cO}w\|_{\frH_{p,\theta}^{\gamma+\alpha}(\cO,T)} &\leq N \|\psi^{\alpha/2} h\|_{\bH_{p,\theta}^\gamma(\cO,T)} \nonumber
\\
&\leq N \sum_{i=1}^d \|\psi^{-\alpha/2} r_{(0,T)\times \cO}v_i \|_{\bH_{p,\theta}^{\gamma+1+\alpha}(\cO,T)}.
\end{align}
Therefore,  $u:=v-w$ becomes a solution to \eqref{eqn. zero para ini}, and \eqref{ineq zero ini} also follows from \eqref{ineq 1025-8}--\eqref{ineq 1025-9}.

Repeating the above argument, one can treat the case $-(n+1)\leq \gamma<-n$ for $n=1,2,\cdots$ in order, and prove the lemma for all $\gamma<0$.  The lemma is proved.
\end{proof}

\begin{lem} \label{lem zero ellip}
Let $p\in(1,\infty)$, $\gamma\in\bR$ and $\theta\in(d-1,d-1+p)$. Then, for any $f\in \psi^{-\alpha/2}H_{p,\theta}^\gamma(\cO)$, the equation
\begin{equation} \label{eqn. zero ellip}
\begin{cases}
\Delta^{\alpha/2}u = f,\quad &x \in \cO,
\\
u(x)=0,\quad &x \in \overline{\cO}^c.
\end{cases}
\end{equation}
has a unique solution $u \in e_{\cO}H_{p,\theta-\alpha p/2}^{\gamma+\alpha}(\cO)$, and for this solution we have
\begin{align} \label{ineq zero ellip}
\|\psi^{-\alpha/2} r_{\cO}u \|_{H_{p,\theta}^{\gamma+\alpha}(\cO)} \leq N \|\psi^{\alpha/2} f\|_{H_{p,\theta}^\gamma(\cO)}.
\end{align}
\end{lem}

\begin{proof}
We first prove the following:  if  $u \in e_{\cO}H_{p,\theta-\alpha p/2}^{\gamma+\alpha}(\cO)$, then
\begin{align} \label{ineq. 1008}
\|\psi^{-\alpha/2} r_{\cO}u\|_{H_{p,\theta}^{\gamma+\alpha}(\cO)} \leq N \|\psi^{\alpha/2} \Delta^{\alpha/2} u\|_{H_{p,\theta}^{\gamma}(\cO)}.
\end{align}
 We only prove the result when $\gamma<0$, because this inequality follows from \cite[Theorem 2.10]{Dirichlet} if $\gamma\geq0$.  We note that  it is enough to treat the case $\gamma<-\alpha$. Indeed, once \eqref{ineq. 1008} is proved for $\gamma<-\alpha$, then we can use the complex interpolation of operators to prove \eqref{ineq. 1008}  for all $\gamma \in \bR$.

Now we assume $\gamma<-\alpha$. By Lemma \ref{lem space}$(i)$, it suffices to prove \eqref{ineq. 1008} for $u\in C_c^\infty(\cO)$. Let $g\in C_c^\infty(\cO)$. Since $-\gamma>\alpha$, one can find a solution $v\in e_{\cO}H_{p',\theta'-\alpha p'/2}^{-\gamma}(\cO)$ to \eqref{eqn. zero ellip} with $g$ in place of $f$. Thus, by Lemma \ref{lem space}$(iv)$,
\begin{align*}
|(r_{\cO}u,g)_{\cO}| &= |(\Delta^{\alpha/2}u,r_{\cO}v)_{\cO}| \leq N \|\psi^{\alpha/2}\Delta^{\alpha/2} u\|_{H_{p,\theta}^\gamma(\cO)} \|\psi^{-\alpha/2} r_{\cO}v\|_{H_{p',\theta'}^{-\gamma}(\cO)}
\\
&\leq N \|\psi^{\alpha/2}\Delta^{\alpha/2} u\|_{H_{p,\theta}^\gamma(\cO)} \|\psi^{\alpha/2} g\|_{H_{p',\theta'}^{-\gamma-\alpha}(\cO)}.
\end{align*}
Here, $p'$ and  $\theta'$ are defined as in the proof of Lemma \ref{lem zero para}, and for the last inequality we used \eqref{ineq zero ellip} with $g$ and $-\gamma-\alpha$ instead of $f$ and $\gamma$, respectively. Therefore, we have \eqref{ineq. 1008} for $\gamma<-\alpha$ due to Lemma \ref{lem space}$(iii)$.

Now we are ready to prove our main statement. The uniqueness and \eqref{ineq zero ellip} easily follow from  \eqref{ineq. 1008}. Thus, it only remains to prove the existence.
Since the case $\gamma\geq0$ is treated in \cite[Theorem 2.10]{Dirichlet}, we assume $\gamma<0$.
Let $f\in C_c^\infty(\cO)$. Then, we have a solution $u \in e_{\cO}H_{p,\theta-\alpha p/2}^{\alpha}(\cO)$ to \eqref{eqn. zero ellip}. Here, due to $\psi^{\alpha/2}H_{p,\theta}^{\alpha}(\cO) \subset \psi^{\alpha/2}H_{p,\theta}^{\gamma+\alpha}(\cO)$, we can apply \eqref{ineq. 1008} to get \eqref{ineq zero ellip}. For general data $f$,  we use the approximation argument used in the proof of Lemma \ref{lem zero para}.  The lemma is proved.
\end{proof}

Now we consider
\begin{equation} \label{parabolic sec 4}
\begin{cases}
\partial_t u(t,x)=\Delta^{\alpha/2}u(t,x),\quad &(t,x)\in(0,T)\times \cO,
\\
u(0,x)=0,\quad & x\in \cO,
\\
u(t,x)=g(t,x),\quad &(t,x)\in (0,T)\times \overline{\cO}^c,
\end{cases}
\end{equation}
and
\begin{equation} \label{elliptic sec 4}
\begin{cases}
\Delta^{\alpha/2}u(x)=0,\quad &x\in \cO,\\
u(x)=g(x),\quad & x\in \overline{\cO}^c.
\end{cases}
\end{equation}

\begin{lem}[Higher regularity for parabolic equation] \label{higher para}
 Let $0\leq \mu$, $\theta\in (d-1-\frac{\alpha p}{2},d-1+p+\frac{\alpha p}{2})$ and $g\in \psi^{\alpha/2}\bL_{p,\theta,\sigma}(\overline{\cO}^c,T)$.  Suppose $\sigma>-\theta-\alpha p/2$ if $\cO$ is bounded, and $\sigma=0$ if $\cO$ is a half space. Then, for any solution $u \in e_{(0,T)\times\cO}\frH_{p,\theta}^\mu(\cO,T) \oplus e_{(0,T)\times \overline{\cO}^c}\bL_{p,\theta-\alpha p/2,\sigma}(\overline{\cO}^c,T)$ to \eqref{parabolic sec 4}, we have $u\in e_{(0,T)\times \cO}\frH_{p,\theta}^\gamma(\cO,T)$ for any $\gamma\geq 0$, and moreover
\begin{align} \label{ineq. 08.01-8}
&\|\psi^{-\alpha/2}r_{(0,T)\times \cO}u\|_{\bH_{p,\theta}^{\gamma}(\cO,T)} \nonumber
\\
&\leq N \|\psi^{-\alpha/2}r_{(0,T)\times \cO}u\|_{\bH_{p,\theta}^{\mu}(\cO,T)} + N \|\psi^{-\alpha/2}g\|_{\bL_{p,\theta,\sigma}(\overline{\cO}^c,T)},
\end{align}
where $N=N(d,p,\alpha,\gamma,\mu, \theta,\sigma,\cO)$.
\end{lem}

\begin{proof}
$(i)$ We first note that it suffices to consider the case $\gamma=\mu+\alpha/2$.  Indeed, if the claim holds for  $\gamma=\mu+\alpha/2$, then  repeating the result with $\mu'=\mu+\alpha/2, \mu+2\alpha/2, \cdots$ in order, we prove the lemma for $\gamma=\mu+k \alpha/2$, $k\in \bN$. This certainly proves the lemma.

Let $\gamma=\mu+\alpha/2$. For each $n\in\bZ$, denote $u_n(t,x):=u(e^{n\alpha}t, e^n x)$ and $g_n(t,x) := g(e^{n\alpha}t,e^n x)$.
Then, $u_n(t,x)\zeta_{-n}(e^nx)\in \bH_p^\mu(e^{-n\alpha}T)$ is a weak solution to the equation
$$
\begin{cases}
\frac{\partial v}{\partial t}=\Delta^{\alpha/2} v+F_n,   \quad &: (t,x)\in  (0,e^{-n\alpha}T)\times \bR^d
\\
v(0,\cdot)=0 \quad &:  x\in \bR^d
\end{cases}
$$
where $F_n(t,x)= -\Delta^{\alpha/2}(u_n(\cdot,\cdot)\zeta_{-n}(e^n\cdot))(t,x) + \zeta_{-n}(e^nx) \Delta^{\alpha/2}u_n(t,x)$.
By Lemma \ref{lem perturb},  we have (note $\gamma-\alpha=\mu-\alpha/2$)
\begin{align} \label{ineq. 08.01-5}
&\sum_{n\in\bZ}e^{n(\theta-\alpha p/2)}\|F_n(e^{-n\alpha}t,\cdot)\|_{H_p^{\gamma-\alpha}}^p = \sum_{n\in\bZ}e^{n(\theta-\alpha p/2)}\|F_n(e^{-n\alpha}t,\cdot)\|_{H_p^{\mu-\alpha/2}}^p \nonumber
    \\
    &\leq N \|\psi^{-\alpha/2}r_{(0,T)\times \cO}u(t,\cdot)\|_{H_{p,\theta}^{\mu}(\cO)}^p 
     + N \|\psi^{-\alpha/2}g(t,\cdot)\|_{L_{p,\theta,\sigma}(\overline{\cO}^c)}^p. 
     \end{align}
Thus, we have $F_n\in \bH_p^{\gamma-\alpha}(e^{-n\alpha}T):=L_p((0,e^{-n\alpha}T);H_p^{\gamma-\alpha})$. Now we can apply \cite[Theorem 1]{Mikul Cauchy} to conclude $u_n(\cdot,\cdot)\zeta_{-n}(e^n\cdot)\in \bH_p^{\gamma}(e^{-n\alpha}T)$ and
\begin{align} 
\|\Delta^{\alpha/2}(u(\cdot,e^n\cdot)\zeta_{-n}(e^n\cdot))\|_{\bH_p^{\gamma-\alpha}(T)}^p &=e^{n\alpha}\|\Delta^{\alpha/2}(u_n(\cdot,\cdot)\zeta_{-n}(e^n\cdot))\|_{\bH_p^{\gamma-\alpha}(e^{-n\alpha}T)}^p   \nonumber
\\
&\leq N e^{n\alpha}\|F_{n}(\cdot,\cdot)\|_{\bH_p^{\gamma-\alpha}( e^{-n\alpha}T)}^p \nonumber
\\
&= N \|F_{n}(e^{-n\alpha}\cdot,\cdot)\|_{\bH_p^{\gamma-\alpha}(T)}^p.  \label{eqn 3.19.6}
\end{align}
By the relation
\begin{equation}
\label{frac rel}
\|v\|_{H_p^\gamma} \approx \left(\|v\|_{H_p^{\gamma-\alpha}}+\|\Delta^{\alpha/2}v\|_{H_p^{\gamma-\alpha}}\right),
\end{equation}
for any  $w\in H^{\gamma}_{p,\theta-\alpha p/2}(\cO)$ we have
\begin{align*}
&\|w\|^p_{H^{\gamma}_{p,\theta-\alpha p/2}(\cO)}  \nonumber 
\\
& \leq N \sum_n e^{n(\theta-\alpha p/2)} \left(\|w(e^n\cdot)\zeta_{-n}(e^n\cdot)\|^p_{H^{\mu}_p}+ \|\Delta^{\alpha/2}(w(e^n\cdot)\zeta_{-n}(e^n\cdot))\|^p_{H_p^{\gamma-\alpha}}\right) \nonumber 
\\
&\leq N \|w\|^p_{H^{\mu}_{p,\theta-\alpha p/2}(\cO)}+ N \sum_n e^{n(\theta-\alpha p/2)}  \|\Delta^{\alpha/2}(w(e^n\cdot)\zeta_{-n}(e^n\cdot))\|^p_{H_p^{\gamma-\alpha}}. 
\end{align*}
This together with \eqref{ineq. 08.01-5} and \eqref{eqn 3.19.6} easily yields \eqref{ineq. 08.01-8}. The lemma is proved.

\end{proof}

\begin{lem}[Higher regularity for elliptic equation] \label{higher ellip}
Let $0\leq \mu$,  and $\theta\in (d-1-\frac{\alpha p}{2},d-1+p+\frac{\alpha p}{2})$ and $g\in \psi^{\alpha/2}L_{p,\theta,\sigma}(\overline{\cO}^c)$.  Let $\sigma>-\theta-\alpha p/2$ if $\cO$ is bounded, and $\sigma=0$ if $\cO$ is a half  space. 
 Then, for a solution $u\in e_{\cO}H^{\mu}_{p,\theta-\alpha p/2}(\cO)\oplus e_{\overline{\cO}^c}L_{p,\theta-\alpha p/2,\sigma}(\overline{\cO}^c)$ to \eqref{elliptic sec 4}, we have $u\in e_{\cO}H_{p,\theta-\alpha p/2}^{\gamma}(\cO)$ for any $\gamma \in \bR$, and moreover
\begin{align} \label{ineq. 08.01-9}
&\|\psi^{-\alpha/2}r_{\cO}u\|_{H_{p,\theta}^{\gamma}(\cO)} \leq N \|\psi^{-\alpha/2}g\|_{L_{p,\theta,\sigma}(\overline{\cO}^c)} + N \|\psi^{-\alpha/2}r_{\cO}u\|_{H_{p,\theta}^{\mu}(\cO)},
\end{align}
where $N=N(d,p,\alpha,\gamma,\mu, \theta,\sigma,\cO)$. 
\end{lem}

\begin{proof}
We repeat the argument in the proof of Lemma \ref{higher para}. 
As before, we assume $\gamma =\mu+\alpha/2$.

 Since $u$ is a weak solution to \eqref{elliptic sec 4},  for any $n\in \bZ$, $u_n(x):=u(e^n x)$ satisfies
\begin{equation*} 
\Delta^{\alpha/2}(u_n(\cdot)\zeta_{-n}(e^n\cdot))(x) = F_n(x),\quad x\in \bR^d,
\end{equation*}
where $F_n(x) = - \Delta^{\alpha/2}(u_n(\cdot)\zeta_{-n}(e^n\cdot))(x) + \zeta_{-n}(e^nx) \Delta^{\alpha/2}u_n(x)$.
   By Lemma \ref{lem perturb},  
\begin{align*} 
   &\sum_{n\in\bZ}e^{n(\theta-\alpha p/2)}\|\Delta^{\alpha/2}(u_n(\cdot)\zeta_{-n}(e^n\cdot))\|_{H_p^{\gamma-\alpha}}^p =\sum_{n\in\bZ}e^{n(\theta-\alpha p/2)}\|F_n\|_{H_p^{\mu-\alpha/2}}^p 
    \\
    &\leq N\|\psi^{-\alpha/2}r_{\cO}u\|_{H_{p,\theta}^{\mu}(\cO)}^p + N \|\psi^{-\alpha/2}g\|_{L_{p,\theta,\sigma}(\overline{\cO}^c)}^p.
\end{align*}
This and \eqref{frac rel} prove \eqref{ineq. 08.01-9}. 
 The lemma is proved.
\end{proof}

\vspace{1em}
\textbf{Proof of Theorem \ref{main thm para}}

\begin{proof} The uniqueness follows from Lemma \ref{lem zero para}.   

Suppose that Theorem \ref{main thm para} holds when $u_0=0$ and $f=0$, and let $u_1$ denote the  solution  in this case. Also, let $u_2$ and $u_3$ denote the solution  from Lemma \ref{lem zero para} and Lemma \ref{lem zero initial}, respectively. Then, $u:=u_1+u_2+u_3$ becomes a solution to \eqref{para def} satisfying \eqref{est main para}. Thus, the theorem follows from the following lemma.  \end{proof}

\begin{lem} \label{thm neg ext}
Let  $\gamma,\lambda\in\bR$, $\theta\in(d-1,d-1+p)$ and $g\in \psi^{-\alpha/2}\bH_{p,\theta,\sigma}^\lambda(\overline{\cO}^c,T)$. Suppose $\sigma>-\theta-\alpha p/2$ if $\cO$ is bounded, and $\sigma=0$ if $\cO$ is a half space. Then, equation \eqref{parabolic sec 4} has a solution $u \in e_{(0,T)\times\cO}\frH^{\gamma+\alpha}_{p,\theta}(\cO,T) \oplus e_{(0,T)\times\overline{\cO}^c}\bH^{\lambda}_{p,\theta-\alpha p/2,\sigma}(\overline{\cO}^c,T)$ and
\begin{align} \label{ineq ext para}
\|r_{(0,T)\times \cO}u\|_{\frH_{p,\theta}^{\gamma+\alpha}(\cO,T)} \leq N \|\psi^{-\alpha/2} g\|_{\bH_{p,\theta,\sigma}^\lambda(\overline{\cO}^c,T)}.
\end{align}
\end{lem}

\begin{proof}
Since $\frH_{p,\theta}^{\gamma_1}(\cO,T) \subset \frH_{p,\theta}^{\gamma_2}(\cO,T)$ for $\gamma_1 \geq \gamma_2$, we only need to prove the lemma for $\gamma\geq 0$. 

\vspace{1mm}
{\textbf{Case 1 $\lambda\geq0$.}}
It suffices to assume $\lambda=0$. If $g \in C_c^\infty((0,T)\times \overline{\cO}^c)$, then all the claims follow from Lemmas \ref{lem para pp}, \ref{lem zeroth para}, \ref{lem prob para} and \ref{higher para}. For general $g$, one can use the standard approximation argument  as in the proof of Lemma \ref{lem zero para}. 

{\textbf{Case 2 $-1\leq\lambda<0$.}}
By Lemma \ref{negative rep}, there exist $g_i\in \psi^{\alpha/2}\bH_{p,\theta,\sigma}^{\lambda+1}(\overline{\cO}^c,T) \, (i=0,1,\dots,d)$ such that $g=g_0 + \sum_{i=1}^d D_i \left(\psi g_i\right)$
and
\begin{align} \label{ineq 1025-5}
\|\psi^{-\alpha/2} g\|_{H_{p,\theta}^{\lambda}(\cO)} \approx \sum_{i=0}^d\|\psi^{-\alpha/2} g_i\|_{H_{p,\theta}^{\lambda+1}(\cO)}.
\end{align}
Since $\lambda+1\geq0$, by the result for $\lambda\geq0$, there exists a solution $v_i$ to \eqref{parabolic sec 4} with $g_i$ in place of $g$ for each $i=0,1,\dots,d$, which satisfies \eqref{ineq ext para} for any $\gamma \in \bR$. In particular, for given $\gamma\geq 0$, we have  
\begin{align*}
\|r_{(0,T)\times \cO}v_i\|_{\frH_{p,\theta}^{\gamma+1+\alpha}(\cO,T)} \leq N \|\psi^{-\alpha/2} g_i\|_{\bH_{p,\theta,\sigma}^{\lambda+1}(\cO,T)}.
\end{align*}
 Then, for $v:=v_0+ \sum_{i=1}^d D_i \left(\psi v_i\right)$, by Lemma \ref{lem deriv}, $r_{(0,T)\times \cO}v \in \frH_{p,\theta}^{\gamma+\alpha}(\cO,T)$ and $v$ is a solution to
\begin{equation*}
\begin{cases}
\partial_t u = \Delta^{\alpha/2}u + h,\quad &(t,x)\in(0,T)\times \cO,
\\
u(0,x)=0,\quad & x\in \cO,
\\
u(t,x)=g(t,x),\quad &(t,x)\in(0,T)\times \overline{\cO}^c,
\end{cases}
\end{equation*}
where $h=\sum_{i=1}^d D_i\left(\psi\Delta^{\alpha/2}v_i\right) - \sum_{i=1}^d \Delta^{\alpha/2}\left(D_i\psi v_i\right)$.
Here, Theorem \ref{thm frac def} and Lemma \ref{lem deriv} yield $h\in \psi^{-\alpha/2}\bH_{p,\theta}^{\gamma}(\cO,T)$. Thus, by Lemma \ref{lem zero para}, there exists a unique solution $w$ to
\begin{equation*}
\begin{cases}
\partial_t u = \Delta^{\alpha/2}u + h,\quad &(t,x)\in(0,T)\times \cO,
\\
u(0,x)=0,\quad & x\in \cO,
\\
u(t,x)=0,\quad &(t,x)\in (0,T)\times \overline{\cO}^c.
\end{cases}
\end{equation*}
Also, $w$ satisfies 
\begin{align} \label{ineq 1025-7}
\|r_{(0,T)\times \cO}w\|_{\frH_{p,\theta}^{\gamma+\alpha}(\cO,T)} &\leq N \|\psi^{\alpha/2} h\|_{\bH_{p,\theta}^\gamma(\cO,T)} \nonumber
\\
&\leq N \sum_{i=1}^d \|\psi^{-\alpha/2} r_{(0,T)\times \cO}v_i\|_{\bH_{p,\theta}^{\gamma+1+\alpha}(\cO,T)}.
\end{align}
Then, due to \eqref{ineq 1025-5}--\eqref{ineq 1025-7}, $u:=v-w$ is a solution to \eqref{parabolic sec 4} satisfying \eqref{ineq ext para}.

{\textbf{Case 3 $\lambda<-1$.}}
One can repeat the argument in Case 2 to treat the case $\gamma<-1$.

The lemma is proved.
\end{proof}

 The following lemma is a version of Lemma \ref{thm neg ext}. The proof is similar to the one of Lemma \ref{thm neg ext} and thus omitted here.

\begin{lem}
Let  $\gamma,\lambda\in\bR$,   $\theta\in(d-1,d-1+p)$ and $g\in \psi^{-\alpha/2}H_{p,\theta,\sigma}^\lambda(\overline{\cO}^c)$. Suppose $\sigma>-\theta-\alpha p/2$ if $\cO$ is bounded, and $\sigma=0$ is $\cO$ is a half space. Then, equation \eqref{elliptic sec 4} has a solution $u\in e_{\cO}H^{\gamma+\alpha}_{p,\theta-\alpha p/2}(\cO) \oplus e_{\overline{\cO}^c}H^{\lambda}_{p,\theta-\alpha p/2,\sigma}(\overline{\cO}^c)$, and for this solution we have
\begin{align*}
\|\psi^{-\alpha/2}r_{\cO}u\|_{H_{p,\theta}^{\gamma+\alpha}(\cO)} \leq N \|\psi^{-\alpha/2} g\|_{H_{p,\theta,\sigma}^\lambda(\cO)}.
\end{align*}
\end{lem}

\textbf{Proof of Theorem \ref{main thm ellip}}

\begin{proof} 
It suffices to repeat the proof of Theorem \ref{main thm para}, replacing the results for parabolic equations by their corresponding elliptic versions.
\end{proof}

\appendix

\section{Auxiliary results}

Recall that $\cO$ is a half space or bounded $C^{1,1}$ open set.
Let $\fO$ be either $\cO$ or $\overline{\cO}^c$.

\begin{lem}
Let $d\geq2$ and $\alpha\in(0,2)$. Then,
\begin{align} \label{070820}
\int_{\bR^{d-1}} \left( 1\wedge |x|^{-d-\alpha} \right) dx' = \int_{\bR^{d-1}} \left( 1\wedge |(x^1,x')|^{-d-\alpha} \right) dx' \leq N(\alpha,d) \left( 1\wedge |x^1|^{-1-\alpha} \right).
\end{align}
\end{lem}

\begin{proof}
By the change of variables,
\begin{align*}
\int_{\bR^{d-1}} \left( 1\wedge |x|^{-d-\alpha} \right) dx' = N(d) |x^1|^{d-1} \int_{\bR^{d-1}} \left( 1\wedge |x^1|^{-d-\alpha} |(1,x')|^{-d-\alpha} \right) dx'.
\end{align*}
Here, if $|x^1|>1$, then 
\begin{align*}
&\int_{\bR^{d-1}} \left( 1\wedge |x^1|^{-d-\alpha} |(1,x')|^{-d-\alpha} \right) dx' 
\\
&= |x^1|^{-d-\alpha} \int_{\bR^{d-1}} |(1,x')|^{-d-\alpha} dx' \leq N(\alpha,d) |x^1|^{-d-\alpha}.
\end{align*}
Thus, we have \eqref{070820} when $|x^1|>1$. Next, assume $|x^1| \leq 1$. Then,
\begin{align*}
&|x^1|^{d-1}\int_{\bR^{d-1}} \left( 1\wedge |x^1|^{-d-\alpha} |(1,x')|^{-d-\alpha} \right) dx' 
\\
&= |x^1|^{d-1}\int_{|x'|\leq(|x^1|^{-2}-1)^{1/2}} dx'+ |x^1|^{-1-\alpha}\int_{|x'|>(|x^1|^{-2}-1)^{1/2}} |(1,x')|^{-d-\alpha} dx'
\\
&\leq N(d) (1-|x^1|^2)^{(d-1)/2} + |x^1|^{-1-\alpha}\int_{|x'|>(|x^1|^{-2}-1)^{1/2}} |(1,x')|^{-d-\alpha} dx'.
\end{align*}
Here, one can easily show that the last two terms are bounded. Thus, we get \eqref{070820} and the lemma is proved.

\end{proof}

\begin{lem} \label{aux para half}
Let $\cO$ be a half space.
Suppose that $\nu_0,\nu_1 \in \bR$ satisfy
\begin{align} \label{aux half assu}
\nu_0+\nu_1>-\frac{2}{\alpha}, \quad 2>\nu_1 
 \neq -\frac{2}{\alpha}.
\end{align}
Then, for any $(t,x)\in (0,\infty)\times \fO$,
\begin{align} \label{aux half}
&\int_{\overline{\fO}^c} \left( t^{-d/\alpha-1} \wedge |x-z|^{-d-\alpha} \right) \left( 1 \wedge \frac{d_z^{\alpha/2}}{\sqrt{t}} \right)^{\nu_0} d_z^{\nu_1\alpha/2} dz  \nonumber
\\
&\leq N \left( t^{\nu_1/2-1} \wedge  \left( d_x^{\nu_1\alpha/2-\alpha} \vee t^{\nu_1/2+1/\alpha}d_x^{-1-\alpha} \right) \right),
\end{align}
where $N=N(d,\alpha,\nu_0,\nu_1)$.
Moreover, if $\cO$ is bounded, the same result holds for $x\in \fO$ such that $d_x\leq 2diam(\cO)$ with a constant $N$ depend also on $\cO$.
\end{lem}

\begin{proof}
{\bf{1.}} $\cO$ is a half space.

In this case, it is enough to prove \eqref{aux half} when $\fO=\cO$, $t=1$ and $d=1$. Indeed, by the change of variables,
\begin{align*}
&\int_{\overline{\fO}^c} \left( t^{-d/\alpha-1} \wedge |x-z|^{-d-\alpha} \right) \left( 1 \wedge \frac{|z^1|^{\alpha/2}}{\sqrt{t}} \right)^{\nu_0} |z^1|^{\nu_1\alpha/2} dz
\\
&= t^{d/\alpha+\nu_1/2} \int_{\overline{\fO}^c} \left( t^{-d/\alpha-1} \wedge |x-t^{1/\alpha}z|^{-d-\alpha} \right) \left( 1 \wedge |z^1|^{\alpha/2} \right)^{\nu_0} |z^1|^{\nu_1\alpha/2} dz
\\
&= t^{\nu_1/2-1} \int_{\overline{\fO}^c} \left( 1 \wedge |t^{-1/\alpha}x-z|^{-d-\alpha} \right) \left( 1 \wedge |z^1|^{\alpha/2} \right)^{\nu_0} |z^1|^{\nu_1\alpha/2} dz
\\
&\leq N t^{\nu_1/2-1} \left( 1\wedge \left( |t^{-1/\alpha}x^1|^{\nu_1 \alpha/2-\alpha} \vee |t^{-1/\alpha}x^1|^{-1-\alpha} \right) \right) 
\\
&= N \left( t^{\nu_1/2-1} \wedge  \left( |x^1|^{\nu_1\alpha/2} \vee t^{\nu_1/2+1/\alpha}|x^1|^{-1-\alpha} \right) \right).
\end{align*} 
Therefore, we may assume $t=1$. Moreover, due to \eqref{070820}, we can further assume $d=1$. That is, we will show that
\begin{align} \label{aux half one}
\int_{-\infty}^0 \left( 1\wedge |x-z|^{-1-\alpha} \right) \left(1\wedge |z|^{\alpha/2} \right)^{\nu_0} |z|^{\nu_1\alpha/2} dz \leq N \left(1\wedge \left( |x|^{\nu_1\alpha/2-\alpha} \vee |x|^{-1-\alpha} \right) \right).
\end{align}

Let $x\in (0,1)$. Since $1\wedge |x-z|^{-1-\alpha} \leq 1\wedge |z|^{-1-\alpha}$ for $z\in(-\infty,0)$, we get
\begin{align*}
&\int_{-\infty}^0 \left( 1\wedge |x-z|^{-1-\alpha} \right) \left(1\wedge |z|^{\alpha/2} \right)^{\nu_0} |z|^{\nu_1\alpha/2} dz 
\\
&\leq \int_{-1}^0 |z|^{(\nu_0+\nu_1)\alpha/2} dz  + \int_{-\infty}^{-1} |z|^{-1+\nu_1\alpha/2-\alpha} dz.
\end{align*}
Here, the last two terms are finite due to \eqref{aux half assu}. Thus, we have \eqref{aux half one} when $x\in(0,1)$.

Next, let $x\in[1,\infty)$. For $z\in(-1,0)$, $|x-z|\approx |x|$. Thus,
\begin{align} \label{070914-1}
&\int_{-1}^0 \left( 1\wedge |x-z|^{-1-\alpha} \right) \left(1\wedge |z|^{\alpha/2} \right)^{\nu_0} |z|^{\nu_1\alpha/2} dz \nonumber
\\
&\leq N|x|^{-1-\alpha} \int_{-1}^0 |z|^{(\nu_0+\nu_1)\alpha/2} dz \leq N |x|^{-1-\alpha} \leq N \left( |x|^{-1-\alpha} \vee |x|^{\nu_1\alpha/2-\alpha} \right).
\end{align}
By the change of variables,
\begin{align} \label{070914-2}
&\int_{-\infty}^{-1} \left( 1\wedge |x-z|^{-1-\alpha} \right) \left(1\wedge |z|^{\alpha/2} \right)^{\nu_0} |z|^{\nu_1\alpha/2} dz \nonumber
\\
&\leq |x|^{\nu_1\alpha/2-\alpha} \int_{-\infty}^{-|x|^{-1}} |1-z|^{-1-\alpha} |z|^{\nu_1\alpha/2} dz \nonumber
\\
&\leq N |x|^{\nu_1\alpha/2-\alpha}\left(1_{\nu_1>-2/\alpha} + 1_{\nu_1<-2/\alpha}|x|^{-\nu_1\alpha/2-1}\right).
\end{align}
Thus, \eqref{070914-1} and \eqref{070914-2} yield \eqref{aux half one} when $x\in[1,\infty)$.

{\bf{2.}} $\cO$ is bounded and $d_z\leq 2diam(\cO)$.

Since $\cO$ is bounded, for any $R>0$, there exists $x_1,\dots,x_n\in\partial \cO$ such that if $y\in (\cup_{i=1}^n B_{R/2}(x_i))^c$ then $d_y>R/4$. Then,
\begin{align*}
&\int_{\overline{\fO}^c} \left( t^{-d/\alpha-1} \wedge |x-z|^{-d-\alpha} \right) \left( 1 \wedge \frac{d_z^{\alpha/2}}{\sqrt{t}} \right)^{\nu_0} d_z^{\nu_1\alpha/2} dz 
\\
&\leq \sum_{i=1}^n \int_{\overline{\fO}^c \cap B_{R/2}(x_i)} \cdots dz + \int_{\cO_{R/4}} \cdots dz =: \sum_{i=1}^n I_i(t,x) + II(t,x),
\end{align*}
where $\cO_S:=\{z\in \overline{\fO}^c:d_z\geq S\}$.

$(1)$ We estimate $I_i(x)$ for fixed $i\in\{1,2,\dots,n\}$.

Note that by reducing $R$ (if necessary), one can consider a $C^{1,1}$-bijective flattening boundary map $\Phi=(\Phi^1,\cdots,\Phi^d)$ defined on $B_R(x_i)$ such that $\Phi(B_R(x_i)\cap \fO)\subset \bR_{+}^d$ and $d_x\approx \Phi^1(x)$ on $B_R(x_i)\cap \fO$. Thus, for $x\in B_R(x_i) \cap \fO$, using the case $\cO=\bR_+^d$ above, one can easily estimate $I_i$.

Next, assume $x\in \fO\setminus B_R(x_i)$. Then, for any $z\in \overline{\fO}^c\cap B_{R/2}(x_i)$, $|x-z|>R/2$, and one can take $\tilde{x}\in \fO\cap B_R(x_i)$ such that $d_x\leq N(\cO,R) d_{\tilde{x}}$. Thus, for $z\in \overline{\fO}^c\cap B_R(x_i)$,
\begin{equation} \label{eq8152005}
|\tilde{x}-z|\leq |\tilde{x}-x|+|x-z| \leq 3diam(\cO)+R+|x-z| \leq N(\cO,R)|x-z|.
\end{equation}
Hence,
\begin{align*}
I_i(t,x) &\leq NI_i(t,\tilde{x}) \leq N \left( t^{\nu_1/2-1} \wedge  \left( d_{\tilde{x}}^{\nu_1\alpha/2-\alpha} \vee t^{\nu_1/2+1/\alpha}d_{\tilde{x}}^{-1-\alpha} \right) \right)
\\
&\leq N \left( t^{\nu_1/2-1} \wedge  \left( d_x^{\nu_1\alpha/2-\alpha} \vee t^{\nu_1/2+1/\alpha}d_x^{-1-\alpha} \right) \right).
\end{align*}

$(2)$ Next we estimate $II(x)$. 
Note that if $z\in \cO_{R/4}$, then for $y\in \partial \cO$ such that $d_z=|y-z|$,
\begin{align} \label{eq8151621}
d_z\leq |x-z| \leq |x-y|+|y-z| \leq 3diam(\cO)+d_z \leq N(\cO,R) d_z.
\end{align}
That is, $|x-z|\approx d_z$. Thus, there exists $c>0$ such that
\begin{align*}
II(t,x) &= \int_{\cO_{t^{1/\alpha}}} \cdots dz + \int_{\cO_{R/4} \cap \{d_z< t^{1/\alpha}\}} \cdots dz
\\
&\leq N \int_{|x-z|>c(t^{1/\alpha}\vee d_x)} |x-z|^{-d-\alpha+\nu_1\alpha/2} dz
\\
&\quad + N 1_{R/4<t^{1/\alpha}} \int_{|x-z|\leq c t^{1/\alpha}} t^{-d/\alpha-1-\nu_0\alpha/2} |x-z|^{(\nu_0+\nu_1)\alpha/2} dz
\\
&\leq N \left( t^{\nu_1/2-1} \wedge d_x^{\nu_1\alpha/2-\alpha} \right) + N 1_{R/4<t^{1/\alpha}} t^{\nu_1/2-1}
\\
&\leq N \left( t^{\nu_1/2-1} \wedge \left( d_x^{\nu_1\alpha/2-\alpha} \vee t^{\nu_1/2+1/\alpha}d_x^{-1-\alpha} \right) \right).
\end{align*}
The lemma is proved.
\end{proof}

\begin{lem} \label{aux dom}
Let $\cO$ be a bounded $C^{1,1}$ open set. Suppose that $\nu_0,\nu_1 \in \bR$ satisfy
\begin{align} \label{aux dom assu}
\nu_0+\nu_1>-\frac{2}{\alpha}.
\end{align}
Then, for any $(t,z)\in (0,\infty)\times \overline{\cO}^c$,
\begin{align} \label{23.05.07.1603}
\int_{\cO} |x-z|^{-d}\left( 1 \wedge \frac{d_x^{\alpha/2}}{\sqrt{t}} \right)^{\nu_0} d_x^{\nu_1\alpha/2} dx \leq N d_z^{-d} (t^{-\nu_0/2}+1),
\end{align}
where $N=N(d,\alpha,\cO,\nu_0,\nu_1)$.
\end{lem}

\begin{proof}
Note first that $|x-z|\geq d_z$ for $x\in \cO$ and $z\in \overline{\cO}^c$. Thus, by \eqref{21.09.20.1850} and \eqref{aux dom assu}, one can easily show that if $\nu_0\geq0$, then
\begin{align*} 
\int_{\cO} |x-z|^{-d}\left( 1 \wedge \frac{d_x^{\alpha/2}}{\sqrt{t}} \right)^{\nu_0} d_x^{\nu_1\alpha/2} dx &\leq N d_z^{-d} \int_{\cO} d_x^{(\nu_0+\nu_1)\alpha/2} t^{-\nu_0/2} dx \leq N d_z^{-d}t^{-\nu_0/2}.
\end{align*}
Therefore, it remains to prove \eqref{23.05.07.1603} when $\nu_0<0$. In this case, we have $\nu_1>-\frac{2}{\alpha}$ and thus \eqref{21.09.20.1850} yields
\begin{align*} 
&\int_{\cO} |x-z|^{-d}\left( 1 \wedge \frac{d_x^{\alpha/2}}{\sqrt{t}} \right)^{\nu_0} d_x^{\nu_1\alpha/2} dx 
\\
&\leq N d_z^{-d} \int_{\cO} d_x^{(\nu_0+\nu_1)\alpha/2} t^{-\nu_0/2} dx + N d_z^{-d} \int_{\cO} d_x^{\nu_1 \alpha/2} dx \leq N d_z^{-d}(t^{-\nu_0/2}+1).
\end{align*}
\end{proof}

\begin{lem} \label{aux lem est}
(i) Let $-1<\nu_0<\nu_1$. Then, for $x\in \cO$,
\begin{align} \label{aux ineq 1}
\int_{\overline{\cO}^c} d_z^{\nu_0} |x-z|^{-d-\nu_1} dz \leq N d_x^{\nu_0-\nu_1},
\end{align}
where $N=N(d,\cO,\nu_0,\nu_1)$.

(ii) If $\cO$ is bounded, and $-1<\nu_0<\nu_1$, then for $z\in \overline{\cO}^c$,
\begin{align} \label{aux ineq 2}
\int_{\cO} d_x^{\nu_0} |x-z|^{-d-\nu_1} dx \leq N d_z^{\nu_0-\nu_1} (1+d_z)^{-d-\nu_0},
\end{align}
where $N=N(d,\cO,\nu_0,\nu_1)$.
\end{lem}

\begin{proof}
$(i)$ We first prove \eqref{aux ineq 1}.

\textbf{1.} First, assume $\cO$ is a half space. By the change of variables $z\to x^1z$,
\begin{align} \label{22.05.31.1400}
\int_{\overline{\cO}^c} d_z^{\nu_0} |x-z|^{-d-\nu_1} dz &= \int_{\overline{\cO}^c} (-z^1)^{\nu_0} |x^1 e_1-z|^{-d-\nu_1} dz \nonumber
\\
&= N(d) (x^1)^{\nu_0-\nu_1} \int_{\overline{\cO}^c} (-z^1)^{\nu_0} |e_1-z|^{-d-\nu_1} dz.
\end{align}
Here, using the change of variables $z'\to(1-z^1)z'$,
\begin{align*}
\int_{\overline{\cO}^c} (-z^1)^{\nu_0} |e_1-z|^{-d-\nu_1} dz &= 
n(d) \int_{\overline{\cO}^c} (-z^1)^{\nu_0} (1-z^1)^{-1-\nu_1} |(1,z')|^{-d-\nu_1} dz^1 dz' 
\\
&\leq \int_{-\infty}^0 (-z^1)^{\nu_0} (1-z^1)^{-1-\nu_1} dz^1 <\infty.
\end{align*}
This and \eqref{22.05.31.1400} yield \eqref{aux ineq 1} when $\cO$ is a half space.

\textbf{2.} Let $\cO$ be a bounded $C^{1,1}$ open set. Since $\cO$ is bounded, for any $R>0$, there exists $x_1,\dots,x_n\in\partial \cO$ such that if $y\in (\cup_{i=1}^n B_{R/2}(x_i))^c$ then $d_y>R/4$.
Thus,
\begin{align*}
\int_{\overline{\cO}^c} d_z^{\nu_0} |x-z|^{-d-\nu_1} dz &\leq \sum_{i=1}^n \int_{\overline{\cO}^c \cap B_{R/2}(x_i)} d_z^{\nu_0} |x-z|^{-d-\nu_1} dz \nonumber
\\
&\quad+ \int_{\cO_{R/4}} d_z^{\nu_0} |x-z|^{-d-\nu_1} dz =: \sum_{i=1}^n I_i(x) + II(x),
\end{align*}
where $\cO_S:=\{z\in \overline{\cO}^c : d_z \geq S\}$ for $S>0$.

$(1)$ We estimate $I_i(x)$ for fixed $i\in\{1,2,\dots,n\}$.

First, if $x\in B_R(x_i) \cap \cO$, then by considering $C^{1,1}$-bijective flattening boundary map and using the case $\cO=\bR_+^d$, one can easily estimate $I_i$.

Next, assume $x\in \cO\setminus B_R(x_j)$ for some $i\neq j$. In this case, take $\tilde{x}\in \cO \cap B_R(x_i)$ such that $d_x\leq N(\cO,R) d_{\tilde{x}}$. Thus, using \eqref{eq8152005}, we have
\begin{align*} 
I_i(x) \leq N \int_{\overline{\cO}^c\cap B_{R/2}(x_i)} d_z^{\nu_0} |\tilde{x}-z|^{-d-\nu_1} dz \leq N d_{\tilde{x}}^{\nu_0-\nu_1} \leq N d_{x}^{\nu_0-\nu_1}.
\end{align*}

$(2)$  
We now estimate $II(x)$. 

Let $z\in \cO_{R/4}$ and take $y\in \partial \cO$ such that $d_z=|y-z|$. Then, by \eqref{eq8151621},
\begin{align*}
II(x) \leq N \int_{\cO_{R/4}} |x-z|^{-d+\nu_0-\nu_1} dz \leq N \int_{|x-z|>d_x} |x-z|^{-d+\nu_0-\nu_1} dz = N d_x^{\nu_0-\nu_1}.
\end{align*}

$(ii)$ Now we prove \eqref{aux ineq 2}.
Due to \eqref{aux ineq 1}, we only need to prove \eqref{aux ineq 2} when $d_z>1$. In this case,
\begin{align*}
\int_{\cO} d_x^{\nu_0} |x-z|^{-d-\nu_1} dx &\leq d_z^{-d-\nu_1} \int_{\cO} d_x^{\nu_0} dx \leq N d_z^{-d-\nu_1} 
\\
&= N d_z^{\nu_0-\nu_1} d_z^{-d-\nu_0} \leq N d_z^{\nu_0-\nu_1} (1+d_z)^{-d-\nu_0}.
\end{align*}
Here, the second inequality  easily follows from \eqref{21.09.20.1850}. The lemma is proved.

\end{proof}

\end{document}